\documentclass[a4paper, 11pt]{article}

\usepackage{mathtools}
\usepackage{amssymb}
\usepackage{graphicx}
\usepackage{etoolbox}
\usepackage[utf8]{inputenc}
\usepackage[OT1]{fontenc}
\usepackage{amsthm}

\newcommand{\dR}{\mathbb{R}}

\newcommand{\dN}{\mathbb{N}}
\newcommand{\dQ}{\mathbb{Q}}
\newcommand{\dC}{\mathbb{C}}

\newcommand{\dE}{\mathbb{E}}
\newcommand{\dP}{\mathbb{P}}

\newcommand{\cA}{\mathcal{A}}
\newcommand{\cB}{\mathcal{B}}

\newcommand{\cE}{\mathcal{E}}
\newcommand{\cF}{\mathcal{F}}
\newcommand{\cG}{\mathcal{G}}
\newcommand{\cI}{\mathcal{I}}

\newcommand{\cL}{\mathcal{L}}
\newcommand{\cM}{\mathcal{M}}
\newcommand{\cN}{\mathcal{N}}

\newcommand{\cP}{\mathcal{P}}
\newcommand{\cS}{\mathcal{S}}
\newcommand{\cT}{\mathcal{T}}

\newcommand{\cV}{\mathcal{V}}
\newcommand{\cW}{\mathcal{W}}

\newcommand{\ind}{\mathbf{1}}

\providecommand{\given}{}
\DeclarePairedDelimiterXPP{\Pb}[1]{\mathbb{P}}{\lparen}{\rparen}{}{\renewcommand{\given}{\nonscript{}\:\delimsize\vert\nonscript{}\:\mathopen{}} #1}
\DeclarePairedDelimiterXPP{\E}[1]{\mathbb{E}}[]{}{\renewcommand{\given}{\nonscript{}\:\delimsize\vert\nonscript{}\:\mathopen{}} #1}
\DeclarePairedDelimiterX{\Set}[1]\lbrace\rbrace{\renewcommand{\given}{\nonscript{}\:\delimsize\vert\nonscript{}\:\mathopen{}} #1}

\DeclareMathOperator{\diag}{diag}
\DeclareMathOperator{\tr}{tr}
\DeclareMathOperator{\Var}{Var}

\DeclareMathOperator{\img}{im}
\DeclareMathOperator{\vect}{vect}
\DeclareMathOperator{\diam}{diam}
\DeclareMathOperator{\rank}{rank}

\DeclarePairedDelimiterX{\norm}[1]\lVert\rVert{\ifblank{#1}{\: \cdot \:}{#1}}

\DeclareMathOperator{\Poi}{Poi}

\DeclareMathOperator{\dtv}{d_{TV}}

\newcommand{\bilingualcommand}[3]{%
	\newcommand{#1}[1][\ ]{%
		##1%
		\iflanguage{english}{\text{#2}}{%
			\iflanguage{french}{\text{#3}}{}%
		}%
		##1%
	}%
}

\bilingualcommand{\where}{where}{où}
\bilingualcommand{\textif}{if}{si}
\bilingualcommand{\textand}{and}{et}
\bilingualcommand{\textiff}{if and only if}{si et seulement si}
\bilingualcommand{\otherwise}{otherwise}{sinon}

\newcommand{\eps}{\varepsilon}

\newcommand{\quand}{\quad \textand{} \quad}

\newtheorem{theorem}{Theorem}
\newtheorem{lemma}{Lemma}
\newtheorem{proposition}{Proposition}

\newtheorem{corollary}{Corollary}
\newtheorem{conjecture}{Conjecture}
\newtheorem*{remark}{Remark}

\usepackage{enumitem}
\usepackage{fullpage}
\usepackage[square, numbers]{natbib}
\usepackage{microtype}

\usepackage[colorlinks, citecolor=blue, unicode]{hyperref}
\usepackage{doi}

\title{Non-backtracking spectra of weighted inhomogeneous random graphs}
\author{Ludovic Stephan \thanks{Département d’informatique de l’ENS, ENS, CNRS, PSL University, Paris, France} \thanks{Inria, Paris, France} \thanks{Sorbonne Université, Paris, France} \\ \texttt{ludovic.stephan@ens.fr} \\ Corresponding author \and Laurent Massoulié\footnotemark[1] \footnotemark[2] \thanks{Microsoft Research-Inria Joint Centre, Paris, France} \\ \texttt{laurent.massoulie@inria.fr}}

\date{}

\begin{document}

\maketitle

\begin{abstract}
  We study a model of random graphs where each edge is drawn independently (but
  not necessarily identically distributed) from the others, and then assigned a
  random weight. When the mean degree of such a graph is low, it is known that
  the spectrum of the adjacency matrix \( A \) deviates significantly from that of its
  expected value \( \dE A \).

  In contrast, we show that over a wide range of parameters the top eigenvalues
  of the non-backtracking matrix \( B \) --- a matrix whose powers count the
  non-backtracking walks between two edges --- are close to those of \( \dE A \), and all
  other eigenvalues are confined in a bulk with known radius. We also obtain a
  precise characterization of the scalar product between the eigenvectors of \( B \)
  and their deterministic counterparts derived from the model parameters.

  This result has many applications, in domains ranging from (noisy) matrix
  completion to community detection, as well as matrix perturbation theory.
  In particular, we establish as a corollary that a result known as the Baik-Ben
  Arous-Péché phase transition, previously established only for rotationally
  invariant random matrices, holds more generally for matrices \( A \) as above under a
  mild concentration hypothesis.

\end{abstract}

\vspace{1em}

\noindent\emph{Mathematics Subject Classification (2020)}: 60B20.\\
\emph{Keywords}: random graphs, community detection, non-backtraking matrix.

\newpage

\section{Introduction}

Let \( P \in \cM_n(\dR) \) be a symmetric \( n\times n \) matrix with entries in \( [0, 1] \), and \( W \) a (symmetric) weight matrix with independent random entries. We define the inhomogeneous undirected random graph \( G = (V, E) \) associated with the couple \( (P, W) \) as follows: the vertex set is simply \( V = [n] \), and each edge \( \{u, v\} \) is present in \( E \) independently with probability \( P_{uv} \), and holds weight \( W_{uv} \).

The entrywise expected value and variance of the weighted adjacency matrix of \( G \) are
\begin{equation}
  \dE{A} = P \circ \dE W \quand \Var(A) := P \circ \E*{W \circ W} - P \circ P \circ \dE W \circ \dE W,
\end{equation}
where \( \circ \) denotes the Hadamard product. When the entries of $P$ are small, the second term of $\Var(A)$ is negligible and the variance can be well approximated by the entrywise second moment; we thus define
\begin{equation}\label{eq:q_k_definition}
  Q := P \circ \dE W \quand K := P \circ \E*{W \circ W}.
\end{equation}

A natural question, arising from matrix perturbation theory, is then as follows:
\begin{center}
  \emph{What is the relationship between the eigendecomposition of \( A \) and the one of \( Q \)?}
\end{center}

Unfortunately, at least in the unweighted case, when the mean degree of \( G \) is low (\( o(\log(n)) \)), it is known that the largest eigenvalues (and associated eigenvectors) of \( A \) are determined by the large degree vertices; see~\cite{benaychgeorges_largest_2019} for a complete description of this phenomenon. To extract meaningful information on the spectrum of \( Q \), another matrix has shown better performance: the non-backtracking matrix, whose application to community detection has been studied in~\cite{krzakala_spectral_2013, bordenave_nonbacktracking_2018}.

Given a weighted graph \( G \), we define its associated non-backtracking matrix \( B \) as follows: \( B \) is a \( 2|E| \times 2|E| \) matrix indexed by the oriented edges of \( G \), whose coefficients are
\[ B_{ef} = W_f \ind \{e \rightarrow f \} = W_f\ind \{e_2 = f_1 \}\ind \{e_1 \neq f_2 \}, \]
where \( e = (e_1, e_2)  \) and \( f = (f_1, f_2) \). The above question rephrases in our setting as
\begin{center}
  \emph{What is the relationship between the eigendecomposition of \( B \) and the one of \( Q \)?}
\end{center}
and the main focus of this article is to provide an answer as precise as possible to this problem. To this end, let
\[ Q = \sum_{i=1}^r \mu_i \varphi_i\varphi_i^\top \quad \text{with} \quad |\mu_1| \geq |\mu_2| \geq \dots \geq |\mu_r| \]
be the eigendecomposition of \( Q \), and \( \rho = \rho(K) \) the largest eigenvalue (in absolute value) of \( K \). Note that by definition, $Q$ and $K$ are symmetric and therefore all eigenvalues defined above are real. \\
We shall assume that there exists some deterministic bound \( L \) (possibly depending on \( n \)) such that \( \max |W_{ij}| \leq L \). We can then state our main theorem, without detailing the needed hypotheses for now:

\begin{theorem}[informal statement]\label{th:main_summary}
  Assume the following conditions:
  \begin{enumerate}
    \item \( r = n^{o(1)} \),
    \item the graph \( G \) is sparse enough,
    \item the eigenvectors of \( Q \) are sufficiently delocalized.
  \end{enumerate}
  Let \( r_0 \) be the number of eigenvalues of \( Q \) whose absolute value is larger than both \( \sqrt{\rho} \) and \( L \):
  \begin{equation}\label{eq:r0_def}
    \mu_{k} > \sqrt{\rho} \vee L \text{ for all } k\in [r_0]  \quand \mu_{r_0+1} \leq \sqrt{\rho} \vee L
  \end{equation}
  Then, for \( i\leq r_0 \), the \( i \)-th largest eigenvalue of \( B \) is asymptotically (as \( n \) goes to infinity) equal to \( \mu_i \), and all the other eigenvalues of \( B \) are constrained in a circle of center \( 0 \) and radius \( \max(\sqrt{\rho}, L) \). Further, if \( i \leq r_0 \) is such that \( \mu_i \) is a sufficiently isolated eigenvalue of \( Q \), then the eigenvector associated with the \( i \)-th eigenvalue of \( B \) is correlated to a lifted version of \( \varphi_i \).
\end{theorem}

Next section consists in the detailed statement of this theorem (with precise hypotheses and bounds given).

\section{Detailed setting and results}
\subsection{Notations}

\paragraph{General notations:}

Throughout this paper, we use the following notations:
\begin{itemize}
  \item for integer \( n \), \( [n] \) denotes the set \( \{ 1, \dots, n\} \).
  \item for \( x\in \dR^n \), we shall denote by \( x_i \) or \( x(i) \) the \( i \)-th coordinate of \( x \), whichever is most convenient. \( \norm{x} \) is the 2-norm of \( x \), and \( \norm{x}_\infty \) the infinity norm of \( x \).
  \item the operator norm of a matrix \( M \) is noted \( \norm{M} \); it is the maximal singular value of \( M \). Its Frobenius norm is noted \( \norm{M}_F \) and its infinity norm \( \norm{M}_\infty = \sup_{i, j} |M_{ij}| \).
  \item \( \ind \) denotes the all-one vector, and \( \ind \{\cdot \} \) is the indicator function of an event.
  \item the group of permutations on \( r \) elements is noted \( \mathfrak S_r \).
  \item the max (resp.\ min) of two numbers \( a, b \) is noted \( a \vee b \) (resp. \( a \wedge b \)).
  \item the letter \( c \) denotes any absolute constant, whose value should be assumed to be the maximum of any such constant encountered so far. To improve the readability  of our computations, we use numbered constants \( c_i \) during proofs.
\end{itemize}

\paragraph{Graph theoretic notations:}

For a graph \( g = (V, E) \), let \( \vec E \) be the set of oriented edges in \( E \), and
\[ \vec E(V) = \Set*{(u, v) \given u \neq v \in V} \]
be the set of all directed edges of the \emph{complete graph} on \( V \). If \( t \) is an integer, \( g = (V, E) \) is a graph and \( x \in V \), then the \emph{ball} \( {(g, x)}_t \) is the subgraph induced by all edges at distance at most \( t \) from \( x \), and \( \partial {(g, x)}_t \) is the boundary of the ball, i.e.\ the set of vertices at distance exactly \( t \) from \( x \). Finally, the set of all non-backtracking paths of length \( t \) starting with \( x \) will be denoted \( \cP_g (x, t) \).

\paragraph{Non-backtracking matrix:} Since we are interested in the spectrum of the non-backtracking matrix \( B \), we need to be able to translate ``vertex'' quantities such as the vectors \( \varphi_i \) into ``edge'' quantities. Recall that \( V = [n] \), and identify \( \vec E \) with the set \( [2m] \); we define the \( 2m \times n \) \emph{start} and \emph{terminal} matrices \( S \) and \( T \) as
\begin{equation}\label{eq:start-end}
  \forall e\in \vec E, i\in [n],\quad  S_{ei} = \ind \{e_1 = i \} \quand   T_{ei} = \ind \{e_2 = i \}.
\end{equation}
For a vector \( \phi\in \dR^n \), this implies that \( [T\phi](e) = \phi(e_2) \) for every edge \( e\in \vec E \). We then define the ``lifted'' eigenvectors \( \chi_i = T\varphi_i \) for \( i\in [r] \).

\medskip

We also define the \emph{reverse} operator \( J \) such that \( Je = \bar e := (e_2, e_1) \), and the diagonal matrix \( D_W \) such that \( D_W(e, e) = W_e \); from the definition of \( B \) and symmetry of \( W \) it is straightforward to see that \( J D_W \) = \( D_W J \) and for all \( t\geq 0 \)
\begin{equation}\label{eq:parity-time}
  J D_W B^t = {(B^*)}^t D_W J,
\end{equation}
which is known in mathematical physics as parity-time invariance. For any vector \( x \in \dR^{\vec E} \), we denote the vector \( Jx \) by \( \check x \).

Building upon the sketch in the introduction, we now expand on the model definition. Recall that the \emph{expectation} and \emph{variance} matrices were defined as
\[ Q = P \circ \dE W \quand K = P \circ \E*{W \circ W}. \]

\subsection{Defining the convergence parameters} \label{subsec:conv_parameters}
In full generality, with no assumptions on \( P \) and \( W \), we do not expect meaningful results to hold; however, we are still able to provide interesting properties on a large class of matrices. We define in the following the parameters that will govern the convergence behavior :
\begin{enumerate}
  \item the \emph{rank}
        \[ r = \max\left(\rank(Q), \sqrt{\rank(K)}\right); \]
        note that in most practical applications (such as the unweighted case), we shall have $r = \rank(Q)$, but we also treat cases where $r \gg \rank(Q)$.

  \item the sparsity parameter
        \[ d = n \max_{i, j \in [n]} P_{ij}; \]

  \item the eigenvector delocalization parameter
        \[ b = \sqrt{n} \max _{i \leq \rank(Q)} \norm{\varphi_i}_\infty; \]

  \item the signal-to-noise ratio
        \[ \tau = \max_{\mu_i^2 > \mu_1} \frac{\mu_1}{\mu_i^2}; \]

  \item and finally the almost sure probability bound
        \[ \norm{W}_\infty = L; \]
        our results hold trivially whenever $L = +\infty$ so we shall restrict ourselves to the case
        where $L$ is finite, and the $W_{ij}$ are almost surely bounded. While Theorem \ref{th:main} below requires an almost sure bound, techniques for dealing with high probability bounds are discussed in Theorem \ref{th:gaussian}.

\end{enumerate}

The average degree of a vertex \( i \) will be noted by
\[ d_i = \sum_{j\in[n]} P_{ij} \leq d, \]
which corresponds to the entries of the vector $P \ind$. To ensure that $G$ is connected enough for spectral properties to hold, we make the (common) assumption that $d_i \geq 1$ for all $i \in [n]$. The entries of $K \ind$ can be viewed as an extension of the average degrees in the weighted case (see \cite{alt_extremal_2020} or \cite{benaychgeorges_spectral_2020} for examples), and for the same reason as above we require that $K \ind$ is bounded away from zero by a constant.

\subsection{Main theorem}

In the following, \( G = G(P, W) \) is the random graph defined in the introduction, \( B \) is the non-backtracking matrix associated with \( G \), and \( |\lambda_1| \geq \cdots \geq |\lambda_{2m}| \) are its eigenvalues.

\bigskip

In its most general form, our main result is as follows:

\begin{theorem}\label{th:main}
  Let $n \geq 0$ and $(P, W)$ be a couple of $n \times n$ matrices defining a random graph $G$. Define \( \rho = \rho(K) \), \( r_0 \) as in \eqref{eq:r0_def}, $r, b, d, \tau, L$ as in Subsection \ref{subsec:conv_parameters}, and \( \tilde L = L/\mu_1 \).

  Let
  \[ \ell = \frac{1 - \epsilon}{16}\frac{\log(n)}{\log(d)}, \]
  for arbitrary $\eps > 0$. There exist numbers \( n_0 \) and \( C_0 \), all depending on \( n \) and the convergence parameters, such that the following holds:
  \begin{enumerate}
    \item \( C_0 \) is smaller than
          \[ c{\left( \frac{r b d \tilde L \log(n)}{1 - \tau} \right)}^{25},\]
          and \( n_0 \) is smaller than
          \[ \exp\left( c \frac{\max\Set*{\log(r), \log(b), \log{(d)}^2, \log{(\tilde L)}, \log(\log(n))}}{\log(\tau^{-1})} \right). \]
    \item If \( n \geq n_0 \), define
          \begin{equation}\label{eq:sigma_def}
            \sigma := C_0 \mu_1 \tau^{\ell/2}.
          \end{equation}
          Then the following holds with probability at least $1 - c/\log(n)$, there exists a permutation $s$ of $[r_0]$ such that
          \begin{equation}
            \max_{i\in [r_0]} \left| \lambda_i - \mu_{s(i)} \right| \leq \sigma,
          \end{equation}
          and all the remaining eigenvalues of \( B \) are less than \( C_0^{1/\ell} \left(\sqrt{\rho}\vee L\right) \).
    \item For any \( i \in [r_0] \), if
          \begin{equation} \label{eq:delta_i_def}
            \delta_i := \min_{j\neq s(i)} |\mu_{s(i)} - \mu_j| \geq 2 \sigma,
          \end{equation}
          then there exists a normed eigenvector \( \xi \) associated with \( \lambda_i \) such that
          \[ \langle \xi, \xi_i \rangle \geq \sqrt{1 - r d^2 \tilde L^2 \frac{\rho}{\mu_i^2}} + O\left(\frac{\sigma}{\delta_i - \sigma}\right) \quad \text{where} \quad \xi_i = \frac{T\varphi_i}{\norm{T\varphi_i}}. \]
  \end{enumerate}
\end{theorem}

In order to get an applicable and useful result, we need \(  n \geq n_0 \) when \( n \) is sufficiently large, and \( C_0^{\frac 1 \ell} \) goes to 1 as \( n \) goes to infinity. Both conditions are verified in particular when
\[ 1 - \tau = \Omega(1), \quad r, b = n^{o(1)}\quand \log{(d)}^2 = o(\log(n)). \]
By definition of $\tilde L$, whenever $\tilde L > 1$ we have $\mu_1 < L$ and thus $r_0 = 0$. We can therefore safely assume $\tilde L \leq 1$ in applications and not focus on any bound for $L$.

The proof of this theorem follows the same method as in many spectral proofs, from~\cite{massoulie_community_2014} to more recent papers such as~\cite{bordenave_detection_2020}. It consists of the following:
\begin{itemize}
  \item show that the neighbourhood of any vertex \( v \) is close to a suitably defined random tree,
  \item study a family of graph functionals that give rise to approximate eigenvectors of the random tree,
  \item use a concentration argument to transpose those tree eigenvectors to pseudo-eigenvectors of the non-backtracking matrix,
  \item bound the remaining eigenvalues using a variant of the trace method in~\cite{furedi_eigenvalues_1981},
  \item conclude by a matrix perturbation argument.
\end{itemize}
A large portion of the remainder of this paper is dedicated to implementing this method; however, we first provide several applications of our result to the fields of random matrix theory and random graph theory.

\section{Applications}

\subsection{Phase transition in random graphs}

Matrix perturbation theory focuses on finding the eigenvalues and eigenvectors of matrices of the form \( X + \Delta \), where \( X \) is a known matrix and \( \Delta \) is a perturbation assumed ``small'' in a sense. Celebrated results in this field include the Bauer-Fike theorem~\cite{bauer_norms_1960} for asymmetric matrices, and the Weyl~\cite{weyl_asymptotische_1912} and Davis-Kahan~\cite{yu_useful_2015} theorems for symmetric ones; incidentally the present paper makes use of those results in its proofs. Finding sharp general theorems without additional assumptions is known to be hard, since the eigenvalues and eigenvectors depend on the interactions between the eigenspaces of \( X \) and \( \Delta \).

In the last two decades, growing attention has been paid to problems of the following form: finding the eigenvectors of \( X_n + P_n \) (or, in its multiplicative form, \( X_n(I_n + P_n) \)), where \( P_n \) is an \( n \times n \) matrix with low rank \( r \ll n \) (usually fixed) and known eigenvalues, and \( X_n \) is a random matrix with known distribution. Examples of this setting are the spiked covariance model~\cite{baik_phase_2005,johnstone_pca_2018} and additive perturbations of Wigner matrices~\cite{peche_largest_2006,feral_largest_2007,capitaine_largest_2009}. A more systematic study has been performed in~\cite{benaychgeorges_singular_2012,benaychgeorges_spectral_2020} on orthogonally invariant random matrices.

A staple  of those results is the existence of a so-called \emph{BBP phase transition} (named after Baik-Ben Arous-Péché, from the seminal article~\cite{baik_phase_2005}): in the limit \( n \to \infty \), each eigenvalue of \( P_n \) that is above a certain threshold gets reflected (albeit perturbed) in the spectrum of \( X_n + P_n \), with the associated eigenvector correlated to the one of \( P_n \).

\paragraph{Phase transition for the adjacency matrix}
The adjacency matrix \( A \) of our random graph \( G \) can be viewed as a perturbation model by writing
\[ A = \dE A + (A - \dE A) = Q - \diag(Q) + (A - \dE A). \]
The term \( \diag(Q) \) being negligible with respect to the others, we can see \( A \) as the sum of a deterministic low-rank matrix and a random noise matrix with i.i.d centered entries. Further, the entrywise variance of \( A \) is equal (up to a negligible term) to \( K \), so the parameter \( \rho \) can be seen as an equivalent to the variance in the Wigner model.
We thus expect, whenever \( \sqrt{\rho} \gg L \) (so that \( \sqrt{\rho} \) is the actual threshold in Theorem~\ref{th:main}), to find a phase transition akin to the one in~\cite{benaych-georges_eigenvalues_2011}; and indeed the following theorem holds:
\begin{theorem}\label{th:bbp_transition}
    Let \( (P, W) \) be a matrix couple of size $n \times n$ and $r, b, d, \tau, L$ as above. Assume further that:
    \begin{enumerate}
        \item the Perron-Frobenius eigenvector of \( K \) is \( \ind \); that is \( K\ind = \rho \ind \),
        \item the above eigenvector equation concentrates, i.e.\ with high probability there exists \( \eps \leq 1/2 \) such that for all \( i\in [n] \),
              \begin{equation} 
                  \left| \sum_{j \sim i} W_{ij}^2 - \rho \right| \leq \eps\rho
              \end{equation}
    \end{enumerate}
    Then, if \( i \in [r_0] \) is such that \( \mu_i^2 \geq 2L^2 \), there exists an eigenvalue \( \nu_i \) of \( A \) that verifies
    \begin{equation}\label{eq:bbp_eigenvalue_bound}
        \nu_i = \mu_i + \frac{\rho}{\mu_i} + \frac{\rho}{\mu_i} \cdot O\left(\frac{L}{\mu_i} + \frac{L^2}{\mu_i^2} + \eps \right). 
    \end{equation}
    Further, if the mean degree \( d_j \) for all \( j \) is equal to \( d_0 > 1 \), and \( i \) is such that \( \delta_i \geq 2\sigma \) (with $\sigma$ and \( \delta_i \) defined in \eqref{eq:sigma_def} and \eqref{eq:delta_i_def}), then there exists a normed eigenvector \( \zeta \) of \( A \) with corresponfing eigenvalue \( \nu_i \) such that
    \begin{equation}\label{eq:bbp_eigenvector_scalar}
        \langle \zeta, \varphi_i \rangle = \sqrt{1 - \frac{\rho}{\mu_i^2}} + O\left[\frac 1{\delta_i - \sigma}\left(\frac{L\rho}{\mu_i^2} + \frac{L^2 \rho}{\mu_i^3} + \eps \frac{\rho}{\mu_i} \right)\right].
    \end{equation}
\end{theorem}
Whenever \( \rho \gg L^2\), and \( \eps \) goes to zero as \( n \to \infty \), then the condition $\mu_i^2 \geq 2L^2$ is always verified and the \( O(\cdot) \) term in~\eqref{eq:bbp_eigenvalue_bound} vanishes, and the obtained expansion is therefore asymptotically correct. The presence of \( \delta_i \) renders a similar result on the scalar product harder to obtain; however, assuming \( \delta_i = \Theta(\sqrt{\rho}) \) (that is, the eigenvalues of \( Q \) are somewhat regularly spaced) implies similarly that the \( O(\cdot) \) term in~\eqref{eq:bbp_eigenvector_scalar} vanishes.

The obtained expression for \( \nu_i \), as well as the scalar product expansion, are identical to the ones in~\cite{benaych-georges_eigenvalues_2011}, for low-rank additive perturbations of Gaussian Wigner matrices. Our result is thus a direct extension of~\cite{benaych-georges_eigenvalues_2011}, for a larger class of matrices upon a sparsity and concentration condition. Such an extension isn't unexpected, in view of results concerning the universality of the semicircle law for Bernoulli random matrices, such as~\cite{erdos_local_2013}.

An especially interesting particular case of Theorem~\ref{th:bbp_transition} is the unweighted random graph setting, where \( W_{ij} = 1 \) for all \( i, j \). In this case, we have \( K = P \) so the eigenvector equation \( K\ind = \rho\ind \) is equivalent to all the average degrees being equal, i.e. \( d_i = d_0 = \rho \) for \( i \in [n] \). It is a well known fact (see for example~\cite{feige_spectral_2005}) that for unweighted random graphs the degree concentration property holds with \( \eps = 2\sqrt{\log(n)/d_0} \). A slight modification of the proof of Theorem \ref{th:bbp_transition} further removes several error terms, and the following corollary ensues:
\begin{corollary}\label{cor:bbp_unweighted}
    Let \( P \) be a $n \times n$ matrix and $r, b, d, \tau$ as above, with $W = \ind^*\ind$. Assume further that for all \( i \in [n] \),
    \[ \sum_{j \in [n]} P_{ij} = d_0 > 16 \log(n).  \]
    Then for all \( i \in [r_0] \), there exists an eigenvalue \( \nu_i \) of \( A \) that verifies
    \[ \nu_i = \mu_i + \frac{d_0}{\mu_i} + O\left(\sqrt{\frac{\log(n)}{d_0}} \frac{d_0}{\mu_i} \right), \]
    and if \( i \) is such that \( \delta_i > 2\sigma \), there exists a normed eigenvector of \( A \) such that
    \[ \langle \zeta, \varphi_i \rangle = \sqrt{1 - \frac{d_0}{\mu_i^2}} + O\left(\frac 1{\delta_i - \sigma}\sqrt{\frac{\log(n)}{d_0}} \frac{d_0}{\mu_i}\right). \]
    In particular we have
    \[ \nu_1 = d_0 + 1 + O\left( \sqrt{\frac{\log(n)}{d_0}} \right) \]
\end{corollary}
This is an improvement on the results of \cite{benaychgeorges_spectral_2020}, which only give $\nu_i = \mu_i + O(\sqrt{d_0})$. The condition $d_0 > 16 \log(n)$ ensures that the degrees of $G$ concentrate. Since our result is really only meaningful whenever $d_0 \gg \log(n)$, so that the error term is negligible before $d_0/\mu_i$, we do not perform the same detailed analysis as in \cite{alt_extremal_2020}. However, a more precise phase transition around $d_0 \asymp \log(n)$ is not excluded.

\bigskip

Theorem~\ref{th:bbp_transition} is derived from Theorem~\ref{th:main} through an adaptation of the Ihara-Bass formula~\cite{bass_iharaselberg_1992}, obtained by expanding arguments from~\cite{benaychgeorges_largest_2019, watanabe_graph_2009}:
\begin{proposition}\label{prop:ihara_bass}
    Let \( x \) be an eigenvector of the matrix \( B \) with associated eigenvalue \( \lambda \), such that \( \lambda^2 \neq W_{ij}^2 \) for every \( i, j \). Define the weighted adjacency matrix $\tilde A(\lambda)$ and the diagonal degree matrix $\tilde D(\lambda)$ by
    \[ \tilde A{(\lambda)}_{ij} = \ind \{i \sim j \} \frac{\lambda W_{ij}}{\lambda^2 - W_{ij}^2} \quand \tilde D{(\lambda)}_{ii} = \sum_{j \sim i} \frac{W_{ij}^2}{\lambda^2 - W_{ij}^2} \]
    Then the vector \( y = S^* D_W x \), where $S$ is defined as in \eqref{eq:start-end}, is a null vector of the laplacian matrix
    \[ \Delta(\lambda) = I - \tilde A(\lambda) + \tilde D(\lambda). \]
\end{proposition}
The details and computations are left to the appendix.

\subsection{Community detection in random networks}

Community detection is a clustering problem that aims to identify large subgroups (or communities) with similar characteristics inside a large population, with the only data available being the pairwise interactions between individuals. Starting from its introductory paper~\cite{holland_stochastic_1983}, the stochastic block model has been a popular generative model for algorithm design; it consists of a random graph \( G \) where vertices are partitioned randomly in communities, and edges are present independently with probability depending only on the community membership of their endpoints. Popular algorithms for recovering communities include semi-definite programming methods~\cite{montanari_semidefinite_2016}, belief propagation~\cite{abbe_proof_2018}, and spectral methods~\cite{lei_consistency_2015,massoulie_community_2014}; a comprehensive review of algorithms and results can be found in~\cite{abbe_community_2018}.

\paragraph{Unlabeled stochastic block model}

In a general form, we can define the stochastic block model \( \mathrm{SBM}(n, r, \theta, M) \), where \( \theta \in {[r]}^n \) and \( M \in [0, 1]^{r \times r} \) as follows:
\begin{itemize}
    \item the vertex set is \( V = [n] \),
    \item each vertex \( i \in [n] \) has a community label \( \theta_i \) in \( [r] \),
    \item for any pair of vertices \( (i, j) \), an edge is present between \( i \) and \( j \) independently from the others with probability \( M_{\theta_i \theta_j} \).
\end{itemize}
It is common to assume \( M = \frac{\alpha}n M_0 \), where \( M_0 \) does not depend on \( n \) and \( \alpha \) is a scaling parameter. It is easy to see that up to diagonal terms, the expected adjacency matrix has the form
\[ P = \Theta M \Theta^*, \]
where \( \Theta \) is a \( n \times r \) matrix such that \( \Theta_{ij} = 1 \) if \( \theta_i = j \), and \( 0 \) otherwise. We shall assume that for any \( k\in [r] \),
\begin{equation}\label{eq:linear_communities}
    \frac{\#\Set{i\in [n] \given \theta_i = k}}n = \pi_k > 0,
\end{equation}
where \( \pi \) is a deterministic probability vector. Let \( \mu_1 \geq \cdots \geq |\mu_r| \) the eigenvalues of \( \diag(\pi) M_0 \), with \( \alpha \) chosen such that \( |\mu_1| = 1 \), and \( \phi_1, \dots, \phi_r \) the associated eigenvectors. Then the non-zero eigenvalues of \( P \) are easily found to be the \( \alpha \mu_i \), with associated eigenvectors \( \Theta \phi_i \).

A common assumption is that each vertex type has the same average degree, i.e.
\[ P\ind = \alpha \ind, \]
otherwise a simple clustering based on vertex degree correlates with the underlying communities. Making this additional assumption, the following theorem holds:

\begin{theorem}\label{th:unlabeled_sbm}
    Assume that \( r \) is constant, and $\alpha = n^{o(1)}$. Let $r_0$ be defined as follows :
    \begin{itemize}
        \item if \( \alpha \geq 1 \) is constant, \( r_0 \) is the only integer in $[r]$ such that
              \[ \alpha \mu_k^2 > 1 \quad \text{for } i\in {r_0},\quad \alpha\mu_{r_0+1}^2 \leq 1. \]
        \item if \( \alpha = \omega(1) \), \( r_0 = r \).
    \end{itemize}
    Then, for any \( n \) larger than an absolute constant and all \( i\in [r_0] \) one has
    \[ |\lambda_i - \mu_i| \leq c {(\alpha \log(n))}^{a} {(\alpha \mu_{r_0})}^{-\kappa\log_\alpha(n)} := \sigma \]
    for some positive constants \( c, a ,\kappa \), and all other eigenvalues of $B$ are confined in a circle with radius \( (1+o(1))\sqrt{\alpha} \). Further, if \( \mu_i \) is an isolated eigenvalue of \( \diag(\pi)M_0 \), then there exists an eigenvector \( \xi \) of the non-backtracking matrix \( B \) associated with \( \lambda_i \) such that
    \[ \langle \xi, \xi_i \rangle \geq \sqrt{1 - \frac{1}{\alpha\mu_i^2}} + O(\sigma') \quad \text{where} \quad \xi_i = \frac{T\Theta \phi_i}{\norm{T\Theta\phi_i}}. \]
\end{theorem}

This theorem is essentially a corollary of Theorem~\ref{th:main}, with some simplifications due to \( Q = K = P \) and \( P\ind = \alpha \ind \); the error bound $\sigma$ is the same as in the main theorem. It is a direct generalization of Theorem 4 in~\cite{bordenave_nonbacktracking_2018}, for a diverging degree sequence; further, the property \( \langle \xi, \xi_i \rangle = 1 - o(1) \) as soon as \( \alpha \gg 1 \) suggests that a clustering algorithm such as $k$-means performed on the eigenvectors of \( B \) recovers all but a vanishing fraction of the community memberships in this regime, which would provide an alternative to the Sphere-comparison algorithm presented in~\cite{abbe_community_2015}.

\begin{conjecture}
    In the SBM defined as above, as soon as \( \alpha = \omega(1) \), running an approximate \( k \)-means algorithm on the top \( r \) eigenvectors of \( B \) allows to recover the community memberships of every vertex but a vanishing fraction as \( n \to \infty \).
\end{conjecture}
Proving this conjecture would require a more careful eigenspace analysis for eigenvalues with multiplicity more than one, such as the one performed in~\cite{stephan_robustness_2019}, as well as an error bound on the clustering step similar to the one in~\cite{lei_consistency_2015}.

\begin{remark}
    When the expected degrees of each vertex type is not the same, an analogous version of Theorem \ref{th:unlabeled_sbm} holds. The main difference in this case is that the scalar product \( \langle \xi, \xi_i \rangle \) has a less elegant asymptotic expansion.
    
    Since the lead eigenvector of $P$ is now non-constant, the condition for reconstruction is simply $\alpha > 1$ (or $r_0 \geq 1$). In particular, this is true as long as the average degree of the graph is above one; indeed by the Courant-Fisher principle 
    \[ \alpha > \frac{\ind^\top P \ind}n. \]
\end{remark}

\paragraph{Labeled block models} In real-world networks, pairwise interactions often carry more information than just a binary one. A popular variant of the stochastic block model is thus a model with added edge labels, as follows: let \( \cL \) be a label space, and consider a SBM drawn under the model described above. We assign to an edge \( (i, j) \) a label \( L_{ij} \in \cL \), drawn independently from a distribution \( \dP_{\theta_i \theta_j} \). Such classes of models have been investigated in full generality in~\cite{heimlicher_community_2012, lelarge_reconstruction_2015}, and a variant with the underlying graph being an Erd\H{o}s-Rényi model in~\cite{saade_spectral_2015}.

We shall focus here on the symmetric two-community SBM, with 
\begin{equation}\label{eq:labeled_sbm_definition}
    \pi = \left( \frac12, \frac12 \right), \quad M = \begin{pmatrix} a & b \\ b & a \end{pmatrix}, \quad \dP_{11} = \dP_{22} = \dP \quand \dP_{12} = \dP_{21} = \dQ,
\end{equation}
and assume that both measures are absolutely continuous with respect to another measure \( m \) (note that we can take \( m = \dP + \dQ \)), with Radon-Nikodym derivatives \( f \) and \( g \). Let \( w: \cL \to \dR \) a bounded weight function, such that \( w(\ell) \leq L \) for any \( \ell \in \cL \); and define the weight matrix \( W_{ij} = w(L_{ij}) \) and the associated weighted non-backtracking matrix \( B \). Then, an application of Theorem~\ref{th:main} yields the following result:
\begin{theorem}\label{th:labeled_sbm}
    Define the parameter \( \tau \) by
    \[ \tau = 2\frac{(a \dE_\dP[w^2] + b \dE_\dQ[w^2]) \vee L}{{(a \dE_\dP[w] - b \dE_\dQ[w])}^2} \]
    Then, whenever \( \tau < 1 \), let \( \xi \) be a normed eigenvector corresponding to the second eigenvalue of \( B \). There exists a parameter
    \[ \sigma \leq {(a \log(n))}^{25} \tau^{\kappa \log_a(n)} \]
    for some constant \( \kappa \) such that
    \[ \langle \xi, \xi_0 \rangle = \sqrt{1 - \tau} + O(\sigma) \quad \text{where} \quad \xi_0 = \frac{\Theta \dbinom{1}{-1}}{\sqrt{n}}\]
\end{theorem}
Whenever this result holds, a proof identical to the one in~\cite{massoulie_community_2014} implies that recovering a positive fraction of the community memberships is possible.

In order to maximize the region in which reconstruction is possible, we need to choose the weights $w(\ell)$ such that $\tau$ is minimized. This optimization step is performed in the appendix, and leads to the following:
\begin{proposition}\label{prop:labeled_symmetric_sbm}
    Define the weight function \( w \) and signal-to-noise ratio \( \beta \) as
    \begin{equation}\label{eq:weight_def}
        w(\ell) = \frac{af(\ell) - bg(\ell)}{af(\ell) + bg(\ell)} \quand \beta = \frac12\int\frac{{(af - bg)}^2}{af + bg} dm,
    \end{equation}
    where \( a, f, b, g \) and \( m \) are defined in Equation~\eqref{eq:labeled_sbm_definition} and below. Then, whenever \( \beta > 1 \), a spectral algorithm based on the matrix \( B \) is able to recover a positive fraction of the community memberships when \( n \to \infty \).
\end{proposition}
This settles a conjecture of~\cite{heimlicher_community_2012}, generalizing the setting from finite to arbitrary label space. Whenever we allow for a higher number of communities, as well as arbitrary choices for the connectivity matrix \( Q \) and distributions \( \dP_{ij} \), the problem proves to be harder; an analog to Theorem~\ref{th:labeled_sbm} does hold, but the optimization problem required to minimize the ratio \( \tau \) looks to be untractable. In the symmetric SBM case, where 
\[ \pi = \frac \ind k, \quad M = a\ind \{i = j \} + b\ind \{i \neq j \} \quand \dP_{ij} = \dP \ind \{i = j \} + \dQ \ind \{i \neq j \}, \]
we make the following conjecture:
\begin{conjecture}
    In the labeled symmetric SBM, partial reconstruction is possible as soon as \( \beta > 1 \), where
    \[ \beta = \frac1k \int\frac{{(af - bg)}^2}{af + (k-1)bg} dm, \]
    and a spectral algorithm based on the non-backtracking matrix with weight function
    \[ w(\ell) = \frac{af(\ell) - bg(\ell)}{af(\ell) + (k-1)bg(\ell)} \]
    recovers a positive fraction of the community memberships in polynomial time.
\end{conjecture}
As with Theorem~\ref{th:unlabeled_sbm}, whenever the mean degree \( \alpha \) of the graph grows to infinity, we have \(  \langle \xi, \xi_0 \rangle = 1 - o(1) \), which brings us our second conjecture:
\begin{conjecture}
    If we have \( a = \alpha a_0 \), \( b = \alpha b_0 \) with \( \alpha = \omega(1) \), $a_0, b_0$ fixed, then as \( n \to \infty \) a clustering algorithm based on the second eigenvector of the weighted non-backtracking matrix \( B \) with the weight function defined in \eqref{eq:weight_def} recovers all but a vanishing fraction of the community memberships.
\end{conjecture}
As a final remark, note that the optimal weight function assumes perfect knowledge of all model parameters, especially the exact label distribution for each community pair. However, in some cases, this weight function is a rescaling of a more agnostic one; as an example, in the censored block model~\cite{abbe_decoding_2014} we find that \( w(\ell) = c\ell \) (with \( \ell = \pm 1 \)), and thus the spectral algorithm mentioned here is the same as in~\cite{saade_spectral_2015}.

\subsection{Extension to gaussian weights}

In the form presented in Theorem \ref{th:main}, our result is only meaningful with almost surely bounded random variables (i.e. with $L < \infty$). With a more careful analysis of the error bounds, this can be extended to
\begin{equation} \label{eq:bounded_moments}
    L = \sup_{i, j \in [n]} \sup_k \E{W_{ij}^k}^{1/k};
\end{equation}
however, we determined the class of distributions satisfying \eqref{eq:bounded_moments} was not different enough from the bounded case to warrant increasing the complexity of the proof.

To the contrary, the setting where the $W_{ij}$ are gaussian random variables is of independent interest; it can be seen as a special case of noisy matrix completion as described in \cite{candes_matrix_2010, keshavan_matrix_2009}. In this case, the moment condition of \eqref{eq:bounded_moments} is far from satisfied, and at least at first glance our proof cannot be adapted readily. Still, we show the following: 
\begin{theorem}\label{th:gaussian}
    Assume that the $W_{ij} \sim \cN(m_{ij}, s_{ij}^2)$ are independent Gaussian random variables, and let
    \[ m = \sum_{i, j} m_{ij} \quand s = \sup_{i, j} s_{ij}. \]
    
    Then the conclusions of Theorem \ref{th:main} apply with
    \[ L = m + 2s\sqrt{\log(n)} \]
\end{theorem}

The loss of a $\sqrt{\log(n)}$ factor comes from the use of a concentration bound for the $W_{ij}$; details can be found in the appendix.

To the best of our knowledge, there isn't much literature to compare with on the topic of eigenvalue reconstruction for noisy matrix completion, the works cited above being focused on reconstructing the whole matrix. However, results on gaussian matrix perturbation such as \cite{benaych-georges_eigenvalues_2011} seem to indicate that the $\sqrt{\log(n)}$ factor is superfluous and can be improved upon with other methods.

\section{A Bauer-Fike type bound for almost orthogonal diagonalization}\label{sec:perturbation}

One important tool in tying together the local analysis of \( G \) is a matrix perturbation theorem, derived from the Bauer-Fike theorem. It mostly consists in a simplification and adaptation of Theorem 8.2 in~\cite{bordenave_detection_2020}, tailored to our needs. We begin by recalling the original Bauer-Fike Theorem:
\begin{theorem}[Bauer-Fike Theorem~\cite{bauer_norms_1960}]\label{th:bauer_fike}
  Let \( D \) be a diagonalizable matrix, such that \( D = V^{-1}\Lambda V \) for some invertible matrix \( V \) and \( \Lambda = \diag(\lambda_1, \dots, \lambda_n) \). Let \( E \) be any matrix of size \( n \times n \). Then, any eigenvalue \( \mu \) of \( D + E \) satisfies
  \begin{equation}\label{eq:bauer_fike}
    |\mu - \lambda_i| \leq \norm{E}\,\kappa(V),
  \end{equation}
  for some \( i \in [n] \), where \( \kappa(V) = \norm{V}\norm{V^{-1}} \) is the condition number of \( V \).

  Let \( R \) be the RHS of~\eqref{eq:bauer_fike}, and \( C_i := \cB(\lambda_i, R) \) the ball centered at \( \lambda_i \) with radius \( R \) (in \( \dC \)). Let \( \cI \subseteq [n] \) be a set of indices such that
  \[ \left(\bigcup_{i \in \cI} C_i\right) \cap \left(\bigcup_{i \notin \cI} C_i\right) = \emptyset. \]
  Then the number of eigenvalues of \( D + E \) in \( \bigcup_{i \in \cI} C_i \) is exactly \( |\cI| \).
\end{theorem}

\subsection{A custom perturbation lemma for almost diagonalizable matrices}

Building on this theorem, we now expose this section's first result. Let \( U = (u_1, \dots, u_r) \) and \( V = (v_1, \dots, v_r) \) be \( n \times r \) matrices; our nearly diagonalizable matrix shall be \( S = U \Sigma V^* \) with \( \Sigma = \diag(\theta_1, \dots, \theta_r) \). We shall assume that the \( \theta_i \) are in decreasing order of modulus:
\[ |\theta_r| \leq |\theta_{r-1}| \leq \cdots \leq |\theta_1| = 1. \]

Now, let \( A \) be a \( n\times n \) matrix, not necessarily diagonalizable. The assumptions needed for our results are as follows:
\begin{enumerate}
  \item For some small constant \( \eps > 0 \),
        \[ \norm{A - S} \leq \eps. \]
  \item The matrices \( U \) and \( V \) are well-conditioned: both \( U^*U \) and \( V^*V \) are nonsingular, and there exist two constants \( \alpha, \beta > 1 \) such that
        \begin{align*}
          \norm{U^*U}          & \leq \alpha, & \norm{V^*V}          & \leq \alpha, \\
          \norm{{(U^*U)}^{-1}} & \leq \beta,  & \norm{{(V^*V)}^{-1}} & \leq \beta.
        \end{align*}
  \item There exists another constant \( 0 < \delta < 1 \) such that
        \[ \norm{U^*V - I_r}_\infty \leq \delta. \]
  \item The \( \theta_i \) are well-separated from \( 0 \), in the sense that
        \begin{equation}\label{eq:perturbation_eigen_separation}
          |\theta_r| > 2\sigma,
        \end{equation}
        where an exact expression for $\sigma$ will be given over the course of the proof.
\end{enumerate}

Then the following result, whose statement and proof (regarding the eigenvalue perturbation) are adapted from~\cite{bordenave_detection_2020}, holds:
\begin{theorem}\label{th:bauer_personal}
  Let $A$ be a matrix satisfying assumptions (i)-(iv) above, and let \( |\lambda_1| \geq |\lambda_2| \geq \cdots \geq |\lambda_r| \) be the \( r \) eigenvalues of \( A \) with largest modulus. There exists a permutation $\pi$ such that for all \( i \in [r] \)
  \[ |\lambda_{\pi(i)} - \theta_i| \leq r \times \sigma, \]
  and the other \( n - r \) eigenvalues of \( A \) all have modulus at most \( \sigma \). Additionally, if \( i \) is such that
  \begin{equation}\label{eq:eigenvalue_separation}
    B(\theta_i, \sigma) \cap \left(\bigcup_{j \neq i} B(\theta_j, \sigma) \right) = \emptyset,
  \end{equation}
  then there exists a normed eigenvector \( \xi \) associated with \( \lambda_{\pi(i)} \) such that
  \[ \norm*{\xi - \frac{u_i}{\norm{u_i}}} \leq \frac{3\sigma}{\delta_i - \sigma}, \]
  where \( \delta_i \) is the minimum distance from \( \theta_i \) to another eigenvalue:
  \[ \delta_i = \min_{j \neq i} {|\theta_j - \theta_i|} \geq 2\sigma. \]
\end{theorem}

\begin{proof}
  We begin with defining an alternative matrix \( \bar U \) such that \( \bar U^*V = I_{r} \). Let \( H_i \) be the subspace of \( \dR^n \) such that
  \[ H_i = \vect(v_j \ \vert\  j \neq i), \]
  and consider the vectors \( \tilde u_i \) and \( \bar u_i \) defined as
  \[ \tilde u_i = u_i - P_{H_i}(u_i) \quand \bar u_i = \frac{\tilde u_i}{\langle \tilde u_i, v_i \rangle} \]
  with \( P_{H_i} \) the projection on \( H_i \), and \( \tilde U \), \( \bar U \) the associated \( n \times r \) matrices. Then it is straightforward to see that
  \[ \langle \bar u_i, v_i \rangle = 1 \quand \langle \bar u_i, v_j \rangle = 0, \]
  for all \( j \neq i \), which shows that \( \bar U^*V = I_{r} \).
  Now, if we let \( V_i \) be the matrix \( V \) with the \( i \)-th column and row deleted,
  \[ P_{H_i} = V_i{(V_i^*V_i)}^{-1}V_i^*, \]
  and
  \[ \norm{V_i^*u_i}^2 = \sum_{j \neq i} \langle v_j, u_i \rangle^2 \leq r \delta^2, \]
  hence we can compute \( \norm{u_i - \tilde u_i} \):
  \[ \norm{u_i - \tilde u_i} = \norm{P_{H_i}(u_i)} \leq \norm{V_i}\norm{{(V_i^*V_i)}^{-1}} \norm{V_i^*u_i}, \]
  and by the interlacing theorem \( \norm{V_i} \leq \sqrt{\alpha} \) and \( \norm{{(V_i^*V_i)}^{-1}} \leq \beta \) since \( V_i \) is a principal submatrix of \( V \). Using the fact that \( \norm{M} \leq \norm{M}_F \) for any matrix \( M \), we find
  \[ \norm{U - \tilde U} \leq r^2 \sqrt{\alpha} \beta \delta.  \]

  For the second part, note that by the Cauchy-Schwarz inequality,
  \begin{align*}
    \left|\langle \tilde u_i, v_i \rangle - 1 \right| & \leq \left|\langle u_i, v_i \rangle - 1 \right| + \norm{u_i - \tilde u_i} \cdot \norm{v_i} \\
                                                      & \leq \delta(1 + r\alpha \beta),
  \end{align*}
  with the (generous) inequality \( \norm{v_i} \leq \norm{V} \) used in the last line. Whenever \( \delta \) is small enough, we can use the inequality \( \left|{(1 - t)}^{-1} - 1 \right| \leq 2t \) which is valid for \( t \leq 1/2 \):
  \[ \left|\frac{1}{\langle \tilde u_i, v_i \rangle} - 1 \right| \leq 2 \delta(1 + r\alpha \beta). \]
  As a result,
  \begin{align*}
    \norm{\bar u_i - \tilde u_i} & = \norm{\tilde u_i} \left|\frac{1}{\langle \tilde u_i, v_i \rangle} - 1 \right| \\
                                 & \leq 2 \delta \sqrt{\alpha} (1 + r\alpha\beta)                                  \\
                                 & \leq 4r \alpha^{3/2} \beta \delta.
  \end{align*}
  Using again the norm equivalence bound and the triangle inequality,
  \begin{equation}\label{eq:perturbation_bound_u}
    \norm{\bar U - U} \leq 5r^2\alpha^{3/2} \beta \delta,
  \end{equation}
  which ends the preliminary part of the proof.

  \bigskip

  We now set accordingly \( \bar S = \bar U \Sigma V^* \), and claim that \( S \) is now a truly diagonalizable matrix. Indeed, any \( \bar u_i \) is an eigenvector of \( \bar S \) with associated eigenvalue \( \theta_i \), and a basis of \( \img{(V)}^\bot \) provides a family of eigenvectors of \( \Sigma \) with eigenvalue \( 0 \). We consequently set
  \[ \Pi = \begin{pmatrix} \bar U & Y \end{pmatrix}, \]
  where \( Y \) is an orthonormal basis of \( \img{(V)}^\bot \); \( \Pi \) is the matrix of an eigenvector basis for \( S \). Further, we have
  \[ \norm{\bar S - S} \leq \norm{U - \bar U} \norm{\Sigma} \norm{V} \leq 5r^2 \alpha^2 \beta \delta := \eps'. \]
  The above bound implies that the matrices \( A \) and \( \bar S \) are still close:
  \begin{equation}\label{eq:bauer_fike_condition}
    \norm{A - \bar S} \leq \norm{A - S} + \norm{S - \bar S} \leq \eps + \eps',
  \end{equation}
  and we can apply the Bauer-Fike theorem to \( A \) and \( \bar S \); the eigenvalues of \( A \) are contained in the union of the balls \( B(\theta_i, \eps'') \) and \( B(0, \eps'') \), where
  \[ \eps'' = (\eps + \eps')\kappa(\Pi). \]
  The computation of \( \kappa(\Pi) \) being cumbersome, we defer the following lemma:
  \begin{lemma}\label{lem:norm_inverse}
    Let \( X \) be a \( n\times r \) matrix with rank \( r \), and \( X \) such that \( X^* X'  = I_r \). Let \( Y \) be a matrix for an orthonormal basis of \( \img{(X')}^{\bot} = \ker({(X')}^*) \), and \( P = (X, Y) \).
    Then, if \( \norm{X} \geq 1 \) and \( \norm{X'} \geq 1 \),
    \[ \norm{\Pi} \leq \sqrt{2}\norm{X} \quand \norm{\Pi^{-1}} \leq \sqrt{2}(1 + \norm{X}\norm{X'})\]
  \end{lemma}
  Applying this to \( X = \bar U \) and \( X' = V \) gives the bound
  \[ \kappa(\Pi) \leq 2\left( \norm{\bar U} + \norm{\bar U}^2\norm{V} \right), \]
  and we use the triangle inequality to bound \( \norm{\bar U} \):
  \[ \norm{\bar U} \leq \norm{U} + \norm{\bar U - U} \leq 6r^2 \alpha^{3/2} \beta, \]
  a very loose but sufficient bound, that entails
  \[ \kappa(\Pi) \leq 84 r^2\alpha^{7/2}\beta. \]
  The corresponding bound on \( \eps'' \) reads
  \[ \eps'' \leq 84 r^2\alpha^{7/2}\beta(\eps + 5r \alpha^2 \beta \delta), \]
  and we define $\sigma$ to be the right-hand side of this inequality:
  \begin{equation}\label{eq:perturbation_definition_sigma}
    \sigma := 84 r^2\alpha^{7/2}\beta(\eps + 5r \alpha^2 \beta \delta)
  \end{equation}
  Going back to the Bauer-Fike application, the separation condition~\eqref{eq:perturbation_eigen_separation} implies that \( B(0, \sigma) \) is disjoint from \( B(\theta_i, \sigma) \) for \( i \in [r] \) and we can apply the second part of the theorem: there are exactly \( r \) eigenvalues of \( A \) inside the region
  \[ \Omega = \bigcup_{i\in [r]} B(\theta_i, \sigma), \]
  and all other eigenvalues of \( A \) have modulus less than \( \sigma \). Further, again by the second part of Theorem \ref{th:bauer_fike}, all connected components of $\Omega$ have the same number of eigenvalues of $A$ and $B$. As a result, there exists a permutation $\pi$ such that for all \( i \in [r] \), we have
  \[ \left| \lambda_{\pi(i)} - \theta_i \right| \leq \sup_{\Omega' \subseteq \Omega}\diam(\Omega') \leq 2r\sigma, \]
  where the supremum is taken over all connected subsets of \( \Omega \).

  We now move on to the eigenvector perturbation bound; let \( \xi \) be a normed eigenvector of \( A \) associated with the eigenvalue \( \lambda_{\pi(i)} \). We write \( \xi = \Pi x \) with \( \Pi \) the matrix defined before, and use~\eqref{eq:bauer_fike}:
  \begin{equation*}
    \norm*{\lambda_{\pi(i)} \Pi x - \sum_{j = 1}^r {\theta_j x_j \bar u_j}} = \norm*{\left(A - \bar S \right)x} \leq \eps + \eps',
  \end{equation*}
  which we rewrite as
  \begin{equation*}
    \norm*{\Pi\left( \lambda_{\pi(i)} x - \sum_{j \in [r]} \theta_j x_j e_j \right)} \leq \eps + \eps',
  \end{equation*}
  with \( (e_1, \dots, e_n) \) the usual orthonormal basis of \( \dR^n \). Using the inequality \( \norm{v} \leq \norm{P^{-1}}\norm{Pv} \) holding for any vector \( v \),
  \[ \norm*{ \lambda_{\pi(i)} x - \sum_{j \in [r]} \theta_j x_j e_j} \leq \norm{\Pi^{-1}}(\eps + \eps'). \]
  We introduce the notation \( \theta_{r+1} = \cdots = \theta_n = 0 \); whenever the ball \( B(\theta_i, \sigma) \) is disjoint from all other such balls, we have \( |\lambda_{\pi(i)} - \theta_i| \leq \sigma \), and thus for \( j \neq i \)
  \[ |\lambda_{\pi(i)} - \theta_j| \geq |\theta_j - \theta_i| - |\lambda_{\pi(i)} - \theta_i| \geq \delta_i - \sigma, \]
  so that
  \[ \norm{x - x_i e_i} = \norm*{\sum_{j \neq i} x_j e_j} \leq \frac{1}{\delta_j - \sigma} \norm*{\sum_{j \neq i} (\lambda_{\pi(i)} - \theta_j) x_j e_j} \leq \frac{\norm{\Pi^{-1}}(\eps + \eps')}{\delta_j - \sigma}. \]
  We now apply \( \Pi \) inside the norm the LHS, and use the fact that \( \kappa(\Pi) (\eps + \eps') \leq \sigma \):
  \[ \norm{\xi - x_i \bar u_i} \leq \frac{\sigma}{\delta_i - \sigma}. \]
  Now, for any vectors \( w, w' \in \dR^n \), we have
  \begin{equation}\label{eq:normed_diff_bound}
    \norm*{\frac{w}{\norm{w}} - \frac{w'}{\norm{w'}}} \leq \frac{2\norm{w - w'}}{\norm{w}},
  \end{equation}
  and all that remains is to write
  \begin{align*}
    \norm*{\xi - \frac{u_i}{\norm{u_i}}} & \leq \norm{\xi - \frac{\bar u_i}{\norm{\bar u_i}}} + \norm*{\frac{u_i}{\norm{u_i}} - \frac{\bar u_i}{\norm{\bar u_i}}} \\
                                         & \leq \frac{2\sigma}{\delta_i - \sigma} + 2\norm{u_i - \bar u_i}                                                        \\
                                         & \leq \frac{3\sigma}{\delta_i - \sigma},
  \end{align*}
  having used~\eqref{eq:normed_diff_bound} twice and \( 2\norm{u_i - \bar u_i} \leq \sigma \). This ends the proof.
\end{proof}

As announced, we now prove the aforementioned Lemma~\ref{lem:norm_inverse} on the condition number of \( P \):
\begin{proof}
  Let \( z \in \dR^n \) be a unit vector, and write \( z = \binom{x}{y} \) with \( x \) of size \( r \) and \( y \) of size \( n - r \). Then, using that \( \norm{Y} = 1 \),
  \begin{align*}
    \norm{\Pi z} = \norm{X x + Yy} & \leq \norm{X} \cdot \norm{x} + \norm{Y}\cdot \norm{y}    \\
                                   & \leq \left( 1 \vee \norm{X} \right)(\norm{x} + \norm{y}) \\
                                   & \leq \sqrt{2}\norm{X},
  \end{align*}
  which proves the first inequality. The second one relies on the following explicit formula for \( \Pi^{-1} \):
  \[ \Pi^{-1} = \begin{pmatrix}
      {(X')}^* \\
      - Y^* X {(X')}^* + Y^*
    \end{pmatrix}.\]
  Indeed, using the relations \( Y^*Y = I_{n-r} \) and \( {(X')}^*Y = 0 \):
  \begin{align*}
    \begin{pmatrix}
      {(X')}^* \\
      - Y^* X {(X')}^* + Y^*
    \end{pmatrix}P & = \begin{pmatrix}
      {(X')}^* \\
      - Y^* X {(X')}^* + Y^*
    \end{pmatrix} (X\  Y) \\
                                & =\begin{pmatrix}
      {(X')}^* X                 & {(X')}^* Y                 \\
      - Y^* X {(X')}^* X + Y^* X & - Y^* X {(X')}^* Y + Y^* Y
    \end{pmatrix}          \\
                                & =\begin{pmatrix}
      I_r            & 0     \\
      - Y^* X+ Y^* X & Y^* Y
    \end{pmatrix}          \\
                                & =\begin{pmatrix}
      I_r & 0       \\
      0   & I_{n-r}
    \end{pmatrix}          \\
                                & = I_n.
  \end{align*}

  Furthermore, we have
  \[ \norm{- Y^* X {(X')}^* + Y^*} \leq \norm{Y}\norm{I_{n} - X {(X')}^*} \leq 1 + \norm{X}\norm{X'}, \]
  and the exact same argument as in the first inequality yields
  \[ \norm{P^{-1}} \leq \sqrt{2}\left(1 + \norm{X}\norm{X'}\right) \]
\end{proof}

\subsection{Matrix power perturbation and phase perturbation control}
We aim in the following section to apply Theorem~\ref{th:bauer_personal} to powers of the matrix \( B \); however, such a process introduces uncertainty on the phase of the eigenvalues of \( B \). The next theorem, adapted from~\cite{bordenave_detection_2020} and~\cite{bordenave_nonbacktracking_2018}, develops a method to control such uncertainty. As before, let \( \Sigma = \diag(\theta_1, \dots, \theta_r) \) with
\[ 1 = |\theta_1| \geq \cdots \geq |\theta_r|, \]
and \( U, U', V, V' \) four \( n \times r \) matrices. We set
\[ S = U\Sigma^\ell V^* \quand S' = U'\Sigma^{\ell'}{(V')}^*, \]
for two integers \( \ell, \ell' \).

\begin{theorem}\label{th:bauer_powers}
  Assume the following:
  \begin{enumerate}
    \item the integers \( \ell, \ell' \) are relatively prime,\label{item:bauer_power_cond_i}

    \item the matrices \( U, U', V, V' \) are well-conditioned:\label{item:bauer_power_cond_ii}
          \begin{itemize}
            \item they all are of rank \( r \),
            \item for some \( \alpha, \beta \geq 1 \), for \( X \) in \( \{U, V, U', V' \} \),
                  \[ \norm{X^*X} \leq \alpha \quand \norm{{(X^*X)}^{-1}} \leq \beta, \]
            \item for some small \( \delta < 1 \),
                  \[ \norm{U^*V - I_r} \leq \delta \quand  \norm{{(U')}^*V' - I_r} \leq \delta, \]
          \end{itemize}
    \item there exists a small constant \( \eps > 0 \) such that\label{item:bauer_power_cond_iii}
          \[ \norm{A^\ell - S} \leq \eps \quand \norm{A^{\ell'} - S'} \leq \eps, \]
    \item if we let\label{item:bauer_power_cond_iv}
          \[ \sigma_0 := 84 r^3\alpha^{7/2}\beta(\eps + 5r \alpha^2 \beta \delta), \]
          then
          \begin{equation}\label{eq:power_perturbation_bound}
            \sigma_0 < \ell \, |\theta_r|^\ell \quand \sigma_0 < \ell'\, |\theta_r|^{\ell'}.
          \end{equation}
  \end{enumerate}
  Assume without loss of generality that \( \ell \) is odd, and let
  \[ \sigma := \frac{\sigma_0}{\ell |\theta_r|^\ell}. \]
  Then, the \( r \) largest eigenvalues of \( A \) are close to the \( \theta_i \) in the following sense: there exists a permutation $\pi$ of $[r]$ such that for \( i\in [r] \),
  \[ \left| \lambda_{\pi(i)} - \theta_{i} \right| \leq 4\sigma, \]
  and all other eigenvalues of \( A \) are less that \( \sigma_0^{1/\ell} \). Additionally, if \( i \) is such that
  \begin{equation}\label{eq:eigenvalue_power_separation}
    B(\theta_i, \sigma) \cap \left(\bigcup_{j \neq i} B(\theta_j, \sigma) \right) = \emptyset,
  \end{equation}
  then there exists a normed eigenvector \( \xi \) associated to \( \lambda_{\pi(i)} \) such that
  \[ \norm*{\xi - \frac{u_i}{\norm{u_i}}} \leq \frac{3\sigma}{\delta_i - \sigma}, \]
  with \( \delta_i \) defined as in Theorem~\ref{th:bauer_personal}.
\end{theorem}

\begin{proof}
  We apply Theorem~\ref{th:bauer_personal} to \( A^\ell, S \) and \( A^{\ell'}, S' \); for any \( i \in [r] \),
  \begin{equation}\label{eq:power_perturbation_error_bound}
    \left| \lambda_{\pi(i)}^\ell - \theta_{i}^\ell \right| \leq \sigma_0 \quand \left| \lambda_{\pi'(i)}^{\ell'} - \theta_{i}^\ell \right| \leq \sigma_0.
  \end{equation}
  Examining the proof of Theorem \ref{th:bauer_personal}, we notice that we can take $\pi = \pi'$ since taking the $\ell$-th power does not change the ordering.
  We fix \( i \in [r] \) and let \( \lambda = \lambda_{\pi(i)} = |\lambda|e^{i\omega} \) and \( \theta = \theta_{i} \) for now; then
  \[ \left| \frac{\lambda^\ell}{\theta^\ell} - 1 \right| \leq \nu := \frac{\sigma_0}{|\theta|^\ell}. \]

  The argument of \( {(\lambda/\theta)}^\ell \) is thus between \( -\xi \) and \( \xi \), with
  \[ \xi = \left|2\arcsin(\nu/2) \right| \leq \pi/2 \nu, \]
  and the same holds for \( \ell' \) (with \( \nu' \) defined accordingly). Thus, there exists two integers \( p, p' \) and two numbers \( s, s' \) with absolute value less than \( \pi/2 \nu \) (resp. \( \pi/2 \nu' \)), such that
  \[ \ell\omega = p\pi + s \quand \ell'\omega = p'\pi + s'. \]
  This implies
  \[ p\ell' - p'\ell = \frac{s'\ell - s\ell'}{\pi} \]
  The LHS of this inequality is an integer, and using condition~\eqref{eq:power_perturbation_bound} both terms in the RHS have a magnitude strictly lower than \( 1/2 \), so both sides are 0. As \( \ell \) and \( \ell' \) are relatively prime, \( \ell \) divides \( p \) and \( \ell' \) divides \( p' \), so that
  \[ \omega = k\pi + \frac s \ell. \]
  Whenever \( \theta_{i} \) is positive, \( k \) is even and we can take \( \omega = s/\ell \), and when \( k \) is odd we choose \( \omega = \pi + s/\ell \).

  We now come back to~\eqref{eq:power_perturbation_error_bound}, and write
  \[ \lambda_i^\ell = \theta_{i}^\ell(1+z) \]
  with \( |z| \leq \nu \). Taking the modulus on both sides we find \( |\lambda_i| = |\theta_{i}||1+z|^{\frac 1 \ell} \) and we use the inequality \( ||1+z|^{\frac1\ell} - 1| \leq 2|z|/\ell \) (valid for \( |z| \leq 1/2 \)) to find
  \[ \left| |\lambda_i| - |\theta_{i}|\right| \leq \frac{2\sigma_0}{\ell |\theta_{i}|^\ell}.  \]
  We can now prove the lemma: whether \( \theta \) is positive or negative, a case analysis yields
  \begin{align*}
    |\lambda_i - \theta_{i}| & \leq \left| |\lambda_i| - |\theta_{i}|\right| + |\theta_i|\left|e^{is/\ell} - 1\right| \\
                             & \leq \frac{2\sigma_0}{\ell |\theta_{i}|^\ell} + |\theta_i|\frac{|s|}\ell               \\
                             & \leq \frac{4\sigma_0}{\ell |\theta_{i}|^\ell},
  \end{align*}
  the desired bound. Now, assuming that \( \ell \) is odd, we have by the mean value theorem
  \[ |\theta_i^\ell - \theta_j^\ell| \geq \ell {(|\theta_i|\wedge |\theta_j|)}^{\ell - 1} |\theta_i - \theta_j| \geq \ell |\theta_r|^\ell |\theta_i - \theta_j|,  \]
  so that condition~\eqref{eq:eigenvalue_power_separation} implies the separation condition~\eqref{eq:eigenvalue_separation} applied to \( A^\ell \). We can then apply the same proof as in Theorem~\ref{th:bauer_personal} and get
  \[ \norm*{\xi - \frac{u_i}{\norm{u_i}}} \leq \frac{3\sigma_0}{\ell |\theta_r|^\ell \delta_i - \sigma_0}, \]
  which is equivalent to the theorem bound.
\end{proof}

\section{Proof of Theorem~\ref{th:main}}

We prove in this section the main result on the spectral properties of \( B \). We shall use the same notations as in Theorem \ref{th:main}; since the statement of the theorem is invariant upon multiplying the entries of $W$ by a common constant, we shall assume in the rest of the paper that $\mu_1 = 1$.

Our candidates for the singular vectors of \( B^\ell \) are the vectors \( (u_1, \dots, u_{r_0}) \) and \( (v_1, \dots, v_{r_0}) \), where for \( i \in [r_0] \)
\begin{equation}\label{eq:def_ui_vi}
  u_i = \frac{B^\ell \chi_i}{\mu_i^{\ell}}\quand v_i = \frac{{(B^*)}^\ell D_W\check\chi_i}{\mu_i^{\ell+1}},
\end{equation}
with associated eigenvalue \( \mu_i^\ell \). We let \( U \) (resp. \( V \)) be the \( n \times r \) matrix whose columns are the \( u_i \) (resp \( v_i \)), and \( D = \diag(\mu_1, \dots, \mu_{r_0}) \). The subspace spanned by the vectors $(v_1, \dots, v_{r_0})$ will be denoted by $H$, and we let $P_H$ and $P_{H^\bot}$ be the projections on $H$ and its orthogonal $H^\bot$, respectively.

Finally, we'll need an approximation of the Gram matrix of the vectors \( u \) (and \( v \)); we define for every \( t \geq 0 \) the \emph{covariance} matrices \( \Gamma_U^{(t)} \) and \( \Gamma_V^{(t)} \) such that for \( i, j \in [r_0] \),
\begin{equation} \label{eq:def_gamma}
  \Gamma_{U, ij}^{(t)} = \sum_{s = 0}^t \frac{\langle P\ind, K^s\varphi^{i, j} \rangle}{{(\mu_i \mu_j)}^s} \quand \Gamma_{V, ij}^{(t)} = \sum_{s = 0}^t \frac{\langle K \ind, K^s\varphi^{i, j} \rangle}{{(\mu_i \mu_j)}^{s+1}},
\end{equation}
where \( \varphi^{i, j} = \varphi_i \circ \varphi_j \).

\begin{remark}
  In the classical stochastic block model, we have $K = Q$ and the all-one vector $\ind$ is an eigenvector of $Q$. This implies that the matrices $\Gamma_{U, ij}^{(t)}$ and $\Gamma_{V, ij}^{(t)}$ are diagonal, and thus the $u_i$ (resp. $v_i$) are asymptotically orthogonal. This greatly simplifies the perturbation analysis of Theorem \ref{th:bauer_personal} for this special case.
\end{remark}

\subsection{Structure of the matrices \texorpdfstring{\( U \)}{U} and \texorpdfstring{\( V \)}{V}}

Following from the subsequent local analysis of \( G \), as well as a trace bound argument, we gather the following relations between matrices \( B^\ell \), \( D^\ell \) and \( U \).

\begin{theorem}\label{th:bl_u_bounds}
  Let \( r, d, b, \tau, L \) be parameters as above, such that \( a \leq n^{1/4} \), and \( (P, W) \) be any matrices in \( \mathcal C(r, d, b, \tau, L) \). Let \( \ell \) be any integer such that
  \begin{equation}\label{eq:bound_ell}
    \ell \leq \frac{1 - \epsilon}{16} \frac{\log(n)}{\log(d)},
  \end{equation}
  for some \( \epsilon > 0 \). Then there exists an event with probability at least \( 1 - c/\log(n) \) and a parameter \( N_0 \leq a^{12}L^6 \) such that if \( n \geq N_0 \)

  \begin{align}
    \norm{U^*U - \Gamma_U^{(\ell)}} & \leq C\times n^{-1/4} , \label{eq:Ustar_U}                          \\
    \norm{V^*V - \Gamma_V^{(\ell)}} & \leq C\times n^{-1/4}, \label{eq:Vstar_V}                           \\
    \norm{U^*V - I_{r_0}}_\infty    & \leq C\times n^{-1/4}, \label{eq:Ustar_V}                           \\
    \norm{B^\ell U - UD^\ell}       & \leq C'{(\sqrt{\rho} \vee L)}^\ell, \label{eq:Bl_U}            \\
    \norm{B^\ell P_{H^\bot}}        & \leq C'{(\sqrt{\rho} \vee L)}^\ell, \label{eq:norm_orthogonal}
  \end{align}
  where \( C \) and \( C' \) satisfy
  \[ C\leq c r d^4 b^2 L \quand C' \leq c r^2 d^6 b^2 L^2 \log{(n)}^{20}.\]
  Furthermore, on this same event, we have the following bound:
  \begin{equation}
    \norm{B^\ell} \leq c \log(n) n^{1/4} L^{\ell}. \label{eq:norm_Bl}
  \end{equation}
\end{theorem}

The proof of this theorem will occupy the next few pages of this article; we first show how it implies the statement of Theorem~\ref{th:main}.

\subsection{Proof of the perturbation bounds}

The goal here is to apply Theorem~\ref{th:bauer_powers} to \( B^\ell \), \( U \) and \( V \): we choose \( \ell \) equal to the upper bound in~\eqref{eq:bound_ell} (with arbitrary \( \epsilon \), say \( 0.01 \)) and \( \ell' = \ell + 1 \), and let
\[ S = UD^\ell V^* \quand S' = U'D^{\ell'}V'^*, \]
where $U', V'$ are defined identically to $U$ and $V$ replacing $\ell$ by $\ell'$.
We now check all the conditions of Theorem~\ref{th:bauer_powers}:

\medskip
\textbf{Condition~\ref{item:bauer_power_cond_i}} Since \( \ell' = \ell + 1 \), \( \ell \) and \( \ell' \) are relatively prime.

\medskip
\textbf{Condition~\ref{item:bauer_power_cond_ii}} We shall need a small lemma on the spectral properties of the covariance matrices, which will be proven in a subsequent section:
\begin{lemma}
  For all \( t \geq 1 \); the matrix \( \Gamma_U^{(t)} \) (resp. \( \Gamma_V^{(t)} \)) is a positive definite matrix, with all its eigenvalues greater than 1 (resp. \( c_0^{-1} \)) and such that
  \[ 1 \leq \norm{\Gamma_U^{(t)}} \leq \frac{r^2 d^3 L^2}{1 - \tau} \quand c_0^{-1} \leq \norm{\Gamma_V^{(t)}} \leq \frac{r^2 d^2 L^2}{1 - \tau}.\]
\end{lemma}
Then, the minimum eigenvalue of \( V^*V \) is at least \( c_0^{-1} - Cn^{-1/4} \), which is more than \( c_0^{-1}/2 \) as soon as
\[ n \geq c_1 r^4a^4d^{16}b^8 L^4, \]
and we can take \( \beta = 2c_0 \) whenever this holds. On the other hand,
\[ \norm{V^*V} \leq \frac{r d^2 L^2}{1 - \tau} + \frac{r b^2 d^4 L}{n^{1/4}} \leq \frac{2 r b^2 d^4 L^2}{1 - \tau}. \]
Performing the same computations on \( U^*U \) leads us to the choice
\[ \alpha = \frac{2 r b^2 d^4 L^2}{1 - \tau}. \]
Finally, equation~\eqref{eq:Ustar_V} allows us to take
\[ \delta = Cn^{-1/4}. \]

\medskip
\textbf{Condition~\ref{item:bauer_power_cond_iii}} This condition requires some additional computations. Recall that $H = \mathrm{im}(V)$; we have the formula
\[ P_H = V{(V^*V)}^{-1}V^*. \]
Noticing that \( SP_H = S \), we can bound \( \norm{B^\ell - S} \) as follows:
\begin{align*}
  \norm{B^\ell - S} & \leq \norm{B^\ell P_H - SP_H} + \norm{SP_{H^\bot}} + \norm{B^\ell P_{H^\bot}} \\
                    & \leq \norm{B^\ell P_H - S} + \norm{B^\ell P_{H^\bot}}                         \\
                    & \leq \norm{B^\ell V{(V^*V)}^{-1} - UD}\norm{V^*} + \norm{B^\ell P_{H^\bot}}.
\end{align*}
To apply~\eqref{eq:Bl_U}, we let
\[ U = P_H U + P_{H^\bot U} = V{(V^*V)}^{-1} + \tilde U + P_{H^\bot U}. \]
The second term is equal to \( V{(V^*V)}^{-1}(V^*U - I_{r_0}) \), and be can thus use~\eqref{eq:Ustar_V}:
\[ \norm{\tilde U} \leq \norm{V}\norm{{(V^*V)}^{-1}} \norm{V^*U - I_{r_0}} \leq r\sqrt{\alpha}\beta \delta. \]
Going back to the above inequality, we find
\begin{equation*}
  \norm{B^\ell - S} \leq \norm{B^\ell U - UD^\ell} + \norm{B^\ell}\norm{\tilde U} + \norm{B^\ell P_{H^\bot}}\norm{U} + \norm{B^\ell P_{H^\bot}},
\end{equation*}
and the bounds in Theorem~\ref{th:bl_u_bounds} readily imply that all terms in the above inequality are bounded above by \( \eps := C'' {(\sqrt{\rho} \vee L)}^\ell \), with
\[ C'' \leq \frac{c_3 r^3 a d^{8} b^3 L^3 \log{(n)}^{20}}{1 - \tau}. \]

\medskip
\textbf{Condition~\ref{item:bauer_power_cond_iv}} Using all the bounds proven in the above computations, we find that
\[ \sigma_0 \leq C_0 {(\sqrt{\rho} \vee L)}^\ell \quad \text{with} \quad C_0 \leq \frac{c_4\,a^2 r^{11}d^{25}b^{13}L^{12}\log{(n)}^{20}}{{(1 - \tau)}^6}. \]
The bound we have to check is therefore
\[ C_0 {(\sqrt{\rho} \vee L)}^\ell \leq \ell |\mu_{r_0}|^\ell \quad\Longleftrightarrow\quad C_0 \tau^\ell \leq \ell, \]
which happens as soon as
\[ \log(n) \geq \frac{20 \log(C_0)\log(d)}{\log(\tau^{-1})}. \]
The same proof holds for \( \ell' \), with the same constants.

\bigskip
Having checked all assumptions of Theorem~\ref{th:bauer_powers}, we can now apply it to \( B^\ell \); this implies the existence of a permutation \( \pi \in \mathfrak S_{r_0} \) (possibly depending on \( n \)) such that for \( i \in [r_0] \),
\[ \left| \lambda_i - \mu_{\pi(i)} \right| \leq \sigma := C_0 \tau^\ell, \]
and all the other eigenvalues of \( B \) satisfy
\[ |\lambda| \leq C_0^{\frac 1 \ell} (\sqrt{\rho} \vee L). \]

Now, assume that for some \( i \in [r_0] \), \( \delta_i \geq 2\sigma \). Then, applying the last part of Theorem~\ref{th:bauer_powers}, there exists an eigenvector of \( B \) associated with \( \lambda_i \) such that
\[ \norm*{\xi - \frac{u_i}{\norm{u_i}}} \leq \frac{3\sigma}{\delta_i - \sigma}. \]
We define in the following
\[ \gamma_i = \langle P\ind, {\left(I_{n} - \mu_i^{-2}K\right)}^{-1} \varphi^{i, i} \rangle. \]
If we rewrite the definition of \( \Gamma_{U, ii}^{(t)} \) as
\[ \Gamma_{U, ii}^{(t)} = \left\langle P\ind,\sum_{s=0}^t {(\mu_i)}^{-2s}K^s \varphi^{i, j} \right\rangle, \]
the matrix sum converges as \( t \to \infty \) since \( \rho(K) < \mu_i^2 \), and using Lemma~\ref{lem:scalar_Kt} below we have
\begin{align*}
  \left|\Gamma_{U, ii}^{(\ell)} - \gamma_i \right| & = \sum_{t = \ell + 1}^{\infty} \mu_i^{-2t}\langle P\ind, K^t \varphi^{i, j} \rangle \\
                                                   & \leq \sum_{t = \ell + 1}^{\infty} r d^3 L^2 \rho^t\mu_i^{-2t}                       \\
                                                   & \leq \sigma,
\end{align*}
and combined with~\eqref{eq:Ustar_U} yields
\[ \left|\norm{u_i}^2 - \gamma_i \right| \leq 2\sigma. \]
On the other hand, we shall prove the following inequality in the following sections (equation~\eqref{eq:u_phi_dotp}): for all \( t\leq 2\ell \),
\[ \left|\langle B^t \chi_i, \chi_i \rangle - \mu_i^t \langle P\ind,  \varphi^{i, i} \rangle \right| \leq \sigma \mu_i^t. \]
Setting \( t = 0 \) and \( t = \ell \) in this inequality yields at the same time
\[ \left|\norm{\chi_i}^2 - \langle P\ind,  \varphi^{i, i} \rangle\right| \leq \sigma \quand \left|\langle u_i, \chi_i \rangle - \langle P\ind,  \varphi^{i, i} \rangle\right| \leq \sigma. \]
We now have, using the Cauchy-Schwarz inequality,
\begin{align*}
  \left|\langle \xi, \xi_i \rangle - \sqrt{\frac{\langle P\ind,  \varphi^{i, i} \rangle}{\gamma_i}}\right| & \leq \norm*{\xi - \frac{u_i}{\norm{u_i}}} +
  \left|\left\langle \frac{u_i}{\norm{u_i}}, \frac{\chi_i}{\norm{\chi_i}} \right \rangle - \sqrt{\frac{\langle P\ind,  \varphi^{i, i} \rangle}{\gamma_i}}\right|  \\
                                                                                                           & \leq \frac{3\sigma}{\delta_i - \sigma} + c_5\,\sigma \\
                                                                                                           & \leq \frac{c_6\,\sigma}{\delta_i - \sigma}.
\end{align*}
Finally, notice that
\[ \gamma_i = \langle P\ind,  \varphi^{i, i} \rangle + \left\langle P\ind,\sum_{s=1}^\infty {(\mu_i)}^{-2s}K^s \varphi^{i, j} \right\rangle \leq \langle P\ind,  \varphi^{i, i} \rangle + r d^2 L^2 \frac{\rho/\mu_i^2}{1 - \rho/\mu_i^2}. \]
Using that \( r d^2 L^2 \geq 1 \) and \( \langle P\ind,  \varphi^{i, i} \rangle \geq 1 \), we find
\[ \frac{\langle P\ind,  \varphi^{i, i} \rangle}{\gamma_i} \geq 1 - r d^2 L^2 \frac{\rho}{\mu_i^2}. \]

\section{Preliminary computations}

We begin the proof of Theorem~\ref{th:bl_u_bounds} with some elementary computations on the entries of \( K \) and \( \Gamma^{(t)} \), which will be of use in the later parts of the proof. Most of the results from this section are adapted from~\cite{bordenave_detection_2020}, although sometimes improved and adapted to our setting.

\paragraph{Bounding \( \rho \) and \( L \) from below} We begin with a simple bound on \( \rho = \rho(K) \); by the Courant-Fisher theorem, \( \rho \geq \langle w, Kw \rangle \) for every unit vector \( w \), and applying it to \( w = \ind/\sqrt{n} \) yields
\begin{align*}
    \rho &\geq \frac{\langle w, Kw \rangle}{n} \\
         &= \frac1n \sum_{i, j\in[n]} P_{ij}\E*{W_{ij}^2} \\
         &= \frac1d \sum_{i, j\in [n]} P_{ij}^2 \E*{W_{ij}}^2 \\
         &= \frac{\norm{Q}^2_F}d,
\end{align*}
where we used that \( P_{ij} \leq d/n \) and the Jensen inequality. The Frobenius norm of \( Q \) is then greater than \( \mu_1^2 = 1 \), which in turns implies
\begin{equation}\label{eq:bound_rho_below}
    \rho \geq \frac1d,
\end{equation}
so that \( \rho \) is bounded away from zero. In order to prove a similar bound on \( L \), we write for \( x\in [n] \)
\[ \varphi_1(x) = \sum_{y \in [n]}Q_{xy}\varphi_1(y) \leq \sqrt{\sum_{y} Q_{xy}^2} \leq \frac{dL}{\sqrt{n}}. \]
Squaring and summing those inequalities over \( x \) gives
\[ 1 = \norm{\varphi_1}^2 \leq d^2 L^2, \]
so that as with \( \rho \),
\begin{equation}\label{eq:bound_L_below}
    L \geq \frac 1 d.
\end{equation}

\paragraph{A scalar product lemma} Our second step is an important lemma for the following proof, leveraging the entrywise bounds on \( W \):

\begin{lemma}\label{lem:scalar_Kt}
    Let \( \varphi, \varphi' \in \dR^n \) be any unit vectors. Then, for any \( t \geq 0 \),
    \[ \langle \ind, K^t \varphi \circ \varphi' \rangle \leq r d^2 L^2\rho^t \]

\end{lemma}

\begin{proof}
We write the eigendecomposition of \( K \) as
\[ K = \sum_{k = 1}^s \nu_k \psi_k \psi_k^*, \]
with \( \nu_1 = \rho \) the Perron-Frobenius eigenvalue of \( K \) and \( s \leq r^2 \) its rank. Then, for all \( i\in [n] \),

\begin{align*} 
    \sum_{k = 1}^s \nu_k^2 \psi_k{(i)}^2 = {(K^2)}_{ii} &= \sum_{j \in [n]} K_{ij}^2 \\
    &= \sum_{j\in [n]} P_{ij}^2 \E*{W_{ij}^2}^2 \\
    &\leq \sum_{j\in [n]} {\left(\frac d n\right)}^2 L^4 \\
    &\leq \frac{d^2 L^4}n.
\end{align*}
This is akin to a delocalization property on the eigenvectors of \( K \).

We can now prove the above lemma:
\begin{align*}
    \langle \ind, K^t \varphi\circ\varphi' \rangle &= \sum_{k=1}^s \nu_k^t \langle \ind, \psi_k \rangle \langle \psi_k, \varphi\circ\varphi' \rangle \\
    &\leq \rho^{t-1}\sum_{k=1}^s \norm{\psi_k}\norm{\ind} \cdot |\nu_k| \left| \langle \psi_k, \varphi\circ\varphi' \rangle \right| \\
    &\leq \rho^{t-1}\sqrt{n}\sum_{i \in [n]} |\varphi(i)||\varphi'(i)| \sum_{k=1}^s |\nu_k||\psi_k(i)|\\
    &\leq \rho^{t} d \sqrt{n} \sum_{i\in [n]}|\varphi(i)||\varphi'(i)| \sqrt{s}\sqrt{\sum_{k=1}^s \nu_k^2 \psi_k{(i)}^2}\\
    &\leq \rho^t a \sqrt{n} \sqrt{s}\frac{dL^2}{\sqrt{n}} \sum_i{|\varphi(i)||\varphi'(i)|}\\
    &\leq rd^2 L^2 \rho^t,
\end{align*}
where we extensively used the Cauchy-Schwarz inequality, as well as the bound \( \rho^{-1} \leq d \) from~\eqref{eq:bound_rho_below}.
\end{proof}

\paragraph{Entrywise bounds for \( K^t \)} For a more precise estimation of entrywise bounds, we define the scale-invariant \emph{delocalization} parameter
\[ \Psi = \frac{d L^2}{\rho}. \]
Using the same proof technique as in~\eqref{eq:bound_L_below}, as well as~\eqref{eq:bound_rho_below}, we have
\[ 1 \leq \Psi \leq d^2 L^2 \]
for any \( i, j \in [n] \). Recall that, as shown in the proof of Lemma~\ref{lem:scalar_Kt}, for all \( i \in [n] \)
\[ {(K^2)}_{ii} \leq \frac{d^2 L^4}n = \frac{\Psi^2}n \rho^2. \]
Now, for \( t \geq 0 \) and \( i, j \in [n] \),
\begin{align*}
    {(K^t)}_{ij} &= \sum_{k \in [s]}\nu_k^t \psi_k(i)\psi_k(j) \\
    &\leq \rho^{t-2} \sum_{k} \nu_k^2 \left| \psi_k(i) \right|\left|\psi_k(j) \right| \\
    &\leq \rho^{t-2} \sqrt{{(K^2)}_{ii} {(K^2)}_{jj}},
\end{align*}
where we again used the Cauchy-Schwarz inequality at the last line. This yields
\begin{equation}\label{eq:elementwise_Kt}
    {(K^t)}_{ij} \leq \frac{\Psi^2}n \rho^t
\end{equation}
for any \( t \geq 1 \) and \( i, j \in [n] \).

\paragraph{The covariance matrices} We now study the covariance matrices $\Gamma_U^{(t)}$ and $\Gamma_V^{(t)}$ defined in \eqref{eq:def_gamma}. Our aim is to prove the following lemma:
\begin{lemma}
    For all \( t \geq 1 \); the matrix \( \Gamma_U^{(t)} \) (resp. \( \Gamma_V^{(t)} \)) is a positive definite matrix, with all its eigenvalues greater than 1 (resp. \( c^{-1} \)) and such that
    \[ 1 \leq \norm{\Gamma_U^{(t)}} \leq \frac{r^2 d^3 L^2}{1 - \tau} \quand c^{-1} \leq \norm{\Gamma_V^{(t)}} \leq \frac{r^2 d^2 L^2}{1 - \tau}.\]
\end{lemma}

\begin{proof}
    We first prove the bounds for \( \Gamma_V^{(t)} \). Let \( C^{(s)} \) be the \( r_0 \times r_0 \) matrix (where $r_0$ is defined as in Theorem \ref{th:main_summary}) with
    \[ C^{(s)}_{ij} = \frac{\langle K\ind, K^s\varphi^{i, j} \rangle}{{(\mu_i \mu_j)}^{s+1}}. \]
    Then for every \( w\in \dR^{r_0} \) we have
    \begin{align*}
         w^*C^{(s)}w &= \sum_{i, j \in[r_0]}\frac{w_i w_j}{{(\mu_i\mu_j)}^{s+1}}\sum_{x\in [n]}[K^{s+1}\ind](x)\varphi_i(x)\varphi_j(x) \\
         &= \sum_{x \in [n]} [K^{s+1}\ind](x) {\left(\sum_{i \in [r_0]} \frac{w_i \varphi_i(x)}{\mu_i^{s+1/2}}\right)}^2 \\
         &\geq 0,
    \end{align*}
    hence every matrix \( C^{(s)} \) is positive semi-definite. Further, we have
    \[ C^{(0)} = D^{-1}\Phi^*\diag(K\ind)\Phi D^{-1}, \]
    where \( \Phi \) is the \( n \times r \) matrix whose columns are the \( \varphi_i \). Using \( \mu_i \leq 1 \) for any \( i \in [r_0] \), the eigenvalues of \( C^{(0)} \) are all greater than \( \min_{x}[K\ind](x) \geq c^{-1} \) by our initial assumptions. This settles the positive definite property, as well as the minimum eigenvalue of \( \Gamma_V^{(t)} \).

    Now, applying Lemma~\ref{lem:scalar_Kt} to \( \varphi_i \) and \( \varphi_j \), for all \( i, j\in [r_0] \) one has
    \begin{align*}
        \Gamma^{(t)}_{V, ij} &\leq \sum_{t=0}^{t} \frac{r d^2 L^2 \rho^{s+1}}{{(\mu_i\mu_j)}^{s+1}} \\
        &\leq rd^2L^2\, \sum_{s = 0}^{\infty}{\left(\frac{\rho}{\mu_i\mu_j}\right)}^s.
    \end{align*}
    By definition of $\tau$, the summand above is less than \( \tau^{s} \), whose sum converges since \( \tau < 1 \). As a result,
    \[ \norm{\Gamma_V^{(t)}}_\infty \leq \frac{r d^2 L^2}{1 - \tau}, \]
    and the classic bound \( \norm{\Gamma_V^{(t)}} \leq r_0 \norm{\Gamma_V^{(t)}}_\infty \) implies the upper bound.

    The proof for \( \Gamma_U^{(t)} \) is very similar; the upper bound simply ensues from the fact that \( d_x \leq d \) for any \( x\in [n] \). For the lower bound, if we let as above
    \[ C_{ij}'^{(s)} = \frac{\langle P\ind, K^s\varphi^{i, j} \rangle}{{(\mu_i \mu_j)}^s}, \]
    then
    \[ C'^{(0)} = \Phi^* \diag(P\ind) \Phi, \]
    and the minimum of \( P\ind \) is at least \( 1 \). This implies that the eigenvalues of \( C'^{(0)} \) are larger than one, and we conclude as before.
\end{proof}

\section{Local study of \texorpdfstring{\( G \)}{G}}
It is a well-known fact (see for example~\cite{bordenave_nonbacktracking_2018}) that when the mean degree is low enough (\( d = n^{o(1)} \)), the graph \( G \) is locally tree-like --- that is, vertex neighbourhoods behave almost like random trees. The goal of this section is to establish rigorously this result, as well as provide bounds on neighbourhood sizes.

\subsection{Setting and definitions}
\paragraph{Labeled rooted graphs}

A labeled rooted graph is a triplet \( g_* = (g, o, \iota) \) consisting of a graph \( g = (V, E) \), a root \( o \in V \), and a mark function \( \iota: V \to \dN \) with finite support. We shall denote by \( \mathcal G_* \) the set of labeled rooted graphs with \( V = \dN \), and will often write \( g_* = (g, o) \) for an element of \( \mathcal G_* \), dropping the mark function. Notions of subgraphs, induced subgraphs and distance extend naturally from regular graphs to this setting.

\paragraph{Labeling trees and graphs}

We recall that \( G \) is the inhomogeneous random graph defined earlier. For each vertex \( x \in V \), we can define the associated element of \( \mathcal G_* \) as follows: the root is set to \( x \), each vertex \( y \in [n] \) is given a mark \( \iota(y) = y \), and we let \( \iota(z) = 0 \) for all \( z \in \dN \setminus [n] \). The resulting triple \( (G, x, \iota) \) is a random element of \( \cG_* \).

\medskip

Now, let \( o \in [n] \); we define the inhomogeneous random tree as follows: first, the root is given a mark \( \iota(o) = o \). Then, for each vertex \( x \) already labeled, we draw the number of children of \( x \) according to \( \Poi(d_{\iota(x)}) \), where we recall that
\[ d_{\iota(x)} = \sum_{j}P_{\iota(x), j} \leq d. \]
Each child \( y \) of \( x \) receives a label drawn independently at random from the distribution
\begin{equation}\label{eq:markov_field}
  \pi_{\iota(x)} = \left( \frac{P_{\iota(x), 1}}{d_{\iota(x)}}, \dots, \frac{P_{\iota(x), n}}{d_{\iota(x)}} \right),
\end{equation}
which sums to 1 by definition. The resulting tree is a random element of \( \cG_* \), denoted by \( (T, o) \).

\subsection{Growth properties of trees and graphs}\label{subsec:neighbourhoods}

A number of growth properties for neighbourhoods in \( T \) and \( G \) are needed to ensure the successful couplings below. By definition of $d$, \( G \) (resp. \( (T, o) \)) is dominated by an Erd\H{o}s-Rényi graph \( \cG(n, d/n) \) (resp.\ a Galton-Watson tree with offspring distribution \( \Poi(d) \)); we are thus able to direcly lift properties from~\cite{bordenave_nonbacktracking_2018}, Sections 8 and 9.

\begin{lemma}
  Let \( v \) be an arbitrary vertex in \( G \); then, there exist absolute constants \( c_0, c_1 > 0 \) such that for every \( s > 0 \), we have
  \begin{equation}\label{eq:concentration_graph}
    \Pb*{\forall t \geq 1, \ \left|\partial (G, v)_t\right| \leq sd^t} \geq 1 - c_0e^{-c_1s}.
  \end{equation}
  The same result holds when replacing \( (G, v) \) with the tree \( (T, o) \) defined above.
\end{lemma}

Taking \( s = c_1^{-1}\log(c_0 n^2) \) in the above inequality, one gets
\begin{equation}
  \Pb*{\forall t \geq 1, \  \forall v \in V, \ \left|\partial (G, v)_t\right| \leq c_3\log(n)d^t} \geq 1 - \frac{1}{n},\
\end{equation}
for any \( n \geq 3 \). Summing these inequalities for \( 1 \leq t \leq \ell \) yields a similar bound for the whole ball: with probability at least \( 1 - \frac1n \), we have
\begin{equation}\label{eq:neighbourhood_size}
  |{(G, v)}_t| \leq c_4\log(n)d^{t}
\end{equation}
for all \( v \in V \) and \( t \geq 1 \). In particular, this implies the following useful bound: for any \( v\in V \),
\[ \deg(v) \leq c_4 d\log(n). \]
Another consequence of~\eqref{eq:concentration_graph} is the following useful lemma:
\begin{lemma}\label{lem:bounded_neighbourhood_expectation}
  For every \( p \geq 2 \), there is a constant \( c_p \) such that
  \begin{equation}
    \E*{\max_{v\in V} \sup_{t \geq 1} {\left( \frac{\left|\partial (G, v)_t\right|}{d^t} \right)}^p\,} \leq c_p \log{(n)}^p
  \end{equation}
\end{lemma}
\medskip
Similarly to the proof of~\eqref{eq:neighbourhood_size}, we have
\[ \max_{v\in V} |{(G, v)}_t|^p  \leq d^{tp}t^p \max_{x \in V}\sup_{s \leq t} \frac{\left|\partial (G, v)_t\right|^p}{d^{sp}}, \]
which yields
\begin{equation}\label{eq:neighbourhood_expectation}
  \E*{\max_{v\in V}|{(G, v)}_t|^p} \leq c_p t^p \log{(n)}^p d^{tp}
\end{equation}

An important note is that the above results apply to any collection of \( n \) random variables satisfying an inequality like~\eqref{eq:concentration_graph}; in particular, it also applies to an i.i.d collection of inhomogeneous random trees of size \( n \).

\subsection{Local tree-like structure}

We first check that the random graph \( G \) is tree-like. We say that a graph \( g \) is \emph{\( \ell \)-tangle-free} if there is at most one cycle in the \( \ell \)-neighbourhood of every vertex in the graph. As mentioned before, the random graph $G$ is dominated by an Erd\H{o}s-Rényi graph \( \cG(n, d/n) \); we can therefore lift the desired properties from \cite{bordenave_nonbacktracking_2018}.

\begin{lemma}\label{lem:graph_cycles}
  Let \( \ell \leq n \) be any integer parameter.
  \begin{enumerate}
    \item the random graph \( G \) is \( \ell \)-tangle-free with probability at least \( 1 - c a^2 d^{4\ell}/n \).
    \item the probability that a given vertex \( v \) has a cycle in its \( \ell \)-neighbourhood is at most \( c a d^{2\ell}/n \).
  \end{enumerate}
\end{lemma}

We shall assume in the following that the \( 2\ell \)-tangle-free property happens with probability at least \( 1 - cn^{-\epsilon} \) for some \( \epsilon > 0 \), which happens whenever
\begin{equation}\label{eq:condition_ell_1}
  \ell \leq \frac{1 - \epsilon}{10} \log_d(n) \leq c_3\log(n).
\end{equation}

We now gather all the result of the current section into one proposition, for ease of reading. The bound \( \ell \leq c\log(n) \) assumed above is used to simplify the inequalities below.

\begin{proposition}\label{prop:local_summary}
  Let \( G \) be an inhomogeneous random graph, and \( {(T_x, x)}_{x\in [n]} \) a family of random trees as defined above. Let \( \ell \) be small enough so that~\eqref{eq:condition_ell_1} holds. Then there exists an event \( \cE \) with probability at least \( 1 - \frac1{\log(n)} \), under which:
  \begin{enumerate}
    \item the graph \( G \) is \( 2\ell \)-tangle-free,
    \item for all \( v \in G \), \( t \leq 2\ell \), we have
          \begin{equation}
            |{(G, x)}_t| \leq c\log(n) d^t,
          \end{equation}
    \item for any \( t\leq 2\ell \), the number of vertices in \( G \) whose \( t \)-neighbourhood contains a cycle is at most \( c \log{(n)}^2 d^{t+1} \)
  \end{enumerate}
  Furthermore, for any \( t \leq 2\ell \) and \( p \geq 1 \), we have
  \begin{equation}
    \E*{\max_{v\in V}|{(G, v)}_t|^p}^{\frac1p} \leq c\log{(n)}^2 d^t,
  \end{equation}
  and the same holds for the family \( {(T_x, x)}_{x\in [n]} \).
\end{proposition}

\subsection{Coupling between rooted graphs and trees}
We now turn onto the main argument of this proof: we bound the variation distance between the neighbourhoods of \( (G, x) \) and \( (T, x) \) up to size \( \ell \).

First, recall some definitions: if \( \dP_1, \dP_2 \) are two probability measures on the space \( (\Omega, \cF) \), their \emph{total variation distance} is defined as
\begin{equation*}
  \dtv(\dP_1, \dP_2) = \sup_{\cA \in \cF}|\dP_1(\cA) - \dP_2(\cA)|.
\end{equation*}
The following two characterizations of the total variation distance shall be useful: first, whenever \( \Omega \) is countable, we have
\begin{equation}\label{eq:dtv_sum}
  \dtv(\dP_1, \dP_2) = \frac12\norm*{\dP_1 - \dP_2}_1 = \frac12 \sum_{\omega \in \Omega}|\dP_1(\omega) - \dP_2(\omega)|.
\end{equation}
Additionally,
\begin{equation}\label{eq:dtv_coupling}
  \dtv(\dP_1, \dP_2) = \min_{\dP \in \pi(X_1, X_2)} \dP(X_1 \neq X_2),
\end{equation}
where \( \pi(X_1, X_2) \) denotes the set of all \emph{couplings} between \( \dP_1 \) and \( \dP_2 \), i.e.\ probability measures on \( (\Omega^2, \cF \otimes \cF) \) such that the marginal distributions are \( \dP_1 \) and \( \dP_2 \).

Denoting by \( \cL(X) \) the probability distribution of a variable \( X \), the aim of this section is to prove the following:
\begin{proposition}\label{prop:dtv}
  Let \( \ell \leq c_0\log(n) \) for some constant \( c_0 > 0 \). Then, for every vertex \( v \in V \),
  \begin{equation}
    \dtv(\cL({(G, v)}_{\ell}), \cL({(T, v)}_{\ell})) \leq \frac{c\,\log{(n)}^2 d^{2\ell+2}}{n}.
  \end{equation}
\end{proposition}

\subsubsection{A total variation distance lemma for sampling processes}
For an integer \( n \), denote by \( \cS(n) \) the set of all multisets with elements in \( [n] \), and by \( \cP(n) \subset \cS(n) \) the powerset of \( [n] \). Let \( p_1, \dots, p_n \in [0, 1/2] \), with \( \sum p_i = \lambda \) and \( \sum p_i^2 = \alpha \), and consider the two probability laws on \( \cS(n) \):
\begin{itemize}
  \item \( \dP_1 \): each element \( i \) of \( [n] \) is picked with probability \( p_i \),
  \item \( \dP_2 \): the size of the multiset \( S \) is drawn according to a \( \Poi(\lambda) \) distribution, and each element of \( S \) has an i.i.d label with distribution \( (p_1/\lambda, \dots, p_n/\lambda) \).
\end{itemize}
Note that \( \dP_1 \) is actually supported on \( \cP(n) \).
\begin{proposition}\label{prop:dtv_sampling_poisson}
  Let \( \dP_1, \dP_2 \) be defined as above. Then
  \[ \dtv(\dP_1, \dP_2) \leq \alpha + \frac{e^{2\alpha} - 1}2. \]
\end{proposition}

\begin{proof}

  Using characterization~\eqref{eq:dtv_sum}, we have
  \begin{equation}\label{eq:dtv_sampling_decomp}
    2\dtv(\dP_1, \dP_2) = \sum_{S \in \cP(n)}\left| \dP_1(S) - \dP_2(S) \right| + \dP_2(S \notin \cP(n)).
  \end{equation}
  We shall treat those two terms separately. First, notice that for \( S \in \cP(n) \), we have
  \begin{align}
    \dP_1(S) & = \prod_{i\in S}p_i \prod_{i \notin S}(1 - p_i) \label{eq:sampling_prob_1}                      \\
    \dP_2(S) & = \frac{e^{-\lambda}\lambda^{|S|}}{|S|!}\times |S| !\prod_{i\in S}\frac{p_i}{\lambda} \nonumber \\
             & = e^{-\lambda}\prod_{i\in S} p_i, \label{eq:sampling_prob_2}
  \end{align}
  and thus by summing over all sets \( S \),
  \[ \dP_2(S \in \cP(n)) = e^{-\lambda}\prod_{i=1}^n(1 + p_i). \]
  Using the classical inequality \( \log(1+x) \geq x - x^2/2 \), we can bound the second member of~\eqref{eq:dtv_sampling_decomp} as follows:
  \begin{align*}
    \dP_2(S \notin \cP(n)) & = 1 - e^{-\lambda}\prod_{i=1}^n(1 + p_i)    \\
                           & \leq 1 - e^{-\lambda}e^{\lambda - \alpha/2} \\
                           & \leq \alpha/2.
  \end{align*}

  On the other hand, using again~\eqref{eq:sampling_prob_1} and~\eqref{eq:sampling_prob_2}, the first term reduces to
  \begin{align*}
     & \sum_{S \in \cP(n)}\left| \dP_1(S) - \dP_2(S) \right| = \sum_{S \in \cP(n)} \prod_{i \in S}p_i \left|\prod_{i\notin S}(1-p_i) - e^{-\lambda} \right|                                  \\
     & \leq \sum_{S \in \cP(n)} \prod_{i \in S}p_i \left( \left| e^{-\lambda} - \prod_{i=1}^n (1 - p_i) \right| + \left| \prod_{i\notin S}(1-p_i) - \prod_{i=1}^n (1 - p_i) \right| \right).
  \end{align*}

  Both absolute values above can be removed since the expressions inside are nonnegative; further, for \( 0 \leq p \leq 1/2 \), we have \( \log(1-x) \geq -x - x^2 \). Combining all those estimates, we find
  \begin{align*}
     & \sum_{S \in \cP(n)}\left| \dP_1(S) - \dP_2(S) \right|                                                                                                           \\& \leq e^{-\lambda} (1 - e^{-\alpha})\sum_{S \in \cP(n)} \prod_{i \in S}p_i + \prod_{i = 1}^n(1 - p_i)\sum_{S \in \cP(n)} \prod_{i \in S}p_i\left(\prod_{i\in S}\frac1{1-p_i} - 1 \right) \\
     & \leq \alpha e^{-\lambda}\prod_{i=1}^n(1+p_i) + e^{-\lambda}\left( \prod_{i=1}^n\left(1 + \frac{p_i}{1-p_i} \right) - \prod_{i=1}^n\left(1 + p_i \right) \right) \\
     & \leq \alpha + e^{-\lambda}\exp\left(\sum_{i=1}^n \frac{p_i}{1-p_i}\right) - e^{- \frac{\alpha}2},
  \end{align*}
  where we again used the logarithm inequalities extensively. Finally, for \( 0 \leq p \leq 1/2 \), we have \( p/(1-p) \leq p + 2p^2 \), which allows us to finish the computation:
  \begin{equation}\label{eq:dtv_twice_sum}
    \sum_{S \in \cP(n)}\left| \dP_1(S) - \dP_2(S) \right| \leq \frac32 \alpha + e^{2\alpha} - 1.
  \end{equation}
  Combining~\eqref{eq:dtv_twice_sum} with~\eqref{eq:dtv_sampling_decomp} easily implies the lemma.

\end{proof}

We introduce now a family of probability laws on \( \cS(n) \); for a subset \( S \subseteq [n] \), let \( \dP_S \) be the measure corresponding to picking each element \( i \) of \( S \) with probability \( p_i \).

The variation distance between those laws and \( \dP_1 = \dP_{[n]} \) is then easier to bound:
\begin{lemma}\label{lem:dtv_sampling_restricted}
  For any \( S \subseteq [n] \), we have:
  \[ \dtv(\dP_1, \dP_S) \leq \sum_{i\notin S} p_i.  \]
\end{lemma}

\begin{proof}
  Consider the following coupling: we take a realization \( X \) of \( \dP_1 \), and set \( Y = X \cap S \). Then, \( Y \sim \dP_S \), and we find
  \[ \dP(X \neq Y) = \dP_1(X \cap S^c \neq \emptyset) \leq \E*{ |X \cap S^c | } = \sum_{i \notin S}p_i\]
  This ends the proof, since~\eqref{eq:dtv_coupling} ensures that \( \dtv(\dP_1, \dP_S) \leq \dP(X \neq Y) \).
\end{proof}

\subsubsection{Proof of Proposition~\ref{prop:dtv}} Gathering all the previous results, we are now ready to prove Proposition~\ref{prop:dtv}:
\begin{proof}

  Define the classical breadth-first exploration process on the neighbourhood of a vertex \( v \) as follows : start with \( A_0 = \Set{v} \) and at stage \( t \geq 0 \), if \( A_t \) is not empty, take a vertex \( v_t \in A_t \) at minimal distance from \( v \), reveal its neighbours \( N_{t} \) in \( V \setminus A_t \), and update \( A_{t+1} = (A_t \cup N_t) \setminus \Set{v_t} \). We denote by \( {(\cF_t)}_{t \geq 0} \) the filtration generated by the \( {(A_t)}_{t\geq 0} \), and by \( D_t = \bigcup_{s \leq t}A_s \) the set of vertices already visited at time \( t \), and \( \tau \) the first time at which all vertices in \( {(G, v)}_\ell \) have been revealed.

  We perform the same exploration process in parallel on \( (T, v) \), which corresponds to a breadth-first search of the tree. At step \( t \), we denote by \( \dP_t \) the distribution of $N_t$ given \( \cF_t \), and \( \dQ_t \) the distribution of the offspring of \( v_t \) in \( T \) (no conditioning is needed there).

  Let \( E_\ell \) denote the event that \( {(G, v)}_\ell \) is a tree and contains no more than \( c_1\log(n)d^\ell \) vertices; from~\eqref{eq:neighbourhood_size} and Lemma~\ref{lem:graph_cycles}, we can choose \( c_1 \) such that \( E_\ell \) has probability at least \( 1 - c_2 d^{2\ell+1}/n \) for some absolute constant \( c_2 \). By iteration, it suffices to show that if \( E_\ell \) holds, there exists a constant \( c_3 > 0 \) such that
  \begin{equation}\label{eq:dtv_vertex}
    \dtv(\dP_t, \dQ_t) \leq \frac{c_3 \log(n) d^{\ell+2}}n \quad \text{for all} \quad t \leq \tau.
  \end{equation}

  Given \( \cF_t \), the probability measure \( \dP_t \) is as follows: each element \( i \) of \( V \setminus A_t \) is selected with probability \( p_i = P_{v_t i} \). Let \( \dP'_t \) denote the same probability measure, but where the selection is made over all of \( V \). Using Lemma~\ref{lem:dtv_sampling_restricted}, we first find that
  \[ \dtv(\dP_t, \dP'_t) \leq \sum_{i \in A_t} P_{v_t i} \leq c_1 \log(n) d^\ell\cdot \frac{d}{n}. \]
  On the other hand, Proposition~\ref{prop:dtv_sampling_poisson} yields
  \[ \dtv(\dP'_t, \dQ_t) \leq c_4 \sum_{i=1}^n P_{v_t i}^2 \leq c_5 \,\frac{d^2}{n}. \]
  Equation~\eqref{eq:dtv_vertex} then results from a straightforward application of the triangle inequality.
\end{proof}

\section{Near eigenvectors of \texorpdfstring{\( G \)}{G}}

\subsection{Functionals on \texorpdfstring{\( (T, o) \)}{(T, o)}}\label{subsec:functionals}
\subsubsection{Vertex functionals on trees}

Similarly to~\cite{bordenave_nonbacktracking_2018}, quantities of interest in the study of \( B \) will be tied to functionals on the random inhomogeneous tree defined above. Define a functional \( f_{\varphi, t} \) on the set of labeled rooted trees \( \cT_* \subset \cG_* \) by
\[ f_{\varphi, t}(T, o) = \sum_{x_t \in \partial {(T, o)}_t}{W_{\iota(o), \iota(x_1)} \dots W_{\iota(x_{t-1}), \iota(x_t)}\varphi(\iota(x_{t}))}, \]
where \( (o, x_1, \dots, x_t) \) is the unique path of length \( t \) between \( o \) and \( x_t \).
Then the following proposition holds:
\begin{proposition}\label{prop:functionals_tree}
  Let \( t \geq 0 \) be an integer. For any \( i, j \in [r] \), the following identities are true:
  \begin{align}
     & \E*{f_{\varphi_i, t}(T, x)} = \mu_i^t\, \varphi_i(x) \label{eq:functional_nonzero_eigen},                                                                                 \\
     & \E*{f_{\varphi_i, t}(T, x)f_{\varphi_j, t}(T, x)} = {(\mu_i\mu_j)}^t\sum_{s = 0}^t{\frac{[K^{s}\varphi^{i, j}](x)}{{(\mu_i\mu_j)}^s}} \label{eq:functional_nonzero_corr}, \\
     & \E*{{\left(f_{\varphi_i, t+1}(T, x) - \mu_i f_{\varphi_i, t}(T, x)\right)}^2} = [K^{t+1}\varphi^{i, i}](x). \label{eq:functional_square_increments}
  \end{align}
  where we recall that \( \varphi^{i, j} = \varphi_i \odot \varphi_j \).
\end{proposition}

\subsubsection{Adapting functionals to non-backtracking paths}

The matrix \( B \) considered here acts on (directed) edges, whereas the functionals considered so far are defined on vertices. Consequently, we define the following transformation: for a function \( f: \cG_* \to \dR \), and a random vector \( w \in \dR^{V} \) with expected value \( \bar w \), let
\[ \vec \partial_w f(g, o) = \sum_{e: e_2 = o} w_{e_1} f(g_e, o), \]
where \( g_e \) denotes the graph \( g \) with the edge \( {e_1, e_2} \) removed.

The expectations from Proposition~\ref{prop:functionals_tree} are then adapted as follows:

\begin{proposition}\label{prop:edge_functionals_tree}
  Let \( t \geq 0 \) be an integer. For any \( i, j \in [r] \), and \( \phi \in \ker(P) \), the following identities are true:
  \begin{align}
     & \E*{\vec \partial_w f_{\varphi_i, t}(T_x, x)} = [P\bar w](x) \cdot \E*{f_{\varphi_i, t}(T_x, x)},\label{eq:edge_nonzero_eigen}                                                \\
     & \E*{\vec \partial_w(f_{\varphi_i, t}\cdot f_{\varphi_j, t})(T_x, x)} = [P\bar w](x) \cdot \E*{f_{\varphi_i, t}(T_x, x)f_{\varphi_j, t}(T_x, x)} \label{eq:edge_nonzero_corr}, \\
     & \notag \E*{\vec \partial_w [{(f_{\varphi_i, t+1} - \mu_i f_{\varphi_i, t})}^2](T_x, x)}                                                                                       \\ &\qquad \qquad= [P\bar w](x) \cdot \E*{{\left(f_{\varphi_i, t+1}(T_x, x) - \mu_i f_{\varphi_i, t}(T_x, x)\right)}^2}. \label{eq:edge_square_increments}
  \end{align}
\end{proposition}

The proof for those results makes use of properties specific to moments of Poisson random variables; as with the preceding results, it is deferred to a later section.

\subsection{Spatial averaging of graph functionals}

In this section, we leverage the coupling obtained above to provide bounds on quantities of the form \( \frac1n \sum_{x \in V} f(G, x) \), for local functions \( f \). The tools and results used in this section are essentially identical to those in~\cite{bordenave_nonbacktracking_2018}, with a few improvements and clarifications added when necessary.

We begin with a result that encodes the fact that the \( t \)-neighbourhoods in \( G \) are approximately independent. We say that a function \( f \) from \( \cG_* \) to \( \dR \) is \( t \)-local if \( f(g, o) \) is only function of \( {(g, o)}_t \).

\begin{proposition}
  Let \( t \leq c_0\log(n) \) for some constant \( c_0 > 0 \). Let \( f, \psi: \cG_* \to \dR \) be two \( t \)-local functions such that \( |f(g, o)| \leq \psi(g, o) \) for all \( (g, o) \in \cG_* \) and \( \psi \) is non decreasing by the addition of edges. Then
  \[ \Var\left( \sum_{o \in V} f(G, o) \right) \leq c \log{(n)}^4 n d^{2t} \cdot \sqrt{\E*{\max_{o \in V}\psi{(G, o)}^4}}. \]
\end{proposition}

\begin{proof}
  For \( x\in V \), denote by \( E_x \) the set \( \Set*{ \{u, x\} \in E \given u \leq x  } \); the vector \( (E_1, \dots, E_n) \) is an independent vector, and we have
  \[ Y := \sum_{v \in V} f(G, v) = F(E_1, \dots, E_n). \]
  for some measurable function \( F \).\\
  Define now \( G_x \) the graph with vertex set \( V \) and edge set \( \bigcup_{y \neq x} E_y \), and set
  \[  Y_x = \sum_{v \in V} f(G_x, v).  \]
  The random variable \( Y_x \) is \( \bigcup_{y \neq x} E_y \)-measurable, so the Efron-Stein inequality applies:
  \[ \Var(Y) \leq \sum_{x \in [n]}\E*{{(Y - Y_x)}^2}. \]
  For a given \( x \in V \), the difference \( f(G, o) - f(G_x, o) \) is always zero except if \( x \in {(G, o)}_t \), due to the locality property; consequently,
  \begin{align*}
    |Y - Y_x| & \leq \sum_{o \in V}|f(G, o) - f(G_x, o)|                                       \\
              & \leq \sum_{o \in {(G, x)}_t} \psi(G, o) + \psi(G_x, o)                         \\
              & \leq 2 \max_{x\in [n]} \left| {(G, x)}_t \right|\cdot \max_{o\in V}\psi(G, o),
  \end{align*}
  where we used the non-decreasing property of \( \psi \) in the last line. By the Cauchy-Schwarz inequality and equation~\eqref{eq:neighbourhood_expectation}, we can write
  \begin{align*}
    \E*{{(Y - Y_x)}^2} & \leq 4\sqrt{\E*{\left| \max_{x\in [n]} {(G, x)}_t \right|^4}}\cdot \sqrt{\E*{\max_{o \in V}\psi{(G, o)}^4}} \\
                       & \leq c_1 t^2 \log{(n)}^2 d^{2t} \cdot \sqrt{\E*{\max_{o \in V}\psi{(G, o)}^4}}.
  \end{align*}
  Using that \( t \leq c_0\log(n) \), and the linearity of expectation, yields the desired bound.
\end{proof}

We now use our previous coupling results to provide a concentration bound between a functional on graphs and its expectation on trees:

\begin{proposition}\label{prop:functionals_concentration}
  Let \( t \in \dN \) and \( f, \psi: \cG_* \to \dR \) be as in the previous proposition. Then, with probability at least \( 1 - \frac{1}{r^2 \log{(n)}^2} \), the following inequality holds:
  \[ \left| \sum_{v \in V}f(G, v) - \E*{\sum_{x\in [n]}f(T_x, x)}\right|  \leq c\, r \log{(n)}^3 d^{t+1} \sqrt{n} \norm*{\psi}_\star, \]
  where \( \norm*{\psi}_\star \) is defined as
  \[ \norm*{\psi}_\star = {\left( \E*{\max_{v \in V}\psi{(G, v)}^4} \right)}^{\frac14} \, \vee \, {\left(\max_{x \in [n]}\E*{\psi{(T_x, x)}^2} \right)}^{\frac12}. \]
\end{proposition}

\begin{proof}
  Using the Chebyshev inequality and the variance bound from the preceding proposition, we have with probability at least \( 1 - \frac1{r^2\log{(n)}^2} \)
  \[ \left| \sum_{v \in V}f(G, v) - \E*{\sum_{v \in V}f(G, v)} \right| \leq c_1\,r\log{(n)}^3 d^t \sqrt{n} \norm*{\psi}_\star. \]
  It then remains to bound the difference between the expectation term and its counterpart on trees. For \( x \in V \), let \( \cE_x \) denote the event that the coupling bewteen \( {(G, x)}_t \) and \( {(T_x, x)}_t \) fails; by the locality property, \( f(G, x) = f(T_x, x) \) on \( \cE_x \). Therefore, using the Cauchy-Schwarz inequality,
  \begin{align*}
     & \left| \sum_{x \in [n]} \E*{f(G, x) - f(T_x, x)} \right| \leq \sum_{x \in [n]} \E*{|f(G, x)|1_{\cE_x} + |f(T_x, x)|1_{\cE_x}}             \\
     & \leq \sum_{x \in [n]} \sqrt{\Pb*{\cE_x}}\left( \sqrt{\E*{\psi{(G, x)}^2}} + \sqrt{\E*{\psi{(T_x, x)}^2}} \right)                          \\
     & \leq \sqrt{\frac{c_2\log{(n)}^2d^{2t+2}}n}\cdot \sum_{x \in [n]}\left(\E*{\psi{(G, x)}^4}^{\frac14} + \sqrt{\E*{\psi{(T_x, x)}^2}}\right) \\
     & \leq c_3 \log(n) a d^{t+1} \sqrt{n} \,\norm*{\psi}_\star.
  \end{align*}

  It is then straightforward to check that both obtained bounds are less than the RHS in the proposition, upon adjusting \( c \).
\end{proof}

\subsection{Structure of near eigenvectors}

In the following, the aim is to obtain bounds on the norms and scalar product of the near eigenvectors $u_i$ and $v_i$ defined in \eqref{eq:def_ui_vi}. The main result of this section is as follows:

\begin{proposition}\label{prop:pseudo_eigenvectors}
  Let \( \ell \) be small enough so that~\eqref{eq:condition_ell_1} holds. On an event with probability \( 1 - c_1/\log(n) \), the following inequalities hold for all \( i, j \in [r] \), \( t \leq 2\ell \) and some absolute constant \( c > 0 \):
  \begin{align}
     & \left|\langle B^t \chi_i, \chi_j \rangle - \mu_i^t \langle \varphi_i, D_P\varphi_j \rangle \right|                                 \leq \frac{c\,r b^2 d^2 \log{(n)}^6 d^{2t}L^t}{\sqrt{n}}, \label{eq:u_phi_dotp}                                \\
     & \left| \langle B^t\chi_i, D_W\check\chi_j \rangle - \mu_i^{t+1}\delta_{ij} \right|                                                 \leq \frac{c\, r b^2 d^3 L \log{(n)}^6 d^{2t} L^t}{\sqrt{n}}, \label{eq:u_v_dotp}                              \\
     & \left| \langle B^t \chi_i, B^t \chi_j \rangle - \mu_i^t \mu_j^t\Gamma_{U, ij}^{(t)} \right|                                        \leq  \frac{c\,r b^2 d^2 \log{(n)}^7 d^{3t}L^{2t}}{\sqrt{n}}, \label{eq:u_u_dotp}                              \\
     & \left| \langle {(B^*)}^t D_W \check\chi_i, {(B^*)}^t D_W \check\chi_j \rangle - \mu_i^{t+1}\mu_j^{t+1}\Gamma_{ij}^{(t+1)} \right|  \leq \frac{c\,r b^2 d^2 L^2 \log{(n)}^6 d^{3t}L^{2t}}{\sqrt{n}}, \label{eq: v_v_dotp}                          \\
     & \norm*{B^{t+1} \chi_i - \mu_i B^t \chi_i}^2                                                                                        \leq r d^3 L^2 \rho^{t+1} + \frac{c r b^2 d^3 \log{(n)}^7 d^{3t}L^{2t}}{\sqrt{n}}. \label{eq:u_eigen_equation}
  \end{align}
\end{proposition}

\begin{proof}
  The proof of those inequalities relies on careful applications of Proposition~\ref{prop:functionals_concentration} to previously considered functionals. We aim to prove that each of those inequalities hold with probability \( 1 - c_2/r\log(n) \); we fix in the following an integer \( t \leq 2\ell \) and \( i, j \in [r] \). Let \( \cV_t \) be the set of vertices such that \( {(G, v)}_t \) is not a tree; we place ourselves in the event described in Proposition~\ref{prop:local_summary} and as a consequence
  \[ \cV_t \leq c_3 \log{(n)}^2 d^{t+1}. \]

  \bigskip

  We first prove~\eqref{eq:u_phi_dotp}; let
  \[ f(g, o) =  \ind_{{(g, o)}_t \text{ has no cycles}}\,\varphi_j(o)\vec\partial_\ind f_{\varphi_i, t}(g, o). \]
  The function \( f \) is clearly \( t \)-local, and
  \begin{align*}
    \left |f(g, o) \right| & \leq \norm*{\varphi_i}_\infty \norm*{\varphi_j}_\infty \deg(o) \left| \partial{(g, o)}_t \right| L^t \\
                           & \leq \frac{b^2}n \deg(o) \left|{(g, o)}_t \right| L^t := \psi(g, o).
  \end{align*}
  The function \( \psi \) thus defined is non-decreasing by the addition of edges.
  When \( v \notin \cV_t \), we notice that
  \[ f(G, v) = \varphi_j(v) \cdot [T^* B^t \chi_i](v), \]
  hence,
  \[ \left| \langle B^t\chi_i, \chi_j \rangle - \sum_{v \in V}f(G, v) \right| = \left|\sum_{v \in \cV_t} \varphi_i(v) T^* B^t \chi_j\right| \leq 2|\cV_t| \max_v \psi(G, v), \]
  since by the tangle-free property there are at most two paths from \( v \) to any vertex in \( {(G, v)}_t \). Furthermore, using the results in \autoref{subsec:neighbourhoods}, we find that with probability at least \( 1 - 1/n \)
  \[ \max_{v}\psi(G, v) \leq \frac{c_4\,b^2\log{(n)}^2 d^{t+1} L^t}n \quand \norm*{\psi}_\star \leq \frac{c_4\,b^2\log{(n)}^3 d^{t+1} L^t}n. \]
  Finally, a direct computation shows that
  \[ \sum_{x \in [x]} \E*{f(T_x, x)} = \sum_{x \in [n]}{\varphi_j(x)\cdot d_x \mu_i^t \varphi_i(x)} = \mu_i^t\langle \varphi_j, D_P\varphi_i \rangle. \]
  Applying Proposition~\ref{prop:functionals_concentration} to \( f \) and \( \psi \), and using the triangle inequality:
  \begin{align*}
    \left| \langle B^t\chi_i, \chi_j \rangle - \mu_i^t\langle \varphi_j, D_P\varphi_i \rangle \right| & \leq \frac{c_5\,b^2\log{(n)}^4 d^{2t+2}L^t}{n} + \frac{c_6\, r b^2 \log{(n)}^6 d^{2t+2}L^t}{\sqrt{n}} \\
                                                                                                      & \leq \frac{c_7\,r b^2 d^2 \log{(n)}^6 d^{2t}L^t}{\sqrt{n}}.
  \end{align*}

  The proof of the other inequalities is very similar, applying Proposition~\ref{prop:functionals_concentration} to other functionals from \autoref{subsec:functionals}. To avoid clutter, it is deferred to the appendix.
\end{proof}

\section{Proof of Theorem~\ref{th:bl_u_bounds}}

Having shown Proposition~\ref{prop:pseudo_eigenvectors}, all that remains is simply to gather the preceding bounds, and simplify them to get an easy-to-read summary. Bounds~\eqref{eq:Ustar_U}-\eqref{eq:Ustar_V}, as well as~\eqref{eq:norm_Bl}, being straightforward computations, they are deferred to the appendix.

\subsection{A telescopic trick: proof of~\eqref{eq:Bl_U}}

Notice that for for a \( r_0 \times r_0 \) matrix \( M \), we have
\begin{equation}\label{eq:two_norm_equiv}
  \norm{M}  \leq r_0 \max_{i} \norm{M_i}.
\end{equation}
where $M_i$ are the columns (or lines) of $M$.
To apply this inequality, we write
\begin{equation}\label{eq:eigen_telescopic_sum}
  \norm{B^{\ell}u_i - \mu_i^\ell u_i} \leq \sum_{t = 0}^{\ell - 1} \mu_i^{\ell - t - 1}\norm{B^{t+1}u_i - \mu_i B^t u_i},
\end{equation}
and~\eqref{eq:u_eigen_equation} yields
\begin{align*} 
   & \norm{B^{t+1}u_i - \mu_i B^t u_i}^2 \leq \mu_i^{-2\ell}\norm{B^{t + \ell + 1}\chi_i - \mu_i B^{t+\ell}\chi_i}^2                               \\ 
   & \qquad\qquad\leq \mu_i^{-2\ell} \left(r d^3 L^2 \rho^{t+\ell+1} + \frac{c r b^2 d^3 \log{(n)}^7 d^{3(t+\ell)}L^{2(t+\ell)}}{\sqrt{n}}\right).
\end{align*}
Since \( i \leq r_0 \), the bounds \( \mu_i^2 \geq \rho \geq 1/d \) apply, so that
\begin{equation}\label{eq:bound_bu_muu}
  \norm{B^{t+1}u_i - \mu_i B^t u_i}^2 \leq r d^3 L^2\rho^{t+\ell+1}\mu_i^{-2\ell} + \frac{c r b^2 d^3 \log{(n)}^7 d^{3t+5\ell}L^{2(t+\ell)}}{\sqrt{n}}.
\end{equation}
We now use the (very crude) inequality \( \sqrt{x+y} \leq \sqrt{x} + \sqrt{y} \) inside~\eqref{eq:bound_bu_muu}:
\begin{align*}
  \norm{B^{\ell}u_i - \mu_i^\ell u_i} & \leq \sum_{t = 0}^{\ell - 1} \left[\mu_i^{\ell - t - 1}\sqrt{r}d^{3/2}L\rho^{\frac{t+\ell+1}2}\mu_i^{-\ell} + \frac{c_1\, b d^{3/2} \log{(n)}^{7/2} d^{\frac{3t+5\ell}2}L^{t+\ell}}{n^{1/4}} \right] \\
                                      & \leq \sqrt{r}d^{3/2}L \rho^{\ell/2}\sum_{t=0}^{\ell-1}{\left(\frac{\sqrt{\rho}}{\mu_i}\right)}^{t+1} + c_2\,b d^{2}\log{(n)}^{9/2} \frac{{(Ld^4)}^\ell}{n^{1/4}} L^\ell.
\end{align*}
The terms in the sum are all less than 1 since \( i \leq r_0 \), and \( \ell < c_3\log(n) \) implies
\[ \norm{B^{\ell}u_i - \mu_i^\ell u_i} \leq c_3 \sqrt{r}d^{3/2}L\log(n) \rho^{\ell/2} + c_2 b d^2 \log{(n)}^{9/2} \frac{{(aLd^3)}^\ell}{n^{1/4}} L^\ell. \]
The bound \( {(Ld^4)}^\ell \leq n^{1/4} \) holds by definition of \( \ell \), and~\eqref{eq:Bl_U} ensues via~\eqref{eq:two_norm_equiv}.

\subsection{Bounding \texorpdfstring{\( \norm{B^\ell P_{H^\bot}} \)}{|| B\textasciicircum{}l P\_H ||}}

Having established the candidates and error bounds for the upper eigenvalues of \( B^\ell \), it remains to bound the remaining eigenvalues (also called the \emph{bulk}) of the matrix. This is done using a method first employed in~\cite{massoulie_community_2014}, and leveraged again in a similar setting in~\cite{bordenave_nonbacktracking_2018, bordenave_detection_2020}. Our approach will be based on the latter two, adapting the non-backtracking method to the weighted case.

\bigskip

Our first preliminary step is the following lemma:

\begin{lemma}\label{lem:near_orthogonality}
  On an event with probability at least \( 1 - 1/\log(n) \), for any \( t \leq \ell \), any unit vector \( w \in H^\bot \) and \( i \in [r_0] \), one has
  \[ \left|\langle {(B^*)}^t D_W \check\chi_i, w\rangle \right| \leq \sqrt{r}d^{3/2}L^2\rho^{t/2} + \frac{c_4\, b d^{3/2} \log{(n)}^{9/2} d^{2\ell}L^{\ell}}{n^{1/4}}. \]
\end{lemma}
Proving this bound is done through the same telescopic sum trick as above, and is done in the appendix.

\subsubsection{Tangle-free decomposition of \texorpdfstring{\( B^\ell \)}{B\textasciicircum{}l}}

We adapt here the decomposition first used in~\cite{bordenave_nonbacktracking_2018} to our setting. Through the remainder of this section, we shall consider \( B \) as an operator on \( \vec E(V) \) instead of \( \vec E \), setting \( B_{ef} = 0 \) whenever \( e \notin \vec E \) or \( f\notin \vec E \). This yields a matrix with \( B \) as a principal submatrix and zeros everywhere else, thus the non-zero spectrum stays identical.

For \( e, f \in \vec E(V) \), and \( t \geq 0 \), we define \( \Gamma^{k}_{ef} \) the set of non-backtracking paths of length \( k \) from \( e \) to \( f \); further, for an edge \( e \) we define \( X_e \) the indicator variable of \( e\in \vec E \), and \( A_e = X_e W_e \), so that \( A \) is the (weighted) adjacency matrix of \( G \).

We then have that
\[ {(B^k)}_{ef} = \sum_{\gamma \in \Gamma^{k+1}_{ef}}X_e\prod_{s = 1}^{k} A_{\gamma_s\gamma_{s+1}}. \]
Define \( F^k_{ef} \) the set of \( \ell \)-tangle-free paths (i.e.\ the set of paths \( \gamma \) such that the subgraph induced by \( \gamma \) is tangle-free). Then, whenever the graph \( G \) is tangle-free, for all \( k \leq \ell \) the matrix \( B^k \) is equal to \( B^{(k)} \), with
\[ {(B^{(k)})}_{ef} = \sum_{\gamma \in F^{k+1}_{ef}} X_e\prod_{s = 1}^{k} A_{\gamma_s\gamma_{s+1}}. \]
Define now the ``centered'' versions of the weighted and unweighted adjacency matrices \( \underline A \) and \( \underline X \) by
\[ \underline A_{ij} = A_{ij} - Q_{ij} \quand \underline X_{ij} = X_{ij} - P_{ij} \]
for every \( i \neq j \), and its centered non-backtracking counterpart as 
\[ {(\Delta^{(k)})}_{ef} = \sum_{\gamma \in F^{k+1}_{ef}}\underline X_{ij}\prod_{s = 1}^{k} \underline A_{\gamma_s\gamma_{s+1}}, \]
with the convention that the product over an empty set is equal to 1.

\medskip

Recall that for any two sets of real numbers \( (x_i), (y_i) \), we have the following:

\[ \prod_{s = 0}^\ell x_s = \prod_{s = 0}^\ell y_s + \sum_{t = 0}^\ell \prod_{s = 0}^{t-1}y_s (x_t - y_t)\prod_{s = t+1}^\ell x_s. \]
Applying this formula to the above definitions, and separating the case \( t = 0 \) in the sum yields
\begin{multline}\label{eq:tangle_free_decomp}
  B^{(\ell)}_{ef} = \Delta^{(\ell)}_{ef} + \sum_{\gamma \in F^{\ell+1}_{ef}} Q_e\prod_{s = 1}^{\ell} A_{\gamma_s\gamma_{s+1}} \\
  + \sum_{t = 1}^\ell\sum_{\gamma \in F^{\ell+1}_{ef}}\underline X_e\prod_{s = 1}^{t-1} \underline A_{\gamma_s\gamma_{s+1}}Q_{\gamma_t\gamma_{t+1}}\prod_{s = t+1}^{\ell} A_{\gamma_s\gamma_{s+1}}.
\end{multline}

Define now \( F^{\ell+1}_{t, ef} \subset \Gamma^{\ell+1}_{ef} \) the set of non-backtracking \emph{tangled} paths \( \gamma \) such that \( (\gamma_0, \dots \gamma_t) \in F^t_{eg} \), \( (\gamma_{t+1}, \dots, \gamma_{\ell+1})\in F^{\ell-t}_{g'f} \) for some edges \( g, g' \in \vec E(V) \). As an edge case, \( F_{0, ef}^{\ell + 1} \) is the set of tangled paths \( \gamma \) such that \( (\gamma_0, \gamma_1) = e_1 \) and \( (\gamma_1, \dots, \gamma_{\ell+1}) \in F^{\ell}_{g'f} \) for some \( g'\in \vec E(V) \) (note that necessarily \( e_2 = g'_1 \)), and similarly for \( F_{\ell, ef} \). Finally, we introduce the two matrices \( M \) and \( M^{(2)} \) as
\[ M_{ef} = \ind \{e \to f \} Q_e \quand M^{(2)}_{ef} = \ind (e \xrightarrow{2} f)Q_{e_2f_1} \]
for \( e, f \in \vec E(V) \), where \(  e \xrightarrow{2} f \) means that there exists a non-backtracking path of length two between \( e \) and \( f \). Then, equation~\eqref{eq:tangle_free_decomp} can be rewritten as
\begin{equation}
  B^{(\ell)} = \Delta^{(\ell)} + M D_W B^{(\ell-1)} + \sum_{t = 1}^{\ell - 1}\Delta^{(t-1)}M^{(2)}D_W B^{(\ell - t - 1)} + \Delta^{(\ell - 1)}M - \sum_{t = 0}^\ell R_t^{(\ell)},
\end{equation}
where
\begin{align*} {(R_t^{(\ell)})}_{ef} & = \sum_{\gamma\in F^{\ell+1}_{t, ef}}\underline X_e\prod_{s = 1}^{t-1} \underline A_{\gamma_s\gamma_{s+1}}Q_{\gamma_t\gamma_{t+1}}\prod_{s = t+1}^{\ell} A_{\gamma_s\gamma_{s+1}} \\
  {(R_0^{(\ell)})}_{ef} & = \sum_{\gamma\in F^{\ell+1}_{t, ef}}Q_e\prod_{s = 1}^{\ell} A_{\gamma_s\gamma_{s+1}}.
\end{align*}

Note that \( M^{(2)} \) is pretty close to a modified version of \( Q \); more specifically, we make the decomposition 
\[ M^{(2)} = TQT^* + \tilde M = \sum_{k = 1}^r \mu_k \chi_k \check\chi_k^* + \tilde M. \]

Then, the following decomposition holds:
\begin{align*} 
  B^{(\ell)} & = \Delta^{(\ell)} + M D_W B^{(\ell-1)} + \sum_{t = 1}^{\ell - 1}\sum_{k=1}^r \mu_k\Delta^{(t-1)} \chi_k\check\chi_k^* D_W B^{(\ell - t - 1)} \\
             & \quad + \sum_{t= 1}^{\ell - 1}\Delta^{(t-1)}\tilde MB^{(\ell - t - 1)} + \Delta^{(\ell - 1)}M - \sum_{t = 0}^\ell R_t^{(\ell)}.
\end{align*}
Noticing that \( \norm{M} \leq d \) and \( \norm{\chi_k} \leq d\log(n) \), the following lemma ensues:
\begin{lemma}\label{lem:tangle_free_decomp}
  On an event with probability at least \( 1 - 1/\log(n) \), the following inequality holds for any normed vector \( x\in \dR^{\vec E(V)} \):
  \begin{align*}
    \norm{B^\ell x} & \leq \norm{\Delta^{(\ell)}} + L\norm{MB^{\ell - 1}} + d\log(n)\sum_{t=1}^{\ell - 1}\norm{\Delta^{(t-1)}} \sum_{k = 1}^r \left| \langle D_W\check\chi_k, B^{\ell - t- 1}x \rangle \right| \\
                    & \quad + \sum_{t=1}^{\ell - 1} \norm{\Delta^{(t-1)}\tilde MB^{\ell - t - 1}} + d \norm{\Delta^{(\ell - 1)}} - \sum_{t = 0}^\ell \norm{R_t^{(\ell)}}.
  \end{align*}
\end{lemma}

\subsubsection{Norm bounds}

It then remains to bound the different quantities in the lemma above; this is done in another section, using a trace bound method. The results are as follows:
\begin{proposition}\label{prop:trace_bounds}
  On an event with probability \( 1 - c_0/\log(n) \), for any \( k \leq c_1\log(n) \), the following bounds hold with probability at least \( 1 - 1/\ln{(n)}^2 \):
  \begin{align}
    \norm{\Delta^{(k-1)}}                      & \leq c d^3 \log{(n)}^{17}{\left(\sqrt{\rho} \vee L\right)}^{k}, \label{eq:trace_bound_delta}                      \\
    \norm{M B^{k-1}}                           & \leq \frac{c d^{7/2} L \log{(n)}^{7}d^{k}L^k}{\sqrt{n}} ,\label{eq:trace_bound_mb}                                \\
    \norm{\Delta^{(t-1)}\tilde MB^{k - t - 1}} & \leq \frac{c d^{13/2} L \log{(n)}^{24} d^k {\left(\sqrt{\rho} \vee L\right)}^{k}}{\sqrt{n}}, \label{eq:delta_m_b} \\
    \norm{R_t^{(k)}}                           & \leq \frac{c d^2 \log{(n)}^{22} d^k L^{k}}{n}. \label{eq:trace_bound_r}
  \end{align}
\end{proposition}

Using these bounds, we are now finally able to prove~\eqref{eq:norm_orthogonal}: 
\begin{proof}
  By definition of \( \ell \), \( d^\ell \leq n^{1/4} \) so most of the summands in Lemma~\ref{lem:tangle_free_decomp} are negligible with respect to the others. More precisely, we have
  \begin{equation}\label{eq:trace_bounds_non_negligible}
    \norm{B^\ell x} \leq c_1\left( \norm{\Delta^{\ell}} + d\log(n) \sum_{t=1}^{\ell - 1}\norm{\Delta^{(t-1)}} \sum_{k = 1}^r \left| \langle D_W\check\chi_k, B^{\ell - t- 1}x \rangle \right| \right).
  \end{equation}
  When \( k \in [r_0] \), Lemma~\ref{lem:near_orthogonality} implies that
  \[ \left| \langle D_W\check\chi_k, B^{\ell - t- 1}x \rangle \right| \leq \sqrt{r}d^{3/2}L^2\rho^{t/2} + \frac{c_4\, b d^{3/2} \log{(n)}^{9/2} d^{2\ell}L^{\ell}}{n^{1/4}}, \]
  and by definition of \( \ell \), \( d^{2\ell}L^{\ell} \leq {\left(1 \wedge \sqrt{\rho}\right)}^\ell \) so the second term is bounded above by the first. On the other hand, for \( k \in [r]\setminus [r_0] \), we can use equation~\eqref{eq: v_v_dotp} as follows:
  \[ \norm{{(B^*)}^t D_W \check\chi_i}^2 \leq \mu_i^{2t+2}\Gamma_{V, ii}^{(t+1)} + \frac{c\,r b^2 d^4 L^2 \log{(n)}^6 d^{3t}L^{2t}}{\sqrt{n}}. \]
  We now apply Lemma~\ref{lem:scalar_Kt}:
  \[ \Gamma_{V, ii}^{(t+1)} \leq \sum_{s = 0}^{t+1}\frac{r d^2 L^2\rho^s}{\mu_i^{2s}} \leq c r d^2 \log(n) L^2 \rho^{t+1} \mu_i^{-2t-2}, \]
  since \( \mu_i^2 < \rho \); the second term being negligible before the first,
  \[ \left| \langle D_W\check\chi_k, B^{\ell - t- 1}x \rangle \right| \leq \norm*{{(B^*)}^{\ell - t - 1} D_W \check\chi_i} \leq c r d \log(n) L \rho^{\frac{\ell - t}2}.\]

  We can now apply the above bounds on the scalar product as well as those of Proposition~\ref{prop:trace_bounds} to equation~\eqref{eq:trace_bounds_non_negligible}, and we get
  \begin{align*} 
    \norm{B^\ell x} & \leq c_2 d^{5/2} L \log{(n)}^{17}{\left(\sqrt{\rho} \vee L\right)}^{\ell} + c_3 r^{2} d^{6} L^2 \log{(n)}^{20}{\left(\sqrt{\rho} \vee L\right)}^{\ell} \\ &\quad \ \ + c_4 d^4 \log{(n)}^{17}{\left(\sqrt{\rho} \vee L\right)}^{\ell} \\
                    & \leq c r^{2}d^6L^2 \log{(n)}^{20}{\left(\sqrt{\rho} \vee L\right)}^{\ell},
  \end{align*}
  which ends the proof of~\eqref{eq:norm_orthogonal}.
\end{proof}

\section{Trace method: proof of Proposition~\ref{prop:trace_bounds}}

The aim of this section is to prove the bounds in Proposition~\ref{prop:trace_bounds}; we leverage here the powerful trace method introduced by Füredi and Komlòs~\cite{furedi_eigenvalues_1981}, and already used with success in~\cite{bordenave_nonbacktracking_2018} and~\cite{bordenave_detection_2020}. We only prove~\eqref{eq:trace_bound_delta} in this section, all other bounds being proven in the appendix.

\bigskip

Let \( m \) be a parameter to be fixed later. We start with the classical bound
\begin{align*} 
  \norm{\Delta^{(k-1)}}^{2m} & = \norm{\Delta^{(k-1)}\Delta^{(k-1)*}}^m       \\
                             & = \norm{{(\Delta^{(k-1)}\Delta^{(k-1)*})}^m}   \\
                             & \leq \tr({(\Delta^{(k-1)}\Delta^{(k-1)*})}^m).
\end{align*}
Expanding the trace above gives
\begin{align}
  \norm{\Delta^{(k-1)}}^{2m} & \leq \sum_{(e_1, \dots, e_{2m})}\prod_{i=1}^m{(\Delta^{(k-1)})}_{e_{2i-1}, e_{2i}}{(\Delta^{(k-1)})}_{e_{2i+1}, e_{2i}} \nonumber                                                 \\
                             & = \sum_{\gamma \in W_{k, m}}\prod_{i = 1}^{2m} \underline X_{\gamma_{i, 0}\gamma_{i, 1}}\prod_{s = 2}^k \underline A_{\gamma_{i, s-1}\gamma_{i, s}}, \label{eq:first_trace_bound}
\end{align}
where \( W_{k, m} \) is the set of sequences of paths \( (\gamma_1, \dots, \gamma_{2m}) \) such that \( \gamma_i = (\gamma_{i, 0}, \dots, \gamma_{i, k}) \) is non-backtracking tangle-free of length \( k \), and with boundary conditions that for all \( i\in [m] \),
\begin{equation}\label{eq:trace_method_boundary_conditions}
  (\gamma_{2i, k-1}, \gamma_{2i, k}) = (\gamma_{2i-1, k-1}, \gamma_{2i-1, k}) \quand (\gamma_{2i+1, 0}, \gamma_{2i+1, 1}) = (\gamma_{2i, 0}, \gamma_{2i, 1}),
\end{equation}
with the convention \( \gamma_{2m+1} = \gamma_1 \). All the random variables in the expression above are centered and independent as soon as they are supported by distinct edges, so the expectation of each term in the sum is zero except when each (unoriented) edge is visited at least twice. We let \( W'_{k, m} \) be the set of all such sequences of paths. To \( \gamma \in W'_{k, m} \), we associate the graph \( G_\gamma = (V_\gamma , E_\gamma ) \) of visited vertices and edges, and let
\[ v_\gamma  = |V_\gamma | \quand e_\gamma  = |E_\gamma |. \]
For an unoriented edge \( e\in E_\gamma  \), we define its multiplicity \( m_e \) as the number of times \( e \) is visited in \( \gamma \); we also let \( S_\gamma  \) be the set of starting edges in \( \gamma \), that is
\[ S_\gamma  = \Set*{(\gamma_{i, 0}, \gamma_{i, 1}) \given i\in [2m]}. \]
Using these definitions, we can bound the expectation as follows:
\[ \E*{\norm{\Delta^{(k-1)}}^{2m}} \leq \sum_{\gamma\in W'_{k, m}}\prod_{e\in S_\gamma }\E*{|\underline X_e| \cdot |\underline A_e|^{m_e - 1}} \prod_{e \notin S_\gamma } \E*{|\underline A_e|^{m_e}}.  \]
We now bound the two terms in the products above: let \( e \) be an edge, and \( p \geq 2 \) be any multiplicity.
Then conditioning on \( X_e \),
\begin{align*}
  \E*{|\underline A_e|^{m_e}} & = P_e \E*{\left|W_e - P_e\E*{W_e} \right|^p} + (1-P_e)P_e^p \E*{W_e}^p                                                                  \\
                              & \leq P_e L^{p-2} {\left(1+ \frac d n\right)}^{p-2} \E*{{(W_e - P_e\E*{W_e})}^2} + {\left(\frac{dL}n\right)}^{p-2}\frac{dP_e}n\E*{W_e}^2 \\
                              & \leq P_e L^{p-2}{\left(1+ \frac d n\right)}^{p-2} \E*{W_e^2} + P_e L^{p-2} \E*{W_e^2} {\left(\frac d n\right)}^{p-2}                    \\
                              & \leq K_e L^{p-2}{\left(1+ \frac d n\right)}^{p}.
\end{align*}
The other product is trickier; whenever \( p \geq 3 \), a similar computation yields
\[ \E*{|\underline X_e |  \cdot \left|\underline A_e \right|^{p-1}} \leq K_e L^{p-3} {\left(1 + \frac d n \right)}^{p}. \]
On the other hand if \( p = 2 \),
\[ \E*{|\underline X_e |  \cdot \left|\underline A_e \right|} \leq \frac d n L{\left(1 + \frac d n\right)}^2. \]
As a consequence, for \( \gamma \in W'_{k, m} \), we define \( S'_\gamma  \subseteq S_\gamma  \) the set of starting edges with multiplicity 2. Then
\[ \E*{\norm{\Delta^{(k-1)}}^{2m}} \leq \sum_{\gamma\in W'_{k, m}} {\left(1 + \frac d n\right)}^{2km} {\left(\frac d n\right)}^{|S'_\gamma |}d^{2m}L^{2km - 2e_\gamma }\prod_{e\notin S'_\gamma}K_e, \]
where we used \( L^{-1} \leq d \) and \( S_\gamma  = 2m \).

We now partition the paths in \( W'_{k, m} \) as follows: we say that \( \gamma \sim \gamma' \) if there exists a permutation \( \sigma \in \mathfrak S_n \) such that \( \gamma_{i, t} = \sigma(\gamma'_{i, t}) \) for all \( i, t \in [2m]\times [k] \). Clearly, all parameters such as \( v_\gamma  \), \( e_\gamma  \) and \( |S'_\gamma | \) are constant on any equivalence class; therefore it makes sense to define \( \cW_{k, m}(v, e) \) the set of equivalence classes of \( W'(k, m) \) such that \( v_\gamma  = v \) and \( e_\gamma  = e \). Then, a path counting argument performed in~\cite{bordenave_nonbacktracking_2018} yields the following estimation:
\begin{lemma}
  Let \( v, e \) be integers such that \( e - v + 1 \geq 0 \). Then
  \begin{equation}\label{eq:tangle_free_path_counting}
    \cW_{k, m}(v, e) \leq k^{2m}{(2km)}^{6m(e-v+1)}.
  \end{equation}
\end{lemma}

All that remains to bound the sum above is to control the contribution of a single equivalence class; this is done through this lemma:
\begin{lemma}\label{lem:equiv_class_contribution}
  Let \( \gamma \in W'_{k, m} \) such that \( v_\gamma  = v \), \( e_\gamma  = e \) and \( |S'_\gamma|  = s \). We have
  \begin{equation}
    \sum_{\gamma' \sim \gamma} \prod_{f \notin S'_{\gamma'}} K_f \leq d^{2m} n^{v - e + s} \rho^{e} {(\Psi^2)}^{3(e - v) + 8m}.
  \end{equation}
\end{lemma}
\begin{proof}
  For a sequence of paths \( \gamma \in W'_{k, m} \), denote by \( E'_\gamma  \) the set \( E_\gamma  \setminus S'_\gamma  \).Then, due to the boundary conditions in~\eqref{eq:trace_method_boundary_conditions}, the graph \( G'_\gamma  \) induced by \( E'_\gamma  \) is connected. We let \( v_j \) (resp. \( v_{\geq j} \)) be the number of vertices with degree \( j \) (resp.\ at least \( j \)) in \( G'_\gamma  \). Again, by~\eqref{eq:trace_method_boundary_conditions}, removing an edge in \( S'_\gamma  \) does not create a vertex of degree 1; therefore we have 
  \[ v_1 \leq 4m, \]
  since a vertex of \( G_\gamma  \) can only be of degree 1 if it is an endpoint of \( \gamma_i \) for some \( i\in [2m] \). Additionally, edge and vertex counting yields
  \[ v_1 + v_2 + v_{\geq 3} \geq v - s \quand v_1 + 2v_2 + 3v_{\geq 3} \leq 2(e-s),  \]
  since removing an edge in \( S'_\gamma \) removes at most one vertex from \( G_\gamma \). Combining those inequalities gives
  \begin{equation}\label{eq:reduced_graph_edge}
    v_{\geq 3} + v_1 \leq 2(e-s) - 2(v-s) + 2v_1 \leq 2(e-v) + 8m;
  \end{equation}
  this inequality encodes the fact that in a union of paths most vertices are of degree 2. We now reduce \( G'_\gamma  \) into a multigraph \( \hat{G}_\gamma  = (\hat{V}_\gamma , \hat{E}_\gamma ) \) as follows: \( \hat{V}_{\gamma} \) is the set of vertices in \( G'_\gamma \) with degree different from 2, and we add an edge between two vertices \( x_1 \) and \( x_2 \) of \( \hat{V}_\gamma \) for each path between \( x_1 \) and \( x_2 \) in \( G'_\gamma \). For \( \hat{f} \in \hat{E}_\gamma \), we annotate \( \hat{f} \) with the length \( q_{\hat{f}} \) of its corresponding path in \( G'_\gamma \).
  
  We let \( \hat{v} \) and \( \hat{e} \) be the number of vertices and edges of \( \hat{G}_\gamma \); a sequence \( \gamma' \sim \gamma \) is uniquely determined by an embedding of \( \hat{V}_\gamma \) in \( [n] \) and for each edge \( \hat{f} \in \hat{E}_\gamma \), an embedding of \( \hat{f} \) as a path of length \( q_{\hat{f}} \). As a result, we have
  \begin{align*}
    \sum_{\gamma' \sim \gamma} \prod_{f \notin S'_{\gamma'}} Q_f & \leq \sum_{y_1, \dots, y_{\hat{v}} \in {[n]}^{\hat{v}}} \prod_{\hat{f} = (y_i, y_j) \in \hat{E}_\gamma} \sum_{x_1, \dots, x_{q_{\hat{f}} - 1} \in [n]} \prod_{t=1}^{q_{\hat{f}}} K_{x_{t-1}, x_t} \\
                                                                 & = \sum_{y_1, \dots, y_{\hat{v}} \in {[n]}^{\hat{v}}} \prod_{\hat{f} = (y_i, y_j) \in \hat{E}_\gamma} {(K^{q_{\hat{f}}})}_{y_i, y_j}                                                               \\
                                                                 & \leq \sum_{y_1, \dots, y_{\hat{v}} \in {[n]}^{\hat{v}}} \prod_{\hat{f} \in \hat{E}_\gamma} \left( \frac{\Psi^2}{n} \rho^{q_{\hat{f}}}\right),
  \end{align*}
  using~\eqref{eq:elementwise_Kt} and recalling that \( \Psi = L^2/\rho \). Now, notice that
  \[ \sum_{\hat{f} \in \hat{E}_\gamma} q_{\hat{f}} = |E'_\gamma| = e - s \quand \hat{e} - \hat{v} = |E'_\gamma| - |V'_\gamma| \geq e - v - s ;\]
  further \( \hat{e} \leq \hat{v} + e - v - s \leq 3(e - v) + 8m - s \) using~\eqref{eq:reduced_graph_edge} and the inequality above. We finally find
  \begin{align*}
    \sum_{\gamma' \sim \gamma} \prod_{f \notin S'_{\gamma'}} Q_f & \leq n^{\hat{v} - \hat{e}} {(\Psi^2)}^{\hat{e}} \rho^{e - s}    \\
                                                                 & \leq n^{v - e + s} \rho^{e - s} {(\Psi^2)}^{3(e - v) + 8m - s},
  \end{align*}
  which ends the proof of Lemma~\ref{lem:equiv_class_contribution}, since \( \Psi^2 \geq 1 \) and \( \rho^{-1} \leq a \). 
\end{proof}
We now are able to conclude; the contribution of one equivalence class in \( \cW_{k, m}(v, e) \) is less than
\begin{align*} 
  C_\gamma & = {\left(1 + \frac d n\right)}^{2km} {\left(\frac d n\right)}^{|S'_\gamma |}d^{2m}L^{2km - 2e }\sum_{\gamma' \sim\gamma}\prod_{e\notin S'_{\gamma'}}K_e \\
           & \leq c_1^{2m} d^{6m} n^{-|S'_\gamma|} L^{2km - 2e} n^{v - e + |S'_\gamma|} \rho^{e} {( \Psi^2)}^{3(e - v) + 8m}                                         \\
           & \leq c_1^{2m} d^{6m} n^{v - e} {(\rho\Psi)}^{km - e} \rho^e {(\Psi^2)}^{3(e-v) + 8m}                                                                    \\
           & \leq c_1^{2m} d^{6m} \rho^{km} {\left(\frac\Psi d \right)}^{km - e} n^{1 - g} {(\Psi^2)}^{3g + 8m},
\end{align*}
with \( g = e - v + 1 \) and we used that \( L = \sqrt{\rho\Psi/d} \) and the bound
\[ {\left(1 + \frac d n\right)}^{k} \leq \exp\left( \frac{dk}n \right) \leq c_1. \]
Summing over all equivalence classes now gives
\begin{align}
   & \E*{\norm{\Delta^{(k-1)}}^{2m}} \leq \sum_{e = 1}^{km} \sum_{v = 1}^{e+1} |\cW_{k, m}(v, e)| \max_{[\gamma] \in \cW_{k, m}(v, e)} C_\gamma \nonumber                                       \\
   & \leq \sum_{e = 1}^{km} \sum_{v = 1}^{e+1} k^{2m}{(2km)}^{6m(e-v+1)} c_1^{2m} d^{6m} \rho^{km} {\left(\frac\Psi d \right)}^{km - e} n^{1 - g} {(\Psi^2)}^{3g + 8m} \nonumber                \\
   & \leq n {(c_1 d^3 k)}^{2m} \rho^{km} \sum_{e = 1}^{km}{\left(\frac\Psi d \right)}^{km - e} \sum_{g = 0}^{\infty}{\left( \frac{\Psi^6 {(2km)}^{6m}}n \right)}^g.  \label{eq:bound_delta_sum}
\end{align}
We set the parameter \( m \) to
\[ m = \left\lceil \frac{\log\left( \frac{n}{\Psi^6} \right)}{12\log(\log(n))} \right\rceil;\]
when \( n \geq c_2\Psi^6 \) for some absolute constant \( c_2 \), we have
\[ \frac{\Psi^6 {(2km)}^{6m}}n < \frac12 \quand n^{\frac 1{2m}} \leq \log{(n)}^{12}. \]
The infinite sum inside~\eqref{eq:bound_delta_sum} thus converges, and
\[ \E*{\norm{\Delta^{(k-1)}}^{2m}}^{\frac 1 {2m}} \leq c_3 d^3 \log{(n)}^{14} \sqrt{\rho}^k {\left(1 \vee \sqrt{\frac \Psi d}\right)}^{k}. \]
Finally, from the definition of \( \Psi \), \( \sqrt{\rho}\left(1 \vee \sqrt{\Psi/d}\right) = \sqrt{\rho} \vee L \), hence~\eqref{eq:trace_bound_delta} by a Markov bound.

\newpage

\appendix

\section{Applications of Theorem~\ref{th:main}}
\subsection{Proof of Proposition~\ref{prop:ihara_bass}}

Let \( x \) be an eigenvector of \( B \) associated with the eigenvalue \( \lambda \); the eigenvalue equation for \( x \) reads
\begin{equation} \label{eq:eigen_equation_B}
  \lambda x_e = \sum_{e \rightarrow f} W_{f}x_f.
\end{equation}
On the other hand, the definition \( y = S^* D_W x \) expands to
\[ y_i = \sum_{e: e_1 = i} W_e x_e. \]
Applying equation~\eqref{eq:eigen_equation_B} to \( e \) and \( e^{-1} \) yields
\[ \lambda x_e = y_{e_2} - W_{e}x_{e^{-1}} \quand \lambda x_{e^{-1}} = y_{e_1} - W_e x_e,\]
and as a result
\[ \lambda^2 x_e = \lambda y_{e_2} - \lambda W_e x_{e^{-1}} = \lambda y_{e_2} - W_e(y_{e_1} - W_e x_e). \]
Rearranging the terms, we find an expression for \( x_e \):
\begin{equation}\label{eq:x_y_expression}
  x_e = \frac{\lambda y_{e_2} - W_e y_{e_1}}{\lambda^2 - W_e^2};
\end{equation}
in particular \( y \neq 0 \) if \( x \neq 0 \). Plugging~\eqref{eq:x_y_expression} into the eigenvalue equation~\eqref{eq:eigen_equation_B}, we get for \( i, j \in [n] \)
\[ \frac{\lambda^2 y_{i} - \lambda W_{ij}y_{j}}{\lambda^2 - W_{ij}^2} = \sum_{\substack{k \sim i \\ k \neq j}} W_{ik}\frac{\lambda y_{k} - W_{ik} y_{i}}{\lambda^2 - W_{ik}^2}, \]
and we rearrange to find
\[ \frac{\lambda^2 y_i}{\lambda^2 - W_{ij}^2} - \frac{W_{ij}^2 y_i}{\lambda^2 - W_{ij^2}} = \sum_{k \sim i} \frac{\lambda W_{ik}}{\lambda^2 - W_{ik}^2} y_k - \sum_{k \sim i} \frac{W_{ik}^2}{\lambda^2 - W_{ik}^2} y_i. \]
The fraction on the LHS cancels out, and writing the RHS as a matrix product
\[ y = \tilde A(\lambda)y - \tilde D(\lambda) y, \]
the desired result.

\subsection{Proof of Theorem~\ref{th:bbp_transition}}

Our first step is to show that the matrices involved in Proposition~\ref{prop:ihara_bass} approximate the matrices \( A \) and \( \rho I \). If \( \lambda^2 \geq 2L^2 \), we have
\[ \left|\lambda \tilde A_{ij}(\lambda) - A_{ij}\right| = \ind \{i \sim j \}\left|\frac{W_{ij}}{1 - \frac{W_{ij}^2}{\lambda^2}} - W_{ij} \right| \leq 1(i \sim j) \frac{2L W_{ij}^2}{\lambda^2}, \]
which implies using the Gershgorin circle theorem
\begin{equation}\label{eq:perturbation_a_lambda}
  \norm*{\lambda \tilde A_{ij}(\lambda) - A_{ij}} \leq \frac{2L}{\lambda^2}\max_i \sum_{j \sim i}W_{ij}^2 \leq \frac{4L \rho}{\lambda^2}.
\end{equation}
Similarly,
\begin{align}
  \left|\lambda^2 \tilde D_{ii}(\lambda) - \rho \right| & \leq \frac{2L^2}{\lambda^2}\sum_{j \sim i} W_{ij}^2 + \left|\sum_{j \sim i} W_{ij}^2 - \rho \right| \nonumber \\
                                                        & \leq \left( \frac{4L^2}{\lambda^2} + \eps \right)\cdot \rho. \label{eq:perturbation_d_lambda}
\end{align}
We now take \( \lambda = \lambda_i \) with \( i\in [r_0] \); then there is a vector \( y \) that is a singular value of
\[ -\lambda_i \Delta(\lambda_i) = A - (\lambda_i + \frac{\rho}{\lambda_i})I + (\lambda \tilde A(\lambda_i) - A) - \lambda_i^{-1}(\lambda^2 \tilde D(\lambda) - \rho I). \]
We can thus apply Weyl's inequality~\cite{weyl_asymptotische_1912} to find that there exists an eigenvalue \( \nu_i \) of \( A \) such that
\[ \left| \nu_i - \left( \lambda_i + \frac{\rho}{\lambda_i} \right) \right| \leq \frac{4L \rho}{\lambda_i^2} + \left( \frac{4L^2}{\lambda_i^2} + \eps \right)\cdot\frac{\rho}{\lambda_i}. \]
Now, we use Theorem~\ref{th:main} to find that \( |\lambda_i - \mu_i|\leq \sigma \), and we have \( \sigma = o(\rho) \) whenever \( n \) is large enough by virtue of~\eqref{eq:bound_rho_below}.
Since
\[ |\lambda_i - \mu_i| \leq \sigma \quand \left|\frac{\rho}{\lambda_i} - \frac{\rho}{\mu_i} \right| \leq \frac{\rho}{\lambda_i \mu_i} \sigma \leq c_0 \sigma, \]
equation~\eqref{eq:bbp_eigenvalue_bound} ensues by noticing that \( \lambda_i > c_1\mu_i \) for some constant \( c_1 \) and \( \sigma \) is negligible before the other error terms.

\bigskip

Assume now that \( \delta_i \geq 2\sigma \); examining the proof of Theorem~\ref{th:main}, we have the existence of an eigenvector \( \xi \) of \( B \) associated with \( \lambda_i \) such that
\[ \norm*{\xi - u_i} \leq \frac{3\sigma \norm{u_i}}{\delta_i - \sigma}. \]
Proposition~\ref{prop:ihara_bass} implies that the vector \( y = S^* D_W \xi \) is a null vector of the deformed laplacian \( \Delta(\lambda) = I - \tilde A(\lambda) + \tilde D(\lambda) \). Notice that the matrix \( S^* D_W^2 S \) is a diagonal matrix such that
\[ {[S^* D_W^2 S]}_{ii} = \sum_{j \sim i} W_{ij}^2 \leq 2\rho, \]
from which we have
\[ \norm{y - S^*D_W u_i} \leq \frac{6\sigma \sqrt{\rho} \norm{u_i}}{\delta_i - \sigma}. \]
We now follow the line of proof of Theorem~\ref{th:main}; we first find
\[ \langle S^* D_W u_i, \varphi_i \rangle = \mu_i^{-\ell} \langle B^\ell \chi_i, D_W \check \chi_i \rangle, \]
and combine it with~\eqref{eq:u_v_dotp} to obtain
\begin{equation}\label{eq:adjacency_scalar_product}
  \left| \langle S^* D_W u_i, \varphi_i \rangle - \mu_i \right| \leq \sigma.
\end{equation}
Computing \( \norm{S^* D_W u_i} \) is trickier; we find
\begin{align*}
  \langle S^* D_W u_i, S^* D_W u_i \rangle & = \mu_i^{-2\ell} \langle S^* D_W B^\ell \chi_i, S^* D_W B^\ell \chi_i \rangle             \\
                                           & = \mu_i^{-2\ell} \langle S^* D_W B^\ell \chi_i,  T^*J D_W B^\ell \chi_i \rangle           \\
                                           & = \mu_i^{-2\ell} \langle TS^* D_W B^\ell \chi_i,  {(B^*)}^\ell D_W \check \chi_i \rangle.
\end{align*}
Writing the coefficients of \( TS^*D_W \) explicitly, we have
\[ {[TS^*D_W]}_{ef} = W_f \sum_{i\in [n]} \ind \{e_2 = i \}\ind \{f_1 = i \} = B_{ef} + {[JD_W]}_{ef},\]
which yields
\[ \langle S^* D_W u_i, S^* D_W u_i \rangle = \mu_i^{-2\ell} \left( \langle B^{2\ell + 1}\chi_i, D_W \check \chi_i \rangle + \langle B^\ell D_W \check \chi_i, B^\ell D_W \check \chi_i \rangle \right). \]
Those scalar products correspond to equations~\eqref{eq:u_v_dotp} and~\eqref{eq: v_v_dotp}, respectively, and we thus get
\[ \left|\norm{S^* D_W u_i}^2 - \mu_i^2(1 + \Gamma_{V, ii}^{(\ell)}) \right| \leq 2\sigma. \]
The hypothesis \( K\ind = \rho\ind \) allows us to approximate \( \Gamma_{V, ii}^{(\ell)} \) efficiently:
\[ \Gamma_{V, ii}^{(\ell)} = \sum_{t = 0}^\ell \frac{\langle \ind, K^{t+1}\varphi^{i, i} \rangle}{\mu_i^{2t+2}} = \sum_{t = 0}^\ell {\left( \frac{\rho}{\mu_i^2} \right)}^{t+1} \]
since \( \norm{\varphi_i} = 1 \), and we have as in the proof of Theorem~\ref{th:main}
\[ \left|\Gamma_{V, ii}^{(\ell)} - \frac{\rho/\mu_i^2}{1 - \rho/\mu_i^2} \right| \leq \sigma. \]
Gathering the previous bounds, we eventually arrive at
\begin{equation}\label{eq:adjacency_norm_eigen}
  \left|\norm{S^* D_W u_i}^2 - \frac{\mu_i^2}{1 - \rho/\mu_i^2} \right| \leq 3\sigma.
\end{equation}
The exact same computations imply that
\[ \norm{u_i}^2 \leq \frac{d}{1 - \rho/\mu_i^2} + c_5 \sigma, \]
and thus noticing that \( \mu_i \geq \sqrt{\rho} \)
\[ \norm*{\frac{y}{\norm{y}} - \frac{S^* D_W u_i}{\norm{S^* D_W u_i}}} \leq \frac{c_6\,\sigma \sqrt{d}}{\delta_i - \sigma}. \]
Combining this error bound with~\eqref{eq:adjacency_scalar_product} and~\eqref{eq:adjacency_norm_eigen}, we find the following result:
\[ \left| \left\langle \frac{y}{\norm{y}}, \varphi_i \right \rangle - \sqrt{1 - \frac{\rho}{\mu_i^2}} \right| \leq \frac{c_7\,\sigma \sqrt{d}}{\delta_i - \sigma}. \]

The final step is to use the Davis-Kahan theorem~\cite{yu_useful_2015} as follows: there exists an eigenvector \( \zeta \) of \( A \) with associated eigenvalue \( \nu_i \), and such that
\[ \norm*{\zeta - \frac{y}{\norm{y}}} \leq \frac{c_8 \left( \frac{4L \rho}{\lambda_i^2} + \left( \frac{4L^2}{\lambda_i^2} + \eps \right)\cdot\frac{\rho}{\lambda_i} \right)}{\delta_i}. \]
This error term dominates all the other ones found above, hence the bound in Theorem~\ref{th:bbp_transition}.

\bigskip

The proof of Corollary~\ref{cor:bbp_unweighted} follows along the same lines; however, we have directly
\[ \tilde A(\lambda) = \frac{\lambda A}{\lambda^2 - 1} \quand \tilde D(\lambda) = \frac{d_0}{\lambda^2 - 1}I, \]
and thus the approximation bounds~\eqref{eq:perturbation_a_lambda} and~\eqref{eq:perturbation_d_lambda} become superfluous.

\subsection{Proof of Theorem~\ref{th:unlabeled_sbm}} \label{subsec: proof_sbm}

We first link the SBM setting to the one of Theorem~\ref{th:main}. In the unweighted case, we have \( Q = K = P \), and the eigenvector equation \( P\ind = \alpha\ind \) yields \( \rho = \alpha \). It is easy to check that whenever \( n \) is large enough, the \( r_0 \) defined in Theorem~\ref{th:unlabeled_sbm} satisfies the assumptions of Theorem~\ref{th:main}, with \( \tau = 1/(\alpha \mu_{r_0}^2) < 1 \). Equation~\eqref{eq:linear_communities} ensures that \( \norm{\varphi_i}_\infty \leq c/\sqrt{n} \) for some absolute constant \( c > 0 \), therefore \( b = O(1) \). Finally, since \( \tau^{-1} = \alpha \mu_{r_0} \), we have
\[ C_0 \leq c \alpha \log{(n)}^{25} \quand n_0 \leq \exp(c \log(d) \log(\log(n))). \]

An application of Theorem~\ref{th:main} thus directly yields the bound on the eigenvalues of \( B \); regarding the eigenvectors, notice that as in the proof of Theorem~\ref{th:bbp_transition}
\[ \norm{u_i}^2 = \frac{\alpha}{1 - 1/(\alpha \mu_i^2)} + O(\sigma) \quand \norm{T\varphi_i} = \alpha + O(\sigma), \]
which gives
\[ \langle \xi, \xi_i \rangle = \sqrt{1 - \frac{1}{\alpha\mu_i^2}} + O(\sigma). \]

\subsection{Proof of Theorem~\ref{th:labeled_sbm} and Proposition~\ref{prop:labeled_symmetric_sbm}}

Letting again \( \Theta \) be the \( n \times 2 \) group membership matrix, we find as in the proof of Theorem~\ref{th:unlabeled_sbm} that we have \( Q = \Theta \tilde Q \Theta^* \) and \( K = \Theta \tilde K \Theta^* \), with
\[ \tilde Q = \frac12\begin{pmatrix} a\dE_\dP[w] & b\dE_\dQ[w] \\b\dE_\dQ[w] & a\dE_\dP[w] \end{pmatrix} \quand \tilde K = \frac12\begin{pmatrix} a\dE_\dP[w^2] & b\dE_\dQ[w^2] \\b\dE_\dQ[w^2] & a\dE_\dP[w^2] \end{pmatrix}. \]
This implies first that
\[ \rho = \frac{a\dE_\dP[w^2] + b\dE_\dQ[w^2]}2, \]
and that the vector \( \Theta\dbinom{1}{-1} \) is an eigenvector of \( Q \) associated with the eigenvalue
\[ \mu_2 = \frac{a\dE_\dP[w] - b\dE_\dQ[w]}2. \]
All other hypotheses of Theorem~\ref{th:main} are easy to check, and we find that the announced results hold as soon as \( \mu_2 ^2 > \rho \vee L \), or
\[ \frac{(a\dE_\dP[w^2] + b\dE_\dQ[w^2]) \vee L}{{(a\dE_\dP[w] - b\dE_\dQ[w])}^2} < 1. \]

Now, let us disregard for a moment the condition on \( L \), and compute \( \rho \):
\[ \rho = \frac12 \int_{\cL}(af(\ell) +bg(\ell))w{(\ell)}^2 \,\mathrm dm(\ell) \]
Define a scalar product on \( \ell^\infty(\cL) \), the set of all bounded functions from \( \cL \) to \( \dR \), as
\[ \langle h_1, h_2 \rangle_{\cL} = \int_{\cL}(af(\ell) +bg(\ell))h_1(\ell)h_2(\ell) \,\mathrm dm(\ell); \]
then \( \rho = \norm{w}_\cL^2 \), and applying the Cauchy-Schwarz theorem
\[ \rho \cdot \norm*{\frac{af - bg}{af + bg}}^2_\cL \geq \left\langle w, \frac{af - bg}{af + bg} \right\rangle_\cL^2 = \mu_2^2. \]
This implies that the signal-to-noise ratio \( \mu_2^2/\rho \) is maximized whenever
\[ w(\ell) = \frac{af(\ell) - bg(\ell)}{af(\ell) + bg(\ell)}, \]
and in this case
\[ \beta = \frac{\mu_2^2}{\rho} = \frac12\int_\cL \frac{{(af(\ell) - bg(\ell))}^2}{af(\ell) + bg(\ell)}\, \mathrm dm(\ell) \]
In particular, we have \( \mu_2 = \rho = \beta \), so \( \beta > 1 \) implies \( \mu_2 > 1 \). It remains to notice that \( w(\ell) \leq 1 \) for any \( \ell \), so the condition \( \mu_2 \geq L \) is redundant as assumed.

\subsection{Proof of Theorem \ref{th:gaussian}}

For $i, j \in [n]$, we note $W_{ij} = m_{ij} + s_{ij} Z_{ij}$ with $Z \sim \cN(0, 1)$ a standard gaussian random variable. Let $\tilde L = 2\sqrt{\log(n)}$; a well known tail bound for gaussians reads
\begin{equation}
  \Pb{|Z_{ij}| \ge \tilde L} \leq \frac{2}{\tilde L} e^{-\tilde L/2} \leq \frac{1}{n^2 \sqrt{\log(n)}}.
\end{equation}

We now define the modified matrix $\tilde W$ with
\[ \tilde W_{ij} = m_{ij} + s_{ij}Z_{ij} \ind\{|Z_{ij}| \leq \tilde L\}, \]
with $\tilde Q$ and $\tilde K$ the associated expected and variance matrices. It is readily seen that $\tilde Q = Q$, and that the variables $\tilde W_{ij}$ are bounded by
\[ L = \sup_{i, j} |m_{ij}| + \tilde L\sup_{i, j} s_{ij}. \]
By a union bound, we have
\[ \Pb{\tilde W \neq W} = \Pb{Z_{ij} > \tilde L \text{ for some } i \in [n]} \leq \dbinom n 2  \frac{1}{n^2 \sqrt{\log(n)}} \leq \frac{1}{2 \sqrt{\log(n)}}, \]
and whenever $\tilde W = W$, then the modified non-backtracking matrix coincides with the original one. Finally, notice that for $i, j \in [n]$
\[ \Var(Z_{ij} \ind\{|Z_{ij}|\leq \tilde L\}) \leq 1, \]
which implies using the Perron-Frobenius theorem that $\rho(\tilde K) \leq \rho(K)$. Theorem \ref{th:main} then applies to the modified couple $(P, \tilde W)$ and the announced result follows.

\section{Computing functionals on trees}
We prove in this section the martingale estimates of Proposition~\ref{prop:functionals_tree} and Proposition~\ref{prop:edge_functionals_tree}.

\subsection{Study of compound Poisson processes}
Many proofs in this section rely on computations of Poisson compound processes, i.e.\ Poisson sums of random variables. For convenience, we gather them all in the following lemma:

\begin{lemma}\label{lem:poisson_sums}
  Let \( N \) be a \( \Poi(d) \) random variable, and \( (X_i) \), \( (Y_i) \), \( (Z_i) \) three iid sequences of random variables, independent from \( N \), such that \( X_i \) and \( Y_j \) (resp. \( Y_i \) and \( Z_j \), or \( Z_i \) and \( X_j \)) are independent whenever \( i \neq j \). Denote by \( A, B \) the random variables
  \[ A = \sum_{i=1}^N X_i \quand B = \sum_{i=1}^N Y_i, \]

  Then the following identities hold:
  \begin{align}
     & \E A = d\E X, \quad \E B = d \E Y, \label{eq:poisson_sum_expectation}                                                                               \\
     & \E{AB}= d\E{XY} + d^2 \E X \E Y = d\E{XY} + \E A \E B, \label{eq:poisson_sum_correlation}                                                           \\
     & \E*{\sum_{i = 1}^N Z_i \left(\sum_{j \neq i} X_j \right)} = d \E{A}\E{Z},   \label{eq:poisson_sum_tree}                                             \\
     & \E*{\sum_{i = 1}^N Z_i \left(\sum_{j \neq i} X_j \right)\left(\sum_{k \neq i} Y_k \right)} = d \E{AB}\E{Z}. \label{eq:poisson_sum_tree_correlation}
  \end{align}
\end{lemma}

Although the first two identities are well-known, we provide a full proof of this lemma:

\begin{proof}
  The sequence \( (X_i) \) being independent from \( N \), we immediately find that
  \[ \E{A \given N} = N\E{X}, \]
  from which eq.~\eqref{eq:poisson_sum_expectation} is derived.
  We then write
  \[ AB = \left(\sum_{i=1}^N X_i\right)\left(\sum_{i=1}^N Y_i\right) = \sum_{i=1}^N X_i Y_i + \sum_{i \neq j}X_i Y_j, \]
  and using the independence property of \( {(X_i)}_i \) and \( {(Y_i)}_i \) yields
  \[ \E*{AB \given N} = N\E{XY} + N(N-1)\E X\E Y. \]
  Since \( N \) is a Poisson random variable, \( \E{N(N-1)} = d^2 \), hence~\eqref{eq:poisson_sum_correlation}.

  \bigskip

  We now move onto the third equation; rearranging terms gives
  \[ \sum_{i = 1}^N Z_i \left(\sum_{j \neq i} X_j \right) = \sum_{i \neq j}{Z_i X_j}, \]
  and therefore the conditional expectation given \( N \) is \( N(N-1) \E{X}\E{Z} \). Using again that \( \E{N(N-1)} = d^2 \) brings~\eqref{eq:poisson_sum_tree}.

  Similarly, we can rearrange
  \[ \sum_{i = 1}^N Z_i \left(\sum_{j \neq i} X_j \right)\left(\sum_{k \neq i} Y_k \right) = \sum_{j \neq i} X_j Y_j Z_i + \sum_{i \neq j \neq k} X_i Y_j Z_j, \]
  and take conditional expectations on both sides to arrive at
  \begin{multline*}
    \E*{\sum_{i = 1}^N Z_i \left(\sum_{j \neq i} X_j \right)\left(\sum_{k \neq i} Y_k \right) \given N} \\= N(N-1)\E{XY}\E{Z} + N(N-1)(N-2)\E X \E Y \E Z.
  \end{multline*}
  Again, the expected value of \( N(N-1)(N-2) \) is \( d^3 \), and we finally find
  \begin{align*}
    \E*{\sum_{i = 1}^N Z_i \left(\sum_{j \neq i} X_j \right)\left(\sum_{k \neq i} Y_k \right)} & = d^2 \E{XY} \E{Z} + d^3 \E X \E Y \E Z \\
                                                                                               & = d\E{AB} \E Z,
  \end{align*}
  which ends the proof.
\end{proof}

\subsection{Decomposing the tree functionals}

We now fix \( t \geq 1 \), \( x \in [n] \) and two vectors \( \varphi, \varphi' \in \dR^n \) for the rest of the section. Let \( N \) be the number of children of the root of \( T \), and \( {(T_k, I_k)}_{k \leq N} \) the subtrees at depth 1. We further introduce the following first moment notations:
\[ g_{\varphi}(t, x) = \E*{f_{\varphi, t}(T_x, x)} \quand h_{\varphi, \varphi'}(t, x) = \E*{f_{\varphi, t}(T_x, x)f_{\varphi', t}(T_x, x)}. \]
We begin by a small elementary computation: let \( \phi \in \dR^n \) be any vector. Then,
\begin{equation}\label{eq:expectation_child_root}
  \E*{W_{xI_k}\phi(I_k)} = \sum_{y \in [n]}{\frac{P_{xy}}{d_x}\E*{W_{xy}}\phi(y)} = \frac{[Q\phi](x)}{d_x}.
\end{equation}
\medskip
Now, by linearity, we have
\begin{equation}\label{eq:galton_watson_linearity}
  f_{\varphi, t}(T_x, x) = \sum_{k = 1}^N{W _{xI_k}f_{\varphi, t-1}(T_k, I_k)}.
\end{equation}
By definition of the Galton-Watson tree, the random variables \( X_k = W _{xI_k}f_{\varphi, t-1}(T_k, I_k) \) and \( Y_k = W _{xI_k}f_{\varphi', t-1}(T_k, I_k) \) satisfy the assumptions of Lemma~\ref{lem:poisson_sums}. Furthermore, conditioning on the value of \( I_k \), we can compute \( \dE{X_k} \):
\begin{align*}
  \E*{ W _{xI_k} f_{\varphi, t-1}(T_k, I_k)} & = \E*{W _{xI_k}g_\varphi(t-1, I_k)}        \\
                                             & = \frac{[Qg_\varphi(t-1, \cdot)](x)}{d_x}.
\end{align*}
Applying~\eqref{eq:poisson_sum_expectation}, and from the definition of \( g_\varphi \), we come to the following recurrence relation:
\[ g_\varphi(t, x) = [Qg(t-1, \cdot)](x).\]
Solving this recurrence is straightforward, and we find
\[ g_\varphi(t, \cdot) = Q^t g_\varphi(0, \cdot) = Q^t\varphi, \]
which implies~\eqref{eq:functional_nonzero_eigen}.

\bigskip

Using now equation~\eqref{eq:poisson_sum_correlation} from Lemma~\ref{lem:poisson_sums}, we derive
\begin{equation}\label{eq:recurrence_correlation_gw}
  \begin{split}
    h_{\varphi, \varphi'}(t, x) &= d_x\E*{W_{xI_k}^2h_{\varphi,\varphi'}(t-1, I_k)} + g_\varphi(t, x)g_{\varphi'}(t, x) \\
    &= [K h_{\varphi, \varphi'}(t, \cdot)](x) + g_\varphi(t, x)g_{\varphi'}(t, x),
  \end{split}
\end{equation}
from which we can solve for \( h_{\varphi, \varphi'} \):
\begin{align*}
  h_{\varphi, \varphi'}(t, \cdot) & = Kh_{\varphi, \varphi'}(t-1, \cdot) + (Q^t \varphi) \odot (Q^t \varphi') \\
                                  & = \sum_{s=0}^t K^s[(Q^{t-s}\varphi) \odot (Q^{t-s}\varphi')].
\end{align*}
The eigenvector equations for \( \varphi_i \) and \( \varphi_j \) then imply~\eqref{eq:functional_nonzero_corr}.

\bigskip

Consider now the funtion \( F_{i, t}(T_x, x) = f_{\varphi_i, t+1}(T_x, x) -  \mu_i f_{\varphi_i, t}(T_x, x) \), and its associated first moment functions
\[ G(t, x) = \E*{F_{i, t}(T_x, x)} \quand H(t, x) = \E*{F_{i, t}{(T_x, x)}^2}. \] The linearity of \( f_{\varphi_i, t} \) implies that \( F_{i, t} \) also verifies equation~\eqref{eq:galton_watson_linearity}, and therefore
\[ G(t, \cdot) = Q^t G(0, \cdot) = 0 \]
for all \( t \geq 0 \). Equation~\eqref{eq:recurrence_correlation_gw} thus reduces to
\begin{align*}
  H(t, \cdot) & = K^t H(0, \cdot)                                                                   \\
              & = K^t \left(h_{\varphi_i, \varphi_i}(1, \cdot) - \mu_i^2\varphi^{i, i}\right)       \\
              & = K^t \left(\mu_i^2 \varphi^{i, i} + K\varphi^{i, i} - \mu_i^2\varphi^{i, i}\right) \\
              & = K^{t+1}\varphi^{i, i},
\end{align*}
which ends the proof.

\subsection{Edge functionals}

Most of the handiwork needed to prove Proposition~\ref{prop:edge_functionals_tree} was done in Lemma~\ref{lem:poisson_sums}; indeed, in the tree \( (T_x, x) \), the edge transformation on \( f_{\varphi, t} \) can be written as
\begin{align*}
  \vec\partial_w f_{\varphi, t}(T_x, x) & = \sum_{j = 1}^N w_{I_j} \sum_{k \neq j} W_{xI_k}f_{\varphi, t-1}(T(k), I_k).
\end{align*}
We define accordingly the random variables
\[ X_k = W _{xI_k}f_{\varphi, t-1}(T(k), I_k), \quad Y_k = W _{xI_k}f_{\varphi', t-1}(T(k), I_k) \quand Z_k = w_{I_k},\]
that verify the assumptions of Lemma~\ref{lem:poisson_sums}. Computing \( \E{Z} \) is straightforward:
\[ \E{Z} = \sum_{y\in [n]} \frac{P_{xy}}{d_x} \E{w_y} = [P\bar w](x). \]
Hence, we can apply equation~\eqref{eq:poisson_sum_tree} to those variables, to deduce~\eqref{eq:edge_nonzero_eigen}. Similarly, the product transformation has the form
\begin{multline*} \vec \partial_w(f_{\varphi, t}\cdot f_{\varphi', t})(T_x, x) = \sum_{j = 1}^N w_{I_j}\left(\sum_{k \neq j} W_{xI_k}f_{\varphi, t-1}(T(k), I_k)\right)\\ \times\left(\sum_{k \neq j} W_{xI_k}f_{\varphi', t-1}(T(k), I_k)\right),
\end{multline*}
which using~\eqref{eq:poisson_sum_tree_correlation} implies~\eqref{eq:edge_nonzero_corr}. Finally, equation~\eqref{eq:edge_square_increments} is proved with the exact same technique, considering \( F_{i, t}(T_x, x) \) instead of \( f_{\varphi, t}(T_x, x) \).

\section{Near eigenvectors: computations}

We finish here the proof of Proposition~\ref{prop:pseudo_eigenvectors}. First, let
\[ f(g, o) = \ind_{{(g, o)}_{t+1} \text{ has no cycles}}\,\varphi_j(o) f_{\varphi, t+1}(g, o). \]
Then \( f \) is \( (t+1) \)-local, and we have
\begin{align*}
  \left|f(g, o)\right| & \leq \norm{\varphi_i}_\infty \norm{\varphi_j}_\infty \left| \partial {(g, o)}_{t+1} \right| L^{t+1} \\
                       & \leq \frac{b^2}n |(g, o)|_{t+1} L^{t+1} := \psi(g, o).
\end{align*}
On the other hand, the scalar product \( \langle B^t\chi_i, D_W\check\chi_j \rangle \) can be written as
\begin{align*}
  \langle B^t\chi_i, D_W\check\chi_j \rangle & = \sum_{e\in \vec E}W_e\varphi_j(e_1) \sum_{\gamma} \prod_{s=1}^{t} W_{\gamma_s\gamma_{s+1}} \varphi_i(\gamma_{t+1}) \\
                                             & = \sum_{e\in \vec E}\varphi_j(e_1) \sum_{\gamma} \prod_{s=0}^{t} W_{\gamma_s\gamma_{s+1}} \varphi_i(\gamma_{t+1}),
\end{align*}
where the sum ranges over all non-backtracking paths \( \gamma = (\gamma_0, \dots, \gamma_{t+1}) \) such that \( (\gamma_0, \gamma_1) = e \). It follows that
\begin{align*}
  \left| \langle B^t\chi_i, D_W\check\chi_j \rangle - \sum_{v\in V}f(G, v) \right| & \leq \left| \sum_{e: e_1 \notin \cV_{t+1}} [B^t\chi_i](e)[D_W\check\chi_j](e) \right| \\
                                                                                   & \leq 2|\cV_{t+1}|\max_v{\psi(G, v)},
\end{align*}
using the tangle-free property as before. This time, the results from \autoref{subsec:neighbourhoods} yield
\[ \max_v \psi(G, v) \leq \frac{c_1 \, b^2 \log{(n)}^2 d^{t+1} L^{t+1}}n \quand \norm{\psi}_\star \leq \frac{c_1\,b^2\log{(n)}^3 d^{t+1}L^{t+1}}{n}, \]
and the expected value on the tree is
\begin{align*}
  \sum_{x\in [n]}\E{f(T_x, x)} & = \sum_{x} \varphi_j(x) \mu_i^{t+1} \varphi_i(x) = \mu_i^{t+1}\delta_{ij}.
\end{align*}
Concluding,
\begin{align*}
  \left| \langle B^t\chi_i, D_W\check\chi_j \rangle - \mu_i^{t+1}\delta_{ij} \right| & \leq \frac{c_2\, b^2 \log{(n)}^{4} d^{2t+3} L^{t+1}}n + \frac{c_3\, r b^2 \log{(n)}^6 d^{2t+3}L^{t+1}}{\sqrt{n}} \\
                                                                                     & \leq \frac{c_4\, r b^2 d^3 L \log{(n)}^6 d^{2t} L^t}{\sqrt{n}},
\end{align*}
which proves~\eqref{eq:u_v_dotp}.
\bigskip

Now, let
\[ f(g, o) = \ind_{{(g, o)}_{t} \text{ has no cycles}}\,\vec\partial_\ind[f_{\varphi_i, t} \cdot f_{\varphi_j, t}](g, o). \]
Again, \( f \) is \( t \)-local, and we have
\begin{align*}
  \left|f(g, o)\right| & \leq \norm*{\varphi_i}_\infty\norm*{\varphi_j}_\infty \deg(o) \left| \partial {(g, o)}_t \right|^2 L^{2t} \\
                       & \leq \frac{b^2}n\deg(o)\left| {(g, o)}_t \right|^2 L^{2t} := \psi(g, o)
\end{align*}
By definition of the \( \vec\partial \) operator, we have, for \( v \in V \),
\[ f(g, v) = \sum_{e:e_2 = v} [B^t\chi_i](e)[B^t\chi_j](e). \]

Hence,
\begin{align*}
  \left| \langle B^t\chi_i, B^t\chi_j \rangle - \sum_{v \in V}f(G, v) \right| & = \left| \sum_{e: e_2 \notin \cV_t} [B^t\chi_i](e)[B^t\chi_j](e) \right| \\
                                                                              & \leq 2|\cV_t| \max_v\psi(G, v),
\end{align*}
using the tangle-free property as before. This time, the results from \autoref{subsec:neighbourhoods} yield
\[ \max_v \psi(G, v) \leq \frac{c_5\,b^2\log{(n)}^3 d^{2t+1}L^{2t}}n \quand \norm*{\psi}_\star \leq \frac{c_5\,b^2\log{(n)}^4 d^{2t+1}L^{2t}}n , \]
and we can compute the expected value on the tree:
\begin{align*}
  \sum_{x \in [n]} \E*{f(T_x, x)} & = \sum_{x}{[P\ind](x) \mu_i^{t}\mu_j^t\sum_{s = 0}^t\frac{[K^s\varphi^{i, j}](x)}{{(\mu_i\mu_j)}^s}} \\
                                  & = {(\mu_i \mu_j)}^t\sum_{s = 0}^t\frac{\langle P\ind, K^s \varphi^{i, j} \rangle}{{(\mu_i \mu_j)}^s} \\
                                  & = {(\mu_i \mu_j)}^t \Gamma_{U, ij}^{(t)}.
\end{align*}

Gathering those estimates, we find
\begin{align*}
  \left| \langle B^t \chi_i, B^t \chi_j \rangle -  {(\mu_i \mu_j)}^t \Gamma_{U, ij}^{(t)} \right| & \leq \frac{c_6\, b^2 \log{(n)}^5 d^{3t+2}L^{2t}} n + \frac{c_7\,rb^2 \log{(n)}^7 d^{3t+2}L^{2t}}{\sqrt{n}} \\
                                                                                                  & \leq \frac{c_8\,r b^2 d^2 \log{(n)}^7 d^{3t}L^{2t}}{\sqrt{n}},
\end{align*}
which proves~\eqref{eq:u_u_dotp}.

\bigskip

Next is~\eqref{eq: v_v_dotp}; we first notice that the parity-time equation~\eqref{eq:parity-time} implies that
\[ \langle {(B^*)}^t D_W \check\chi_i, {(B^*)}^t D_W \check\chi_j \rangle = \langle D_W B^t \chi_i, D_W B^t \chi_j \rangle. \]
Similarly to the previous computation, we therefore let \( w_o = (W_{1o}^2, \dots, W_{no}^2) \), and
\[ f(g, o) = \ind_{{(g, o)}_{t} \text{ has no cycles}}\vec \partial_{w_o}[f_{\varphi_i, t} f_{\varphi_j, t}](g, o). \]
We have similarly
\[ |f(g, o)| \leq \frac{b^2}n |{(g, o)}_{t}|^2 L^{2t+2} := \psi(g, o),\]
Now,
\begin{align*}
  \max_v \psi(G, v) & \leq \frac{c_9\,b^2\log{(n)}^2 d^{2t}L^{2t+2}}n, \\ \norm*{\psi}_\star &\leq \frac{c_{10}\,b^2\log{(n)}^3 d^{2t}L^{2t+2}}n,
\end{align*}
and as above
\begin{align*}
  \sum_{x \in [n]} \E*{f(T_x, x)} & = \sum_{x}{[Pw_x](x) \mu_i^{t}\mu_j^t\sum_{s = 0}^t\frac{[K^s\varphi^{i, j}](x)}{{(\mu_i\mu_j)}^s}}  \\
                                  & = \sum_{x}{[K\ind](x) \mu_i^{t}\mu_j^t\sum_{s = 0}^t\frac{[K^s\varphi^{i, j}](x)}{{(\mu_i\mu_j)}^s}} \\
                                  & = \Gamma_{V, ij}^{(t)}.
\end{align*}
Equation~\eqref{eq: v_v_dotp} is then derived as we did earlier.

\bigskip

Our final inequality to prove is~\eqref{eq:u_eigen_equation}; we consider now the function
\[  F_t(g, o) = \ind_{{(g, o)}_{t+1} \text{ has no cycles}}\,\vec\partial_\ind[F_{i, t}^2](g, o)  \]

For all \( t \geq 0 \), the function \( F_t \) is \( t+1 \)-local, and
\begin{align*}
  \left|F_t(g, o) \right| & \leq \deg(o) {\left(2 \norm*{\varphi_i}_\infty \left| {(g, o)}_{t+1} \right|\right)}^2 L^{2t} \\
                          & \leq 4 \deg(o) \frac{b^2}n \left| {(g, o)}_{t+1} \right|^2 L^{2t} := \psi_t(g, o).
\end{align*}

Whenever \( v \notin \cV_t \),
\[  F_t(G, v) = \sum_{e:e_2 = v}{\left([B^{t+1}\chi_i](v) - \mu_i [B^t \chi_i](v)\right)}^2  \]

The same computations as in the other equations then imply that
\[ \left| \norm*{B^{t+1}\chi_i - \mu_i B^t \chi_i}^2 - \sum_{x \in [x]}{F_{t}(G, v)} \right| \leq 2|\cV_t| \max_v\psi(G, v), \]
and
\[ \max_v \psi(G, v) \leq \frac{c_5 \,b^2 \log{(n)}^3 d^{2t+1} L^{2t}}n \quand \norm{\psi}_\star \leq \frac{c_5 \,b^2 \log{(n)}^4 d^{2t+1} L^{2t}}n.\]

Furthermore,
\[ \sum_{x\in [n]} \E*{F_t(T_x, x)} = \sum_{x \in [n]}[P\ind](x)[K^{t+1}\varphi^{i, i}](x) = \, \langle P\ind, K^{t+1} \varphi^{i, i} \rangle, \]
and we can apply Lemma~\ref{lem:scalar_Kt} to find
\[ \left|\sum_{x\in [n]} \E*{F_t(T_x, x)}\right| \leq r d^3 L^2 \rho^{t+1}. \]
Concluding as above,
\begin{align*}
  \norm*{B^{t+1}\chi_i - \mu_i B^t \chi_i}^2 & \leq r d^3 L^2 \rho^{t+1} + \frac{c_{11}\, b^2\log{(n)}^5 d^{3t+3}L^{2t}}n              \\&\qquad\quad\ \ \qquad+ \frac{c_{12}\,r b^2 \log{(n)}^7 d^{3t+3}L^{2t}}{\sqrt{n}} \\
                                             & \leq r d^3 L^2 \rho^{t+1} + \frac{c_{13} r b^2 d^3 \log{(n)}^7 d^{3t}L^{2t}}{\sqrt{n}}.
\end{align*}

\section{Proofs for Theorem~\ref{th:bl_u_bounds}}

\subsection{Proof of~\eqref{eq:Ustar_U}-\eqref{eq:Ustar_V}}

We shall make use of the following classical bound: for a \( r_0 \times r_0 \) matrix \( M \), we have
\begin{equation}\label{eq:infinity_norm_equiv}
  \norm{M}  \leq r_0\norm{M}_\infty.
\end{equation}
First, the \( (i, j) \) entry of matrix \( U^*U \) is \( \langle u_i, u_j \rangle \), and using~\eqref{eq:u_u_dotp} we find
\[ |\langle u_i, u_j \rangle - \Gamma^{(\ell)}_{U, ij} | \leq \frac{c\,r b^2 d^2 \log{(n)}^7 d^{3\ell}L^{2\ell}}{{(\mu_i\mu_j)}^\ell\sqrt{n}}. \]
Since \( i, j \leq r_0 \), we have \( \mu_i\mu_j \geq L^2 \), thus
\[ |\langle u_i, u_j \rangle - \Gamma^{(\ell)}_{ij} | \leq \frac{c\,r b^2 d^2 \log{(n)}^7 d^{3\ell}}{\sqrt{n}}. \]
By definition of \( \ell \), it is easy to check that $d^{3\ell} \leq n^{-1/4}$. Via~\eqref{eq:infinity_norm_equiv}, this implies that \( \norm{U^*U - \Gamma_U^{(\ell)}} \) is less than \( Cn^{-1/4} \), the desired result. The derivation of~\eqref{eq:Vstar_V} is identical, the bound from Proposition~\ref{prop:pseudo_eigenvectors} being essentially the same for both cases.

\bigskip

We now move onto the proof of~\eqref{eq:Ustar_V}; we write the scalar product \( \langle B^\ell \chi_i, {(B^*)}^\ell D_W \check\chi_j \rangle \) as \( \langle B^{2\ell} \chi_i, D_W\check\chi_j \rangle \) and use~\eqref{eq:u_v_dotp} to find
\begin{align*}
  \left|\langle u_i, v_j \rangle - \delta_{ij} \right| & \leq \frac{c\, r b^2 d^3 L \log{(n)}^6 d^{4\ell} L^{2\ell}}{\mu_i^{2\ell + 1}\sqrt{n}} \\
                                                       & \leq \frac{c\, r b^2 d^{7/2} L \log{(n)}^6 d^{4\ell}}{\sqrt{n}}.
\end{align*}
The bound we now need is $d^{4\ell} \leq n^{1/4}$,
which is true by choice of \( \ell \), and we conclude as above.

\subsection{Bounding \texorpdfstring{\( \norm{B^\ell} \)}{|| B\textasciicircum{}l ||}: proof of~\eqref{eq:norm_Bl}}

Let \( w \) be any unit vector in \( R^{\vec E} \), and assume that we are in the event described in Proposition~\ref{prop:local_summary}. Then

\begin{align*}
  \norm{B^t w}^2 & = \sum_{e \in \vec E} {\left(\sum_{(e_0, \dots, e_t) \in \cP(e, t)}\prod_{i = 0}^{t-1}W_{e_i e_{i+1}} w(e_t) \right)}^2 \\
                 & \leq L^{2\ell}\sum_{e \in \vec E} |\cP(e, t)| \sum_{(e_0, \dots, e_t) \in \cP(e, t)} w{(e_t)}^2.
\end{align*}
by the Cauchy-Schwarz inequality. Under the good event from Proposition~\ref{prop:local_summary}, we have
\[ |\cP(e, t)| \leq 2|{(G, e)}_t| \leq c_1 \log(n) d^{\ell}. \]
Additionally, note that the factor \( w{(e_t)}^2 \) appears for each path of length \( t \) ending at \( e_t \), or equivalently (reversing edge orientation) for each path in \( \cP(e_t^{-1}, t) \). Hence,
\begin{align*}
  \norm{B^t w}^2 & \leq c_1 \log(n) d^{\ell}L^{2\ell} \sum_{e \in \vec E} w{(e)}^2 |\cP(e^{-1}, t)| \\
                 & \leq c_2 \log{(n)}^2 d^{2\ell}L^{2\ell},
\end{align*}
and the definition of \( \ell \) ensures (generously) that \( d^{2\ell} < \sqrt{n} \).

\subsection{Proof of Lemma~\ref{lem:near_orthogonality}}

Note first that for all \( t \geq 0 \), the parity-time equation~\eqref{eq:parity-time} allows the simplification
\[ \langle {(B^*)}^t D_W\chi_i, w\rangle = \langle B^t\chi_i, D_W J w\rangle. \]
and we have \( \norm{D_W J w} \leq L \).
Further, the assumption \( w\in H^\bot \) implies
\[ \mu_i^{-t}\langle {(B^*)}^t D_W\chi_i, w\rangle = \mu_i^{-t}\langle {(B^*)}^t D_W\chi_i, w\rangle - \mu_i^{-\ell}\langle {(B^*)}^\ell D_W\chi_i, w\rangle; \] combining the two above arguments and using a telescopic sum as in the proof of~\eqref{eq:Bl_U} gives
\begin{align*}
  \left|\mu_i^{-t}\langle {(B^*)}^t D_W\chi_i, w\rangle\right| & = \left| \sum_{s = t}^{\ell - 1} \mu_i^{-s}\langle B^s\chi_i, D_W J w\rangle - \mu_i^{-(s+1)}\langle B^{s+1}\chi_i, D_W J w\rangle \right|  \\
                                                               & \leq \sum_{s = t}^{\ell - 1} \mu_i^{-(s+1)} \left| \langle B^{s+1}\chi_i, D_W J w\rangle - \mu_i\langle B^{s}\chi_i, D_W J w\rangle \right| \\
                                                               & \leq L \sum_{s = t}^{\ell - 1} \mu_i^{-(s+1)} \norm{B^{s+1}\chi_i - \mu_i B^s\chi_i},
\end{align*}
where we used the Cauchy-Schwarz inequality at the last line. Now, we can apply equation~\eqref{eq:u_eigen_equation}:
\[\norm{B^{s+1}\chi_i - \mu_i B^s \chi_i}^2 \leq r d^3 L^2 \rho^{s+1} + \frac{c r b^2 d^3 \log{(n)}^7 d^{3s}L^{2s}}{\sqrt{n}}, \]
and still following the proof of~\eqref{eq:Bl_U} we find
\[\norm{B^{s+1}\chi_i - \mu_i B^s \chi_i} \leq \sqrt{r}d^{3/2}L\rho^{\frac{s+1}2} + \frac{c_2 b d^{3/2} \log{(n)}^{7/2} d^{3s/2}L^{s}}{n^{1/4}}. \]
Summing these inequalities (and using \( \ell \leq c_3\log(n) \)) yields
\[ \left|\langle {(B^*)}^t D_W \check\chi_i, w\rangle \right| \leq \sqrt{r}d^{3/2} L^2\mu_i^t \sum_{s=t}^{\ell - 1}{\left(  \frac{\sqrt{\rho}}{\mu_i} \right)}^{s+1} + \mu_i^{-\ell} \frac{c_4 b d^{3/2} \log{(n)}^{9/2} d^{3\ell/2}L^{\ell}}{n^{1/4}}.\]
Since \( i \leq r_0 \), we have \( \mu_i >\sqrt{\rho} \). As a result, all terms in the sum are bounded by the one for \( s = t-1 \), and \( \mu_i^{-\ell} \leq d^{\ell/2} \). We finally get
\[ \left|\langle {(B^*)}^t D_W \check\chi_i, w\rangle \right| \leq \sqrt{r}d^{3/2}L^2\rho^{t/2} + \frac{c_4\, b d^{3/2} \log{(n)}^{9/2} d^{2\ell}L^{\ell}}{n^{1/4}},\]
as desired.

\section{Norm bounds: additional proofs}

\subsection{Bound~\eqref{eq:trace_bound_mb} on \texorpdfstring{\( \norm{MB^k} \)}{|| MB\textasciicircum{}k ||}}

Since \( \norm{M} \) is of order \( 1 \), we notice that~\eqref{eq:trace_bound_mb} improves by a factor of \( \sqrt{n} \) on the crude bound \( \norm{KB^k} \leq\norm{K}\norm{B^k} \). We use the same trace method as above; we have
\begin{align*}
  \norm{M B^{k-2}}^{2m} & \leq \tr\left[{\left(M B^{k-2} {(B^*)}^{k-2}M^*\right)}^m\right]                                                                                                            \\
                        & \leq {\left(\frac dn \right)}^{2m} \sum_{\gamma\in W_{k, m}} \prod_{i = 1}^{m} X_{\gamma_{2i-1, 1}\gamma_{2i-1, 2}} \prod_{s = 3}^k  A_{\gamma_{2i-1, s-1}\gamma_{2i-1, s}} \\ &\qquad \qquad\ \ \qquad \times\prod_{s = 1}^{k-2}  A_{\gamma_{2i, s-1}\gamma_{2i, s}} X_{\gamma_{2i, k-2}\gamma_{2i, k-1}}
\end{align*}
where \( W_{k, m} \) is the set of sequences of paths defined just below equation~\eqref{eq:first_trace_bound}. The set of edges of the form \( (\gamma_{2i-1, 0}, \gamma_{2i-1, 1}) \) or \( (\gamma_{2i, k-1}, \gamma_{2i, k}) \), which support no random variable, has cardinality at most \( m \) by the boundary conditions, hence the bound for any \( \gamma \in W_{k, m} \):
\begin{multline*}
  \prod_{i = 1}^{m} X_{\gamma_{2i-1, 1}\gamma_{2i-1, 2}} \prod_{s = 3}^k  A_{\gamma_{2i-1, s-1}\gamma_{2i-1, s}}\prod_{s = 1}^{k-2}  A_{\gamma_{2i, s-1}\gamma_{2i, s}} X_{\gamma_{2i, k-2}\gamma_{2i, k-1}} \\\leq {\left(\frac d n\right)}^{e_\gamma - m} L^{2(k-2)m}.
\end{multline*}
Using bound~\eqref{eq:tangle_free_path_counting} on \( \cW_{k, m}(v, e) \) and the fact that each equivalence class contains at most \( n^v \) elements, we get
\begin{align}
  \norm{M B^{k-2}}^{2m} & \leq {\left( \frac d n \right)}^{m} \sum_{e = 1}^{2km} \sum_{v=1}^{e+1} k^{2m}{(2km)}^{6m(e-v+1)} n^v {\left(\frac d n\right)}^e L^{2(k-2)m} \nonumber \\
                        & \leq n^{-m} d^{5m} L^{2km} k^{2m} \sum_{e = 1}^{2km}\sum_{v=1}^{e+1}{(2km)}^{6m(e-v+1)} d^e n^{v-e} \nonumber                                          \\
                        & \leq n^{-m+1} d^{5m} L^{2km} k^{2m} (2km) d^{2km} \sum_{g =0}^{\infty} {\left( \frac{{(2km)}^{6m}}{n} \right)}^g \label{eq:final_bound_mb}.
\end{align}
The choice of parameter
\[ m = \left\lceil \frac{\log(n)}{12\log(\log(n))} \right \rceil \]
ensures that the infinite sum in~\eqref{eq:final_bound_mb} converges for \( n \) larger than an absolute constant, which yields~\eqref{eq:trace_bound_mb}.

\subsection{Bound~\eqref{eq:delta_m_b} on \texorpdfstring{\( \norm{\Delta^{(t-1)}\tilde MB^{k - t - 1}} \)}{|| \textDelta\textasciicircum(t-1)\~M B\textasciicircum(k-t-1) ||}}

First, notice that \( M^{(2)}_{ef} \) is equal to \( {(TQT^*)}_{ef} \) except when \( \ind \{e \xrightarrow{2} f \} = 0 \), which happens only when \( e = f \), \( e \rightarrow f \), \( e \rightarrow f^{-1} \) of \( f^{-1} \rightarrow e \). Therefore, we can write
\[ |L_{ef}| \leq \frac d n (\tilde M_1 + \tilde M_2 + \tilde M_3 + \tilde M_4),  \]
where each entry of the matrix \( M_i \) is one whenever the \( i \)-th condition mentioned above is true. Then, for each \( i \) we can write
\[ \norm{\Delta^{(t-1)}\tilde M_i B^{k - t - 1}} \leq \norm{\Delta^{(t-1)}} \norm{\tilde M_i B^{k - t - 1}},\]
and a straightforward adaptation of the proof of bound~\eqref{eq:trace_bound_mb} gives
\[ \frac d n \norm{\tilde M_i B^{k - t - 1}} \leq \frac{c d^{7/2} L \ln{(n)}^{7}d^{k-t}L^{k-t}}{\sqrt{n}}. \]
Combining the above bound with~\eqref{eq:trace_bound_delta} easily implies~\eqref{eq:delta_m_b}.

\subsection{Bound~\eqref{eq:trace_bound_r} on \texorpdfstring{\( R_t^{(\ell)} \)}{|| R\_t\textasciicircum(l) ||}}

The proof of~\eqref{eq:trace_bound_r} is very similar to those above, as well as the one in~\cite{bordenave_nonbacktracking_2018}; we only highlight the main differences. Let \( t \geq 1 \) (the case \( t=0 \) is almost identical), and \( k \leq \log(n) \). The same trace argument gives
\begin{align*}
  \norm{R_t^{(k - 1)}}^{2m} & \leq \tr\left[{\left( R_t^{(k - 1)}R_t^{{(k - 1)}^*}\right)}^m\right]                                                                                                                                                                         \\
                            & = \sum_{\gamma \in T_{k, m, t}} \prod_{i = 1}^{2m}\underline X_{\gamma_{i, 0}\gamma_{i, 1}}\prod_{s = 2}^t \underline A_{\gamma_{i, s-1}\gamma_{i, s}}Q_{\gamma_{i, t}, \gamma_{i, t+1}}\prod_{s = t+2}^{k} A_{\gamma_{i, s-1}\gamma_{i, s}},
\end{align*}
where \( T_{k, m, t} \) is the set of sequences of paths \( (\gamma_1, \dots, \gamma_{2m}) \) such that for all \( i \), \( \gamma_i^1 = (\gamma_{i,0}, \dots, \gamma_{i, t}) \) and \( \gamma_i^2 = (\gamma_{i,t+1}, \dots, \gamma_{i, k}) \) are tangle-free and \( \gamma_i \) is tangled, with similar boundary conditions as in~\eqref{eq:trace_method_boundary_conditions}.

We define \( G_{\gamma} \) as the union of the \( G_{\gamma_i^z} \) for \( z\in [2m], j\in \{1, 2\} \). Since we remove an edge to each path, \( G_\gamma \) need not be connected; however, since \( \gamma_i \) is tangled, each connected component in \( G_{\gamma_i} \) contains a cycle, and the same holds for \( G \). It follows that
\[ v_\gamma \leq e_\gamma \]
for all \( \gamma \in T_{k, m, t} \). As before, we define the equivalence relation \( \sim \) and \( \cT_{k, m, t}(v, e) \) the set of equivalence classes with \( v_\gamma = v \) and \( e_\gamma = e \). Then, the following lemma from~\cite{bordenave_nonbacktracking_2018} holds:
\begin{lemma}
  Let \( v, e \) be any integers such that \( v \leq e \). Then
  \[ \left| \cT_{k, m, t}(v, e)\right| \leq {(4km)}^{12m(e - v +1) + 8m}.  \]
\end{lemma}

As for bounding the contribution of a single path, the computations already performed in bounding~\eqref{eq:trace_bound_delta} work similarly:
\begin{multline*}
  \prod_{i = 1}^{2m}\underline X_{\gamma_{i, 0}\gamma_{i, 1}}\prod_{s = 2}^t \underline A_{\gamma_{i, s-1}\gamma_{i, s}}Q_{\gamma_{i, t}, \gamma_{i, t+1}}\prod_{s = t+2}^{k} A_{\gamma_{i, s-1}\gamma_{i, s}} \\\leq {\left(\frac a n\right)}^{e_\gamma + 2m} {\left(1 + \frac d n\right)}^{2km} d^{2m} L^{2km},
\end{multline*}
using \( Q_{ij} \leq dL/n \) for all \( i, j \). Finally, for \( [\gamma] \in \cT_{k, m, t}(v, e) \), there are at most \( n^{v} \) sequences \( \gamma' \) such that \( \gamma' \sim \gamma \). This yields
\begin{align*}
  \E*{\norm{R_t^{(k - 1)}}^{2m}} & \leq c_1^{2m}d^{2m}L^{2km} {\left(\frac d n\right)}^{2m} {(4km)}^{20m} \sum_{e = 1}^{2km}{(4km)}^{12m(e - v)}  \sum_{v = 1}^e d^e n^{v - e} \\
                                 & \leq c_2^{2m} d^{4m}L^{2km} n^{-2m} \log{(n)}^{40m} (2km) d^{2km} \sum_{g = 0}^\infty {\left(\frac{{(4km)}^{12m}}n\right)}^g,
\end{align*}
using preemptively the bound \( m \leq \log(n) \) and the change of variables \( g = e - v \).  This time, choosing
\[ m = \left\lceil \frac{\log(n)}{24\log(\log(n))} \right \rceil \]
yields a convergent sum, and~\eqref{eq:trace_bound_r} follows.

\bibliographystyle{plainnat}
\bibliography{bbp}

\end{document}